\documentclass[reqno,a4]{amsart}
\usepackage{comment}
\usepackage{amssymb,latexsym,upref}

\newcommand{\mnote}[1]{}

\newcommand{\eqn}[2]{ \begin{equation*} #2  \end{equation*} }
\newcommand{\gath}[2]{ \begin{gather*} #2 \end{gather*} }
\newcommand{\lalign}[2]{ \begin{align} #2 \end{align} }
\newcommand{\alin}[2]{ \begin{align*} #2 \end{align*} }
\newcommand{\leqn}[2]{\begin{equation} #2 \label{#1} \end{equation}}

\newcommand{\lgath}[2]{ \begin{gather} #2 \end{gather} }

\newcommand{\sect}[2]{\section{#2}\label{#1}}
\newcommand{\subsect}[2]{\subsection{#2}\label{#1}}
\newcommand{\subsubsect}[2]{\subsubsection{#2}\label{#1}}

\theoremstyle{plain}
\newtheorem{thm}{Theorem}[section]
\newtheorem{lem}[thm]{Lemma}
\newtheorem{prop}[thm]{Proposition}
\newtheorem{cor}[thm]{Corollary}

\theoremstyle{definition}
\newtheorem{Def}[thm]{Definition}

\theoremstyle{remark}
\newtheorem{rem}[thm]{Remark}

\numberwithin{equation}{section}

\newcommand{\id}{\mathord{{\mathrm 1}\kern-0.27em{\mathrm I}}\kern0.35em}

\DeclareMathOperator{\diag}{diag}

\newcommand{\del}[1]{\partial_{#1}}

\newcommand{\AND}{{\quad\text{and}\quad}}

\DeclareMathOperator{\op}{op}

\newcommand{\Half}{\ensuremath{\textstyle\frac{1}{2}}}

\newcommand{\Third}{\ensuremath{\textstyle\frac{1}{3}}}

\newcommand{\norm}[1]{\|#1\|}

\newcommand{\ip}[2]{ \langle #1 | #2 \rangle}

\newcommand{\Bc}{\mathcal{B}{}}

\newcommand{\Ec}{\mathcal{E}{}}

\newcommand{\Hc}{\mathcal{H}{}}

\newcommand{\Lc}{\mathcal{L}{}}

\newcommand{\Nc}{\mathcal{N}{}}

\newcommand{\Rc}{\mathcal{R}{}}

\newcommand{\Xc}{\mathcal{X}{}}

\newcommand{\gt}{\tilde{g}{}}
\newcommand{\htld}{\tilde{h}{}}

\newcommand{\kt}{\tilde{k}{}}

\newcommand{\st}{\tilde{s}{}}
\newcommand{\ut}{\tilde{u}{}}
\newcommand{\Ut}{\tilde{U}{}}

\newcommand{\alphat}{\tilde{\alpha}{}}
\newcommand{\betat}{\tilde{\beta}{}}

\newcommand{\phit}{\tilde{\phi}{}}

\newcommand{\Ab}{\bar{A}{}}

\newcommand{\fb}{\bar{f}{}}
\newcommand{\Fb}{\bar{F}{}}
\newcommand{\gb}{\bar{g}{}}

\newcommand{\hb}{\bar{h}{}}
\newcommand{\Hb}{\bar{H}{}}

\newcommand{\Ub}{\bar{U}{}}

\newcommand{\Xb}{\bar{X}{}}

\newcommand{\alphab}{\bar{\alpha}{}}
\newcommand{\betab}{\bar{\beta}{}}

\newcommand{\gammab}{\bar{\gamma}{}}

\newcommand{\Uh}{\hat{U}{}}

\newcommand{\Gbb}{\mathbb{G}{}}

\newcommand{\Rbb}{\mathbb{R}{}}

\newcommand{\Tbb}{\mathbb{T}{}}
\newcommand{\Zbb}{\mathbb{Z}{}}

\newcommand{\Jch}{\check{J}{}}

\newcommand{\fv}{\mathbf{f}{}}

\newcommand{\hv}{\mathbf{h}{}}

\newcommand{\Kv}{\mathbf{K}{}}
\newcommand{\qv}{\mathbf{q}{}}
\newcommand{\uv}{\mathbf{u}{}}

\newcommand{\xv}{\mathbf{x}{}}

\newcommand{\zero}{\mathbf{0}}

\newcommand{\muv}{\boldsymbol{\mu}{}}

\newcommand{\ep}{\epsilon}
\newcommand{\Om}{ {\Omega_1}}

\begin{document}

\title[A transmission problem for quasi-linear wave equations]{A transmission problem for quasi-linear wave equations}

\author[L. Andersson]{Lars Andersson}
\address{Albert Einstein Institute, Am M\"uhlenberg 1, D-14476 Potsdam,
  Germany}
\email{lars.andersson@aei.mpg.de}

\author[T.A. Oliynyk]{Todd A. Oliynyk}
\address{School of Mathematical Sciences\\
Monash University, VIC 3800\\
Australia}
\email{todd.oliynyk@sci.monash.edu.au}



\begin{abstract}
We prove the local existence and uniqueness of solutions to a system of quasi-linear wave equations
involving a jump discontinuity in the lower order terms. A continuation principle is also established.
\end{abstract}

\maketitle

\sect{intro}{introduction}

Most of the visible matter in our Universe is composed of gravitating relativistic elastic matter; for example, asteroids, comets, planets and stars, including neutron stars,
are all thought of as being accurately described as elastic bodies \cite{AnderssonComer:2007,ChamelHaensel:2008}.
Due to this, it is of clear theoretical and even practical interest to have
a good analytic understanding of gravitating relativistic elastic bodies with the first step being to establish local existence and
uniqueness results.

In the non-relativistic setting of Newtonian gravity,
local existence and uniqueness theorems are available. In the approximation of a compact (non-fluid) elastic body moving in an external gravitational field, where the
gravitational self-interaction and interaction with the object
generating the external field are ignored,
local existence and uniqueness has been established
in \cite{ShibataNakamura:1989}.  Local existence and uniqueness results for the general case, which includes
gravitational self and mutual interactions between adjacent (non-fluid) elastic bodies, are given in \cite{Andersson_et_al:2011}.
See also \cite{LindbladNordgren:2009} for related results on self-gravitating, incompressible fluid bodies.  In contrast, much less is known in the relativistic setting where local existence
and uniqueness theorems are lacking except in certain restricted situations \cite{ChoquetFriedrich:2006,KindEhlers:1993,Rendall:1992}.

Relativistic compact elastic bodies are governed by the Einstein field equations coupled through the stress-energy
tensor to the field equations of relativistic elasticity. The difficulty in establishing local existence and uniqueness results can be attributed to
two sources:
the free boundary arising from the evolving matter-vacuum interface, and the irregularity in the stress-energy
tensor across the matter-vacuum interface. For elastic bodies, there are essentially two distinct types of irregularities.
The first type corresponds to gaseous fluid
bodies where the proper energy density monotonically decreases in a neighborhood of the vacuum boundary and vanishes identically
there. In this situation, the fluid evolution equations become degenerate and are no longer hyperbolic at the boundary leading to
severe analytic difficulties. The second type of irregularity that occurs for elastic bodies is where the proper energy density has
a finite (positive) limit at the vacuum boundary. Examples of this type are liquid fluids and solid elastic bodies. This case leads to a
jump discontinuity in the stress energy tensor across the vacuum boundary.

Our main motivation for this article is to develop local existence and uniqueness results that are applicable
to the gravitational part of the initial value problem (IVP) for gravitating relativistic elastic bodies that are not fluids\footnote{We recall that relativistic fluids are a special case of relativistic elastic matter.} and have the second type of discontinuity.
For such elastic bodies, it is well known from \cite{BeigSchmidt:2003,BeigWernig:2007}, see also Section \ref{disc}, when harmonic coordinates are employed and the material representation is employed, that the gravitational
component of the field equations consists of a system of non-linear wave equations with
a jump discontinuity at the matter-vacuum boundary while the elastic component consists of a non-linear system of wave equations
with Neumann boundary conditions. This leads us to consider IVPs of the form\footnote{The notation used in this article for coordinates, indices, partial derivatives,
function spaces, and the
like can be found in Section \ref{prelim}.}
\lalign{waveA}{
\del{\mu}\bigl(A^{\mu \nu}(U) \del{\nu} U\bigr)  &= F(U,\del{}U)
+ \chi_{\Omega} H(U, \del{} U) \quad \text{in $[0,T]\times \Rbb^n$}, \label{waveA.1} \\
(U,\del{t}U)|_{t=0} & = (\Ut_0,\Ut_1) \quad \text{in $\Rbb^n$} \label{waveA.2},
}
where
\begin{enumerate}
\item[(i)] $\Omega$ is a bounded open set in $\Rbb^n$ with smooth boundary,
\item[(ii)] $U(\xv) = (U^1(\xv),\ldots,U^{N}(\xv))$ is vector valued,
\item[(iii)] $A(U)=(A^{\mu\nu}(U))$, $F=(F^I(U,\del{}U))$ and $H=(H^I(U,\del{}U))$  (I=1,\ldots,N) are smooth maps with $F(0,0)=0$, and
\item[(iv)] for some $\gamma,\kappa > 0$, $A^{\mu\nu}(U)$ satisfies
\leqn{waveAa}{
\frac{1}{\gamma} |\xi|^2 \leq  A^{ij}(U)\xi_i \xi_j \leq \gamma |\xi|^2  \AND
A^{00}(U) \leq -\kappa
}
for all $(U,\xi) \in \Rbb^N\times \Rbb^n$.
\end{enumerate}

In Section \ref{disc}, we describe how the results of this article can be used in conjunction with the local existence theory from \cite{Koch:1993} to establish the local existence and uniqueness of solutions
that represent gravitating relativistic compact elastic bodies. The complete local existence and uniqueness proof will be provided in a separate article \cite{Andersson_et_al:2013}.
Aside from this application, we believe that the results of this
paper are of independent interest and may be useful for other initial value problems involving systems of wave equations with
lower order coefficients that have a jump discontinuity across a fixed boundary.

Due to the discontinuity in the wave equation \eqref{waveA.1} arising from the
term $\chi_{\Omega} H(U, \del{} U)$, the initial value problem (IVP) \eqref{waveA.1}-\eqref{waveA.2} is a \emph{transmission problem}, that is, a problem where
we can view that total solution as comprised of an interior solution and an exterior solution that are appropriately ``matched'' across the dividing interface $\del{}\Omega$.
Due to the jump discontinuity across $\del{}\Omega$, standard $L^2$ Sobolev spaces $H^s(\Rbb^n)$ do not
provide a suitable setting for establishing the local existence and uniqueness of solutions, and we use instead
the intersection spaces $\Hc^{k,s}(\Rbb^n) = H^s(\Omega)\cap H^s(\Rbb^n\setminus \Omega)\cap H^k(\Rbb^n)$. Similar to the situation that arises for initial boundary value problems,
we also find it necessary to choose initial data
\leqn{compat1}{
(\Ut_0,\Ut_1) \in \Hc^{2,s+1}(\Rbb^n)\times \Hc^{2,s}(\Rbb^n) \quad s\in \Zbb_{>n/2}
}
that satisfy \emph{compatibility conditions} given by
\leqn{compat2}{
\Ut_\ell := \del{t}^\ell U |_{t=0} \in \Hc^{m_{s+1-\ell},s+1-\ell}(\Rbb^n) \quad \ell=2,\ldots,s.
}
Here the time derivatives $\del{t}^\ell U |_{t=0}$ $\ell \geq 2$ are generated
from the initial data \eqref{compat1} by formally differentiating \eqref{waveA.1} with respect to
$t$ at $t=0$. To see how this works, we note that $\del{t}^2 U|_{t=0}$ can be computed by substituting
the initial data \eqref{compat1} in \eqref{waveA.1} and then solving
for $\del{t}^2 U|_{t=0}$. Differentiating \eqref{waveA.1} formally with respect to $t$ while substituting in
the lower time derivatives $\ell=0,1,2$ at $t=0$ then uniquely determines the $\ell=3$ time derivative at $t=0$ in terms of the initial data.
Continuing on by formally differentiating the evolution equations with respect to $t$, it is not difficult to see that
the higher time derivatives $\ell=2,\ldots,s$ at $t=0$ are uniquely determined in terms of the initial data.

We are now ready to state the main local existence and uniqueness result.
\begin{thm} \label{mainthmA}
Suppose $n\geq 3$, $s\in \Zbb_{>n/2}$ and $(\Ut_0,\Ut_1) \in \Hc^{s+1}(\Rbb^n)\times \Hc^{s}(\Rbb^n)$ satisfy the
compatibility conditions \eqref{compat2}. Then there exist a $T>0$ and a map
$U\in CX_{T}^{s+1}(\Rbb^n)$ such that $U$ is the unique solution in $CX_{T}^2(\Rbb^n) \cap \bigcap_{\ell=0}^1 C^\ell\bigl([0,T),W^{1-\ell,\infty}(\Rbb^n)\bigr)$
to the initial value problem \eqref{waveA.1}-\eqref{waveA.2}. Moreover, if $\norm{u}_{W^{1,\infty}((0,T)\times \Tbb^n)} < \infty$, then there exists
a $T^*>T$ such that the $u$ can be continued uniquely to a solution on $[0,T^*)\times \Tbb^n$.
\end{thm}
The proof of this theorem can be found in Section \ref{exist} and relies on a strategy similar to the one employed by Koch in \cite{Koch:1993}
to
establish the existence and uniqueness of solutions to fully non-linear wave equations on bounded domains with Neumann or Dirichlet boundary
conditions. Koch's method involves differentiating the evolution equation $s$ times with respect to $t$ for $s$ sufficiently large. He then views the equations involving the lower order time derivatives $\del{t}^\ell u$
$(\ell=0,\ldots,s-1)$ as a system of coupled elliptic equations for the purpose of obtaining estimates and estimates the top time derivative $\del{t}^s u$ using hyperbolic energy estimates. This allows him to avoid directly differentiating in
directions normal to the boundary. For us, this strategy allows us to avoid differentiating the term $\chi_\Omega H(U,\del{} U)$ across $\del{}\Omega$ where it is discontinuous.

Although the arguments used in this article are structurally similar to those employed in Koch, there are some differences. One difference is that the
elliptic equations that arise in this article are not of a standard type due to
the presence of the discontinuous term $\chi_\Omega H(U,\del{} U)$. As a consequence, we cannot, as did Koch, appeal
to standard elliptic estimates, and instead we employ potential theory to derive the desired estimates. Another
distinction is that we are not able to obtain estimates for all of the derivatives by differentiating tangentially
to the space-time boundary $[0,T]\times\del{}\Omega$ and then using the evolution equations to recover the missing estimate for the derivative normal
to the boundary as was done by Koch in \cite{Koch:1993}. One immediate consequence of this is that we cannot employ Koch's strategy
to derive a continuation principle and instead must argue differently.

\begin{rem} \label{waveArem}
$\;$

\begin{itemize}
\item[(i)]
The assumptions on the IVP \eqref{waveA.1}-\eqref{waveA.2} can easily be relaxed so that
\begin{itemize}
\item[(a)]$A$, $F$ and $H$ depend explicitly on $\xv \in \Rbb^{n+1}$, i.e. $A=A(\xv,U)$, $F=F(\xv,U,\del{}U)$, and
$H=H(\xv,U,\del{}U)$, and are defined for $(\xv,U,\del{} U) \in \Rbb^{n+1}\times \mathcal{U}\times \mathcal{V}$
with $\mathcal{U}$ and $\mathcal{V}$ open in $\Rbb^N$ and $\Rbb^{(n+1)\times N}$, respectively,
\item[(b)] $A$ and $\{F,H\}$ are $s+1$ and $s$ times continuously differentiable in all variables, respectively,
where $s\in \Zbb_{>n/2}$, and
\item[(c)] the inequality \eqref{waveAa} holds for $(\xv,U,\xi) \in \Rbb^{n+1}\times\mathcal{U}\times \Rbb^n$.
\end{itemize}
\item[(ii)] The following generalizations of Theorem \ref{mainthmA} also hold.
\begin{itemize}
\item[(a)] The continuous dependence of the solutions from Theorem \ref{mainthmA} on the initial data satisfying the compatibility conditions \eqref{compat2} can be established using similar arguments as in \cite{Koch:1993}.
\item[(b)] Theorem \ref{mainthmA} is also valid for quasi-linear wave equations
\eqn{genA}{
A^{\mu \nu}(U,\del{}U) \del{\mu}\del{\nu} U  = F(U,\del{}U)
+ \chi_{\Omega} H(U, \del{} U) \quad \text{in $[0,T]\times \Rbb^n$},
}
provided that we take $s>n/2+1$ and change the continuation principle to that of bounding
$\norm{U}_{W^{2,\infty}((0,T)\times \Rbb^n)}$.
\end{itemize}
\end{itemize}
\end{rem}
\sect{prelim}{Preliminaries}

\subsect{notation}{Notation} In this article, we use $(x^{\mu})_{\mu=0}^n$  to denote Cartesian coordinates
on $\Rbb^{n+1}$, and we use  $x^0$ and $t$, interchangeability, to denote the time coordinate,
and $(x^i)_{i=1}^n$ to denote the spatial coordinates. We also use $x=(x^1,\ldots,x^n)$
and $\xv = (x^0,\ldots,x^n)$ to denote spatial and spacetime points, respectively.

Partial derivatives are denoted by
\eqn{partial}{
\del{\mu} = \frac{\partial \;}{\partial x^\mu},
}
and
we use $Du(x) = (\del{1}u(x),\ldots,\del{n}u(x))$ and
$\del{}u(\xv) =  (\del{0}u(\xv),Du(\xv))$ to denote the spatial and spacetime gradients, respectively.
For time derivatives, we often employ the notation
\eqn{ft}{
u_r := \del{t}^r u,
}
and use
\eqn{fvect}{
\uv_r = (u_1, u_2, \ldots, u_r )^{\text{Tr}}
}
to denote the collection of partial derivatives of $u$ with respect to $t$.

\subsect{sets}{Sets}
The following subsets of $\Rbb^n$ will be of interest:
\alin{setsdef}{
Q^-_\delta &= \{\: (x^1,\ldots,x^n) \: |\: -\delta < x^1,x^2,\ldots,x^{n-1} < \delta, \quad -\delta < x^n < 0 \: \}, \\
Q^+_\delta &= \{\: (x^1,\ldots,x^n) \: |\: -\delta < x^1,x^2,\ldots,x^{n-1} < \delta, \quad 0 < x^n < \delta \: \} \\
\intertext{and}
Q_\delta &= \{\: (x^1,\ldots,x^n) \: |\: -\delta \leq  x^1,x^2,\ldots,x^{n} \leq \delta  \:\}.
}
We will also need to identify the opposite sides of the box $Q_\delta$ so that\footnote{Here, $\sim$ denotes the equivalence relation on
$Q_\delta$ determined by the identification of the opposite sides of the boundary.}
\eqn{Tbbdefa}{
 Q_\delta/\sim  \approx\Tbb^n .
}
We note that  under this identification, the Carestian coordinates $x=(x^i)$ on $\Rbb^n$  define periodic coordinates on
$\Tbb^n$.
The following open and connected subset of $\Tbb^n$ with smooth boundary will also be of interest
\eqn{Omegadef}{
\Omega_\delta  =  Q^+_{\delta}/\sim.
}

Finally, given an open set $\Omega$ of $\Gbb^n$, where
\eqn{Gbbdef}{
\text{$\Gbb^n$ = $\Tbb^n$ or $\Rbb^n$,}
}
we let $\chi_\Omega$ denote the characteristic function, and we use $\Omega^c$ to
denote the interior of its complement, that is
\eqn{Omeagcdef}{
\Omega^c := \Gbb^n \setminus \overline{\Omega}.
}

\subsect{funct}{Function spaces}

\subsubsect{sfunct}{Spatial function spaces}

Given an open set $\Omega \subset \Gbb^n$,
we define the Banach spaces
\eqn{HksdefT}{
\Hc^{k,s}(\Gbb^n) =  H^k(\Gbb^n)\cap  H^s(\Omega)\cap H^s(\Omega^c)  \quad (s\geq k; k,s\in\Zbb_{\geq 0})
}
with norm
\eqn{HksnormT}{
\norm{u}_{\Hc^{k,s}(\Gbb^n)}^2 = \norm{u}_{H^s(\Omega)}^2 + \norm{u}_{ H^k(\Gbb^n)}^2 + \norm{ u}^2_{H^s(\Omega^c)},
}
and
\eqn{HksdefQ}{
\Hc^{k,s}(Q_\delta) =  H^k(Q_\delta)\cap  H^s(Q^+_\delta)\cap H^s(Q^-_\delta)  \quad (s\geq k; k,s\in\Zbb_{\geq 0})
}
with norm
\eqn{HksnormTa}{
\norm{u}_{\Hc^{k,s}(Q_\delta)}^2 = \norm{u}_{H^s(Q^+_\delta)}^2 + \norm{u}_{ H^k(Q_\delta)}^2 + \norm{ u}^2_{H^s(Q^-_\delta)}.
}
We also define the following auxiliary spaces
\gath{Xdef}{
X^{k,r} = \prod_{\ell=0}^{r} \Hc^{2,k-\ell}(\Tbb^n) \quad  (k-r\geq 2), \qquad \Xc^{k,r} = \prod_{\ell=0}^{r} \Hc^{0,k-\ell}(\Tbb^n) \quad (k-r\geq 0) 
\intertext{and}
Y^{k,r} = \prod_{\ell=0}^{r} H^{k-\ell}(\Omega)  \quad (k-r\geq 0,\; \Omega \subset \Tbb^n) 
}
with norms
\gath{Xnorm}{
\norm{\uv_r}_{X^{k,r}}^2 = \sum_{\ell=0}^r \norm{u_{\ell}}^2_{\Hc^{2,k-\ell}(\Tbb^n)},\qquad \norm{\uv_r}_{\Xc^{k,r}}^2 = \sum_{\ell=0}^r \norm{u_{\ell}}^2_{\Hc^{0,k-\ell}(\Tbb^n)},
\intertext{and}
\norm{\uv_r}_{Y^{k,r}}^2 = \sum_{\ell=0}^r \norm{u_{\ell}}^2_{H^{k-\ell}(\Omega)},
}
respectively,  where, as above, we employ the vector notation
\eqn{fvectA}{
\uv_r = (u_1,\ldots,u_r)^{\text{Tr}}.
}

\subsubsect{stfunct}{Spacetime function spaces}

Given an open subset $\Omega \subset \Gbb^n$ and a $T>0$, we define the spaces
\leqn{XTdef}{
X^s_T(\Gbb^n) = \bigcap_{\ell=0}^s W^{\ell,\infty}\bigl([0,T),\Hc^{m_{s-\ell},s-\ell}(\Gbb^n)\bigr),
}
where
\eqn{mdef}{
m_\ell = \begin{cases} 2 & \text{if $\ell \geq 2$} \\
\ell & \text{ otherwise } \end{cases},
}
\leqn{XcTdef}{
\Xc^s_T(\Gbb^n) = \bigcap_{\ell=0}^s W^{\ell,\infty}\bigl([0,T),\Hc^{0,s-\ell}(\Gbb^n)\bigr)
}
and
\leqn{YTdef}{
Y^s_T(\Omega) =  \bigcap_{\ell=0}^s W^{\ell,\infty}\bigl([0,T),H^{s-\ell}(\Omega)\bigr).
}

We also define the following \emph{energy norms}:
\alin{XTnormA}{
\norm{u}_{E^s(\Gbb^n)}^2 &= \sum_{\ell=0}^s \norm{\del{t}^\ell u}^2_{\Hc^{m_{s-\ell},s-\ell}(\Gbb^n)},\\
\norm{u}_{\Ec^s(\Gbb^n)}^2 &= \sum_{\ell=0}^s \norm{\del{t}^\ell u}^2_{\Hc^{0,s-\ell}(\Gbb^n)}, \\
\norm{u}^2_{E^s(\Omega)} &=  \sum_{\ell=0}^s \norm{\del{t}^\ell u}^2_{H^{s-\ell}(\Omega)}, \\
\norm{u}_{E^{s,r}(\Gbb^n)}^2 &= \sum_{\ell=0}^r \norm{\del{t}^\ell u}^2_{\Hc^{m_{s-\ell},s-\ell}(\Gbb^n)} \quad (s\geq r)
\intertext{and}
\norm{u}_{E(\Gbb^n)}^2 &= \norm{u}_{E^1(\Gbb^n)}^2 = \norm{u}_{H^1(\Gbb^n)}^2 + \norm{\del{t}u}^2_{L^2(\Gbb^n)}.
}
In terms of these energy norms, we can write the norms of the spaces \eqref{XTdef}, \eqref{XcTdef} and \eqref{YTdef} as
\alin{XTnormB}{
\norm{u}_{X^s_T(\Gbb^n)} & = \sup_{0\leq t < T} \norm{u(t)}_{E^s(\Gbb^n)}, \\
\norm{u}_{\Xc^s_T(\Gbb^n)} & = \sup_{0\leq t < T} \norm{u(t)}_{\Ec^s(\Gbb^n)}
\intertext{and}
\norm{u}_{Y^s_T(\Omega)} & =  \sup_{0\leq t < T} \norm{u(t)}_{E^s(\Omega)},
}
respectively.

Finally, we define the following subspace of \eqref{XTdef}: 
\alin{XVdef}{
CX^s_T(\Gbb^n) &= \bigcap_{\ell=0}^s C^{\ell}\bigl([0,T),\Hc^{m_{s-\ell},s-\ell}(\Gbb^n)\bigr). 
}

\subsect{cost}{Estimates and constants}

We employ that standard notation
\eqn{lesssimA}{
a \lesssim b
}
for inequalities of the form
\eqn{lesssimB}{
a \leq C b
}
in situations where the precise value or dependence on
other quantities of the constant $C$ is not required.  On the other hand,  when the dependence of the constant
on other inequalities needs to be specified, for example if the constant depends on the norms $\norm{u}_{L^\infty(\Tbb^n)}$ and $\norm{v}_{L^\infty(\Omega)}$, we use the notation
\eqn{lesssimC}{
C = C(\norm{u}_{L^\infty(\Tbb^n)},\norm{v}_{L^\infty(\Omega)}).
}
Constants of this type will always be non-negative, non-decreasing, continuous functions of their arguments.

\subsect{domain}{A simple extension operator}

Given an open set $\Omega$ in $\Gbb^n$, we define the trivial extension operator by
\eqn{chiext}{
\chi_\Omega u(x) = \begin{cases} u(x) & \text{if $x\in \Omega$} \\ 0 & \text{otherwise} \end{cases}.
}
Clearly, this defines a bounded linear operator from $L^p(\Omega)$ to $L^p(\Gbb^n)$.


\subsect{mollifier}{Smoothing operator}


\begin{prop} \label{mollprop}
Suppose $\Omega$ is an open set in  $\Tbb^n$ with a
smooth boundary,  $1\leq p < \infty$ and $s \in \Zbb_{\geq 0}$. Then
there exists a family of continuous linear maps
\eqn{mollprop1}{
J_\lambda \: :\: W^{s,p}(\Tbb^n) \longrightarrow W^{s,p}(\Tbb^n) \quad \lambda > 0
}
satisfying
\gath{mollprop2}{
\norm{J_\lambda \chi_\Omega u}_{W^{k,p}(\Tbb^n)} < \infty \quad k\geq s,\\
 \norm{J_\lambda \chi_\Omega u}_{W^{s,p}(\Omega)} \lesssim \norm{u}_{W^{s,p}(\Omega)} \quad 0<\lambda \leq 1 \\
\intertext{and}
\lim_{\lambda\searrow 0} \norm{J_\lambda \chi_\Omega u - u}_{W^{s,p}(\Omega)}  = 0
}
for all $u\in W^{s,p}(\Omega)$.
\end{prop}
\begin{proof}
On $\Rbb^n$, the proof follows directly from \cite[Theorem 2.29]{AdamsFournier:2003} and the discussion in the section \emph{Approximation by Smooth Functions
on $\Omega$} starting on p. 65 of the same reference. On $\Tbb^n$, the proof follows from using a smooth partition of unity to decompose
functions into a finite sum of functions to which
the results on $\Rbb^n$ apply.
\end{proof}

\begin{cor} \label{mollcor}
Suppose $1\leq p < \infty$, $m\in \Zbb_{\geq 0}$, $s \in \Zbb_{\geq m}$, and let $J_\lambda$ be as defined in Proposition \ref{mollprop}.
Then $J_\lambda$ is a well-defined, continuous linear operator on $\Hc^{m,s}(\Tbb^n)$ satisfying
\gath{mollcor1}{
\norm{J_\lambda  u}_{\Hc^{\ell,k}(\Tbb^n)} < \infty \quad k\geq s,\; \ell \geq m, \; k\geq \ell, \\
 \norm{J_\lambda  u}_{\Hc^{m,s}(\Tbb^n)} \lesssim \norm{u}_{\Hc^{m,s}(\Omega)} \quad 0<\lambda \leq 1
\intertext{and}
\lim_{\lambda\searrow 0} \norm{J_\lambda  u - u}_{\Hc^{m,s}(\Tbb^n)}  = 0
}
for all  $u \in \Hc^{m,s}(\Tbb^n)$.
\end{cor}
\sect{linear}{Linear wave equations}

\subsect{linit}{Initial value problem}

Our proof of the existence and uniqueness of solutions to the IVP \eqref{waveA.1}-\eqref{waveA.2} relies on first
analyzing the following linear IVP:
\lalign{linIVP}{
\del{\mu}(A^{\mu\nu}\del{\nu} U) &=  F  +  \chi_\Omega H \quad \text{in $[0,T)\times \Rbb^n$,} \label{linIVP.1} \\
(U,\del{t} U)|_{t=0} &= (\Ut_0,\Ut_1)  \quad \text{ in $\Rbb^n$,}
\label{linIVP.2}
}
where $\Omega$ is a bounded open set in $\Rbb^n$ with smooth boundary, the initial data
\leqn{lincompat1a}{
(\Ut_0,\Ut_1) \in \Hc^{2,s+1}(\Rbb^n)\times \Hc^{2,s}(\Rbb^n) \quad s\in \Zbb_{>n/2}
}
satisfies the \emph{compatibility conditions}\footnote{As described in the introduction, the time derivatives $\del{t}^\ell U |_{t=0}$ $\ell \geq 2$ are generated
from the initial data \eqref{linIVP.2} by formally differentiating \eqref{linIVP.1} with respect to
$t$ and evaluating at $t=0$. }
\leqn{lincompat1}{
\Ut_\ell := \del{t}^\ell U |_{t=0} \in \Hc^{m_{s+1-\ell},s+1-\ell}(\Rbb^n) \quad \ell=2,\ldots,s,
}
and there exist constants $\gamma,\kappa >0$  for which the matrix $A^{\mu\nu}$ satisfies
\lgath{linAa}{
\frac{1}{\gamma} |\xi|^2 \leq  A^{ij}\xi_i \xi_j \leq \gamma |\xi|^2 \quad \text{for all $\xi \in \Rbb^n$} \label{linAa.1}
\intertext{and}
A^{00} \leq -\kappa. \label{linAa.2}
}

Away from the boundary of $\Omega$, the existence, uniqueness and regularity of solutions to \eqref{linIVP.1}-\eqref{linIVP.2}
can be obtained by appealing to standard results on hyperbolic equations. So this leaves us with analyzing the problem in a neighborhood of the boundary $\del{}\Omega$ where standard results
do not apply due to the jump discontinuity in the term $\chi_\Omega H$. Appealing to the property of finite speed of propagation, we
can, using a smooth change of (spatial) coordinates, locally straighten out the boundary of $\Omega$ and localize the problem to a
spacetime region of the form $[0,T)\times Q_\delta$ where
$\delta$ can be chosen as small as we like. To be specific, we fix a point
$x_0 \in \partial{}\Omega$,
and choose an open neighborhood $\Nc_{x_0,\delta}$ of $x_0$  that is diffeomorphic to $Q_\delta$ such that the
diffeomorphism
\eqn{PsidefA}{
\Phi_{x_0,\delta} \: : \: \Nc_{x_0,\delta} \longrightarrow Q_\delta
}
satisfies
\gath{PsidefB}{
\Phi_{x_0,\delta}(x_0) = 0
\intertext{and}
\Phi_{x_0,\delta}\bigl(\partial{}\Omega\cap \Nc_{x_0,\delta}\bigr) = \{\: (x^1,\ldots,x^{n-1},0) \: |\: -\delta < x^1,x^2,\ldots,x^{n-1} < \delta \: \}.
}
We also demand  that all the derivatives
of $\Phi_{x_0,\delta}^{-1}$ are uniformly bounded pointwise on $Q_\delta$ for $\delta \in (0,1]$. To see that this is possible, we
fix a diffeomorphism $\Phi_{x_0,1}$ from $\Nc_{x_0,1}$ to $Q_1$. We then define diffeomorphisms
\eqn{PsidefC}{
\Psi_{x_0,\delta}:= \Phi_{x_0,1}^{-1}|_{Q_\delta} \: : \:  Q_\delta \longrightarrow \Nc_{x_0,\delta}:= \Phi_{x_0,\delta}^{-1}(Q_\delta)
}
for $\delta\in (0,1]$. Clearly, this family of diffeomorphisms satisfies the required properties. We extend $\Psi_{x_0,\delta}$ to a spacetime map by defining
\leqn{PsidefD}{
\psi_{x_0,\delta} \: :\: [0,T) \times Q_\delta \longrightarrow [0,T)\times \Nc_{x_0,\delta} \: :\: (x^0,x) \longmapsto (x^0,\Psi_{x_0,\delta}(x)),
}
and we let
\eqn{Jdef}{
J^\mu_\nu = \del{\nu}\psi_{x_0,\delta}^\mu
}
denote the Jacobian matrix of the diffeomorphism \eqref{PsidefD} and
\eqn{JdefB}{
\Jch = J^{-1}
}
its inverse.

Next, we define
\lalign{barvars}{
\Ub(\xv) &= U(\psi_{x_0,\delta}(\xv)), \label{barvars.1} \\
\Uh_j(x) &= \Ut_j(\Psi_{x_0,\delta}(x)) \quad j=0,1, \label{barvars.2} \\
\Ab^{\mu\nu}(\xv) &= \det{J(\xv)}\Jch^\mu_\alpha(\xv) A^{\alpha\beta}(\psi_{x_0,\delta}(\xv)) \Jch^\nu_\beta(\xv), \label{barvars.3} \\
\Fb(\xv) &= \det{J(\xv)}F(\psi_{x_0,\delta}(\xv)), \label{barvars.4}
\intertext{and}
\Hb(t,x) &= \det{J(\xv)}H(\psi_{x_0,\delta}(\xv)). \label{barvars.5}
}
Letting,
\eqn{gflat}{
g=\eta_{\mu\nu}dx^\mu dx^\nu \qquad (\eta_{\mu\nu}) = \diag(-1,1,1,1)
}
denote the Minkowski metric, we recall the following pullback formula for a vector field $X=X^\mu\del{\mu}$:
\leqn{pback1}{
\frac{1}{\sqrt{|\gb|}}\del{\mu}\bigl(\sqrt{|\gb|}\Xb^\mu\bigr) =  \left(\frac{1}{\sqrt{|g|}}\del{\mu}\bigl(\sqrt{|g|}X^\mu\bigr) \right)\circ \psi_{x_0,\delta}
}
where
\alin{pback2}{
\Xb^\mu & := (\psi_{x_0,\delta}^*X)^\mu = \Jch^\mu_\nu \Xb^\nu \circ \psi_{x_0,\delta},\\
|g| &:= -\det(\eta_{\mu\nu}) = 1
\intertext{and}
|\gb| & := |\psi_{x_0,\delta}^*g| = \det(J)^2.
}
Setting
\eqn{pback3}{
X^\mu = A^{\mu\nu}\del{\nu}U
}
in \eqref{pback1}, a short calculation using the chain rule and the definitions \eqref{barvars.1}-\eqref{barvars.5} shows that $\Ub$ satisfies the IVP
\lalign{linP}{
\del{\mu}(\Ab^{\mu\nu}\del{\nu} \Ub) &=  \Fb  +  \chi_{Q_1^+} \Hb \quad \text{in $[0,T)\times Q_\delta$ ,} \label{linP.1} \\
(\Ub,\del{t} \Ub )|_{t=0} &= (\Uh_0,\Uh_1)  \quad \text{ in $Q_\delta$,}
\label{linP.2}
}
and the compatibility conditions
\leqn{linPcompat1}{
\Uh_\ell := \del{t}^\ell \Ub |_{t=0} \in \Hc^{m_{s+1-\ell},s+1-\ell}(Q_\delta) \quad \ell=0,\ldots,s.
}

\subsect{linrs}{Projection to the $n$-Torus} Before proceeding with the analysis of the IVP \eqref{linP.1}-\eqref{linP.2},
we first introduce two technical refinements. The first is to exploit the freedom to
choose $\delta$ small by rescaling the fields and working on a fixed domain $Q_1$ instead. With this in mind,
we define
\lalign{epscale}{
u(\xv) &= \frac{\Ub(\delta \xv)-\Ub(\zero)}{\delta}, \label{epscale.1} \\
m^{\mu\nu} & = \Ab^{\mu\nu}(\zero), \label{epscale.2} \\
b^{\mu\nu}(\xv) & = \frac{\Ab^{\mu\nu}(\delta \xv )-m^{\mu\nu}}{\delta^\sigma}, \label{epscale.3} \\
f(\xv) &= \delta\Fb(\delta \xv),  \label{epscale.4}
\intertext{and}
h(\xv) &= \delta \Hb(\delta \xv).  \label{epscale.5}
}
We note that by making a linear change of coordinates we can, due to the conditions \eqref{linAa.1}-\eqref{linAa.2}
satisfied by $A^{\mu\nu}$, always arrange
that
\eqn{etadef}{
(m^{\mu\nu})  = \diag(-1,1,\ldots,1).
}
A short calculation then shows that $u$ satisfies
\lalign{linQ}{
\del{\mu}((m^{\mu\nu}+\ep b^{\mu\nu})\del{\nu} u ) &= f+  \chi_{Q_1^+} h  \quad \text{in $[0,T/\delta)\times Q_1$,} \label{linQ.1} \\
(u,\del{t}u)|_{t=0} &= (\ut_0,\ut_1) := \left(\frac{\Uh_0(\delta x)-\Uh_0(0)}{\delta},\Uh_1(\delta x)\right) \quad \text{ in $Q_1$,}
\label{linQ.2}
}
where
\eqn{epdef}{
\ep = \delta^\sigma.
}

The second technical refinement is to avoid the analytic difficulties that arise from boundary of $Q_1$ and, at the same time, put the
equations in a suitable form using potential theory estimates. This is accomplished by using appropriate
cutoff functions and appealing to the finite speed of propagation in order to ``project'' the evolution equations to
a suitable form defined on
\eqn{Tbbdefb}{
\Tbb^n \cong Q_1/\sim.
}

We now describe the projected IVP. First, we fix points
$x_+ \in \Om$ and $x_- \in \Om^c$ and choose a $\rho >0$ so that $B_{3\rho}(x_+)
\in \Omega_1$ and $B_{3\rho}(x_-) \in \Om^c$. Then we let $\psi$ denote a smooth non-negative
function such that $\psi|_{B_{\rho}(x_{\pm})} =1$ and $\psi|_{\Tbb^n\setminus (B_{2\rho}(x_+)\cup B_{2\rho}(x_-))}=0$.
The projected IVP is then defined by
\lalign{linPN}{
\del{\mu}((m^{\mu\nu}+\ep\phi_1 b^{\mu\nu})\del{\nu} u ) -\psi u &= \phi_1 f  + \phi_1 \chi_\Om h   + \mu \quad \text{in $[0,T/\delta)\times \Tbb^n$,} \label{linPN.1} \\
(u,\del{t}u)|_{t=0} &= (\phi_1\ut_0,\phi_1\ut_1) \quad \text{ in $\Tbb^n$,}
\label{linPN.2}
}
where
\begin{itemize}
\item[(i)]
\eqn{phidef}{
\phi_\eta(x^1,\ldots,x^n) := \phi\left(\frac{4x^1}{\eta}\right)\phi\left(\frac{4x^2}{\eta}\right)\cdots
\phi\left(\frac{4x^n}{\eta}\right)
}
with $\phi \in C^\infty(\Rbb)$ a cutoff function satisfying
$\phi(\tau)=1$ for $|\tau|\leq 1$, $\phi(\tau)=0$ for $|\tau|\geq 2$ and $\phi(\tau)\geq 0$ for
all $\tau \in \Rbb$, and
\item[(ii)]
\eqn{mudef}{
\mu = \sum_{\ell=0}^{s-1} \frac{t^\ell}{\ell !} \mu_\ell
}
where the $\mu_\ell$ are determined in Proposition \ref{muprop} below.
\end{itemize}

\begin{prop} \label{muprop}
Suppose $n\geq 3$ and $s\in \Zbb_{>n/2}$, $0 < \delta \leq 1$, $\sigma = \min\{1,s-n/2\}$  and let
$u_\ell = \del{t}^\ell u|_{t=0}$ and $\ut_\ell = \del{t}^\ell u|_{t=0}$,
where the $\ell^{\text{th}}$ time derivative of $u$ is generated from formally differentiating
\eqref{linPN.1} and \eqref{linQ.1}, respectively.
 Then there exist a $\delta_0 \in (0,1]$, $\eta_0 \in (0,1/4]$ and a  sequence $\mu_\ell \in \Hc^{0,s-1-\ell}(\Tbb^n)$
$\ell=0,1,\ldots s-1$ such that
\eqn{muprop2}{
u_\ell = \phi_1\ut_\ell \quad \ell=0,1, \quad \mu_\ell|_{Q_{\eta_0}} = 0 \quad \ell=0,1,\ldots,s-1
}
and
\alin{muprop3}{
\norm{u(0)}_{E^{s+1}(\Tbb^n)} &\lesssim  \norm{\Ub(0)}_{E^{s+1}(Q_1)},\\
\norm{\mu(t)}_{\Ec^{s-1}(\Tbb^n)}  &\lesssim  \bigl(1+ \norm{\Ab(0)}_{\Ec^s(Q_1)}+ \norm{D\Ab(0)}_{\Ec^{s-1}(Q_1)}\bigr)
\norm{\Ub(0)}_{E^{s+1}(Q_1)}
 \\
&\text{\hspace{4.0cm}}+ \delta\bigl(\norm{\Fb(0)}_{\Ec^{s-1}(Q_1)}+
\norm{\Hb(0)}_{E^{s-1}(Q_1^+)}\bigr)
}
for all $\delta \in (0,\delta_0]$.
\end{prop}
\begin{proof}
In the following, we will use the notation
\eqn{muprop8}{
(\cdot)_\ell = \del{t}^\ell (\cdot)|_{t=0}
}
to denote the $\ell^{\text{th}}$ time derivative of a quantity evaluated at $t=0$,
where the time derivatives of $u$ are generated from formally differentiating
\eqref{linPN.1}. Similarly,
 we use the
\eqn{muprop9}{
\tilde{(\cdot)}_\ell = \del{t}^\ell (\cdot)|_{t=0}
}
to denote the $\ell^{\text{th}}$ time derivative of a quantity evaluated at $t=0$ that depends on $u$
where the time derivatives of $u$ are generated from formally differentiating
\eqref{linQ.1}.

We begin by defining
\gath{muprop10}{
\gb^{\mu\nu} = m^{\mu\nu}+\ep\phi_1 b^{\mu\nu},\quad \fb = 
 \phi_1 f - \del{i}\phi_1 b^{i\nu}\del{\nu}u -\phi_1 \del{\mu}b^{\mu\nu} \del{\nu}u
\intertext{and}
\hb = \phi_1 h,
}
which allow us to write \eqref{linPN.1} as
\leqn{muprop11}{
\gb^{\mu\nu}\del{\mu}\del{\nu}u -\psi u =  \fb +  \chi_{\Om}\hb + \mu.
}
Since $n>n/2$, we have by Sobolev's inequality, see Theorem \ref{Sobolev}, that
\eqn{muprop11a}{
\norm{\phi_1 b^{00}(0)}_{L^\infty(\Tbb^n)} = \max \bigl\{\norm{\phi_1 b^{00}(0)}_{L^\infty(\Om)},\norm{\phi_1 b^{00}(0)}_{L^\infty(\Omega_1^c)}\bigr\} \lesssim \norm{b^{00}(0)}_{H^{0,s}(\Tbb^n)}.
}
By Proposition \ref{scalepropA}, it follows from this inequality that
\eqn{muprop11d}{
\norm{\phi_1 b^{00}(0)}_{L^\infty(\Tbb^n)} \leq C \norm{\Ab(0)}_{\Ec^{s}(Q_1)}
}
for some constant $C>0$ independent of $\delta \in (0,1]$. Therefore, since $\ep=\delta^\sigma$ and
$\sigma >0$,
we can choose $\delta_0 \in (0,1]$ small enough so that
\eqn{muprop11b}{
 \ep \norm{\phi_1 b^{00}(0)}_{L^\infty(\Tbb^n)} \leq \frac{1}{2}
}
for $\delta \in (0,\delta_0]$, and this, in turn, guarantees that
\eqn{muprop11c}{
\frac{1}{2} \leq \gb_{00} \leq \frac{3}{2}.
}
Using this, we can solve \eqref{muprop11} for the $2^{\text{nd}}$ order time derivative to get
\leqn{muprop12}{
 \del{t}^2u = \frac{1}{\gb^{00}}\Bigr(-2\gb^{i0}
 \del{i} \del{t}u - \gb^{ij}\del{i}\del{j} u + \psi u
+ \fb+
\chi_{\Omega_1}\hb
+ \mu \Bigr).
}

Setting
\leqn{muprop13}{
\mu_{0} = 2\gb_0^{i0}\del{i}(\phi_1\ut_1) + \gb_0^{ij}\del{i}\del{j}(\phi_1 \ut_0) -\psi \phi_1 \ut_0 -\phi_{1}\bigl(\fb_0
+\chi_{\Omega_1}\hb_0 \bigr) + \gb_0^{00}\phi_1
\ut_2,
}
we see from \eqref{muprop12} that
\leqn{muprop14}{
u_2 = \phi_1 \ut_2.
}
Moreover, it follows directly from the assumption $s>n/2$ and the multiplication inequality from Proposition \ref{elemE} that
\alin{muprop15}{
\norm{\mu_{0}}_{\Hc^{0,s-1}(\Tbb^n)} \lesssim  \bigl(1+ \ep\norm{b(0)}_{\Ec^s(\Tbb^n)}+\ep\norm{Db(0)}_{\Ec^{s-1}(\Tbb^n)}&\bigr)\norm{u(0)}_{E^{s+1}(\Tbb^n)} \\
&+ \norm{f(0)}_{\Ec^{s-1}(\Tbb^n)}+
\norm{h(0)}_{E^{s-1}(\Omega_1)}.
}

Calculating the $2^{\text{nd}}$ order time derivative of $u$ using \eqref{linQ.1}, we find that
\leqn{muprop16}{
 \del{t}^2 u= \frac{1}{g^{00}}\Bigr(-2g^{i0}
 \del{i}\del{t}u - g^{ij}\del{i}\del{j} u  + k+
\chi_{\Omega_1}h \Bigr),
}
where
\eqn{muprop17}{
g^{\mu\nu} = m^{\mu\nu}+\ep b^{\mu\nu} \AND k = f(u,\del{}u,v) -\del{\mu} b^{\mu\nu}\del{\nu} u.
}
Since
\leqn{muprop18}{
\phi_1|_{Q_\eta}=1 \quad \eta\in [0,1/4],
}
and
\leqn{muprop19}{
\psi|_{\{|x^n|\leq \eta_0\}} =0
}
for $\eta_0$ small enough, we see, with the help of \eqref{linPN.2}, \eqref{muprop14} and \eqref{muprop16}, that $\mu_0$, see \eqref{muprop13}, satisfies
\eqn{muprop20}{
\mu_{0}|_{Q_{\eta_0}} = 0
}
for some $\eta_0 \in (0,1/4]$.

With the base case satisfied, we proceed by induction and assume that
\lgath{muprop21}{
u_\ell = \phi_1\ut_\ell \quad \ell=0,1,\ldots,r+1,  \label{muprop21.1} \\
\mu_\ell|_{Q_{\eta_0}} = 0 \quad \ell=0,1,\ldots,r-1 \notag
\intertext{and}
\norm{\mu_\ell}_{\Hc^{0,s-1-\ell}} \lesssim  \bigl(1+ \ep\norm{b(0)}_{\Ec^s(\Tbb^n)}+\ep\norm{Db(0)}_{\Ec^{s-1}(\Tbb^n)}\bigr)\norm{u(0)}_{E^{s+1}(\Tbb^n)} \notag \\
\hspace{6.0cm} + \norm{f(0)}_{\Ec^{s-1}(\Tbb^n)}+
\norm{h(0)}_{E^{s-1}(\Omega_1)}, \notag
}
where $r<s$.

Differentiating \eqref{muprop12} $r$-times with respect to $t$ and evaluating at $t=0$, we find
that
\lalign{muprop22}{
u_{2+r} = \frac{1}{\gb^{00}_0}\Biggl( -\sum_{\ell=0}^{r-1}{r \choose \ell} \gb^{00}_{r-\ell}
u_{2+\ell} - \sum_{\ell=0}^r {r\choose \ell} \Bigl(\gb^{0i}_{r-\ell}&
\del{i}u_{\ell+1} + \gb^{ij}_{r-\ell}
\del{i}\del{j}u_{\ell}\Bigr) \notag \\
&+\psi u_r + \fb_r +\chi_\Om \hb_r
+ \mu_{r} \Biggr). \label{muprop22.1}
}
Setting
\lalign{muprop23}{
\mu_{r} = \sum_{\ell=0}^{r-1}{r \choose \ell} \gb^{00}_{r-\ell}
u_{2+\ell} + \sum_{\ell=0}^r {r\choose \ell}& \Bigl(\gb^{0i}_{r-\ell}
\del{i}u_{\ell+1} + \gb^{ij}_{r-\ell}
\del{i}\del{j}u_{\ell}\Bigr) \notag \\
&-\psi u_r - \fb_r -\chi_\Om \hb_r
+ \gb^{00}_0\phi_1\ut_{2+r}, \label{muprop23.1}
}
it follows immediately from \eqref{muprop22.1} that
\eqn{muprop23a}{
u_{2+r} = \phi_1\ut_{2+r}.
}
Similarly, differentiating \eqref{muprop16} $r$-times with respect to $t$ and evaluating at $t=0$, we obtain
\lalign{muprop24}{
\ut_{2+r} = \frac{1}{\gt^{00}_0}\Biggl( -\sum_{\ell=0}^{r-1}{r \choose \ell} \gt^{00}_{r-\ell}
\ut_{2+\ell} - \sum_{\ell=0}^r {r\choose \ell} \Bigl(\gt^{0i}_{r-\ell}
\del{i}\ut_{\ell+1}
+ \gt^{ij}_{r-\ell}
\del{i}\del{j}\ut_{\ell}\Bigr) + \kt_r +\chi_\Om \htld_r
 \Biggr). \label{muprop24.1}
}
Clearly, the induction hypothesis \eqref{muprop21.1} together with \eqref{muprop18} implies that
\eqn{muprop25}{
\gt_\ell|_{Q_{\eta_0}} = \gt_\ell|_{Q_{\eta_0}}, \quad \fb_\ell|_{Q_{\eta_0}} = \kt_\ell|_{Q_{\eta_0}} \AND \hb_\ell|_{Q_{\eta_0}} = \htld_\ell|_{Q_{\eta_0}}
}
for $\ell=0,\ldots,r+1$. Consequently, it follows directly from \eqref{muprop19}, \eqref{muprop23.1} and \eqref{muprop24.1}
that
\eqn{muprop25a}{
\mu_r|_{Q_{\eta_0}} = 0.
}
Furthermore, applying the multiplication inequality from Proposition \ref{elemE} to \eqref{muprop23.1}, see the proof of Lemma \ref{linlemA} for similar  estimates, it is not difficult to verify that
\lalign{muprop26}{
\norm{\mu_r}_{\Hc^{0,s-1-r}(\Tbb^n)} \lesssim  \bigl(1+ \ep\norm{b(0)}_{\Ec^s(\Tbb^n)}+&\ep\norm{Db(0)}_{\Ec^{s-1}(\Tbb^n)}\bigr)\norm{u(0)}_{E^{s+1}(\Tbb^n)}\notag \\
&+ \norm{f(0)}_{\Ec^{s-1}(\Tbb^n)}+
\norm{h(0)}_{E^{s-1}(\Omega_1)} \label{muprop26.1}.
}
This completes the induction step.

Finally, observing the scaling definitions \eqref{epscale.1}-\eqref{epscale.5}, it is clear from Propositions \ref{scalepropA}
and \ref{scalepropB} that the estimates
\alin{muprop27}{
\norm{u(0)}_{E^{s+1}(\Tbb^n)} &\lesssim \norm{\Ub(0)}_{E^{s+1}(Q_1)}
\intertext{and}
\norm{\mu_\ell}_{\Hc^{0,s-1-\ell}(\Tbb^n)}  &\lesssim  \bigl(1+ \norm{\Ab(0)}_{\Ec^s(Q_1)}+ \norm{D\Ab(0)}_{\Ec^{s-1}(Q_1)}\bigr)\norm{\Ub(0)}_{E^{s+1}(Q_1)} \\
&\text{\hspace{4.0cm}}+ \delta\bigl(\norm{\Fb(0)}_{\Ec^{s-1}(Q_1)}+
\norm{\Hb(0)}_{E^{s-1}(Q_1^+)}\bigr)
}
for $\ell=0,1,\ldots,s$
are a direct consequence of \eqref{muprop26.1}.
\end{proof}

\subsect{existA}{Existence and uniqueness for the linear system \eqref{linIVP.1}-\eqref{linIVP.2}}

In light of the form of the projected system \eqref{linPN.1}-\eqref{linPN.2}, we now turn our attention
to the following class of linear IVPs:
\lalign{linM}{
\del{\mu}((m^{\mu\nu}+\ep b^{\mu\nu})\del{\nu} u ) &= f +  \chi_{Q_1^+} h + \mu  \quad \text{in $[0,T)\times \Tbb^n$,} \label{linM.1} \\
(u,\del{t}u)|_{t=0} &= (\ut_0,\ut_1) \quad \text{ in $\Tbb^n$,}
\label{linM.2}
}
where the initial data is chosen so that the compatibility conditions
\leqn{linMcompat}{
\ut_\ell := \del{t}^\ell u |_{t=0} \in \Hc^{m_{s+1-\ell},s+1-\ell}(\Rbb^n) \quad \ell=2,\ldots,s
}
are satisfied.

\begin{thm} \label{linthm}
Suppose $n\geq 3$, $\delta \in (0,1]$, $\sigma = \min\{1,s-n/2\}$, $\ep=\delta^\sigma$, $s\in \Zbb_{>n/2}$, $T>0$, $b = (b^{\mu\nu}),
 \del{t}b\in \Xc_T^{s}(\Tbb^n)$, $Db \in \Xc_T^{s-1}(\Tbb^n)$, $f \in \Xc_T^{s}(\Tbb^n)$, $h \in Y^{s}_T(\Omega_1)$,
$\mu \in \Xc_T^{s}(\Tbb^n)$ with $\del{t}^s \mu = 0$,
$(\ut_0,\ut_1) \in \Hc^{2,s+1}(\Tbb^n)\times \Hc^{2,s}(\Tbb^n)$
satisfy the compatibility conditions \eqref{linMcompat}, and let
\eqn{linthm1}{
R= \sup_{0\leq t < T}( \norm{b(t)}_{\Ec^s(\Tbb)} + \norm{Db(t)}_{\Ec^{s-1}(\Tbb)}) \AND \beta(t) = 1+\norm{b(t)}_{\Ec^{s}(\Tbb^n)}+\norm{\del{t}b(t)}_{\Ec^{s}(\Tbb^n)}.
}
Then there exists a constant $c_L=c_L(n,s) >0$ such that the IVP \eqref{linM.1}-\eqref{linM.2} has a unique solution
$u \in CX_T^{s+1}(\Tbb^n)$  whenever $\delta$ is chosen so that $\ep$ satisfies $0\leq \ep \leq \min\left\{\frac{1}{3 c_L R},\frac{1}{3}\right\}$.
Moreover, $u$  satisfies the following estimate
\alin{linthm2}{
\norm{u(t)}_{E^{s+1}(\Tbb^n)}&\leq C(c_L) e^{C(c_L)\int_{0}^T \beta(\tau)\,d\tau} \biggl[ \beta(0)\norm{u(0)}_{E^{s+1}(\Tbb^n)}
+ \norm{f(0)}_{\Ec^{s-1}(\Tbb^n)}
+ \norm{h(0)}_{E^{s-1}(\Omega_1)} \notag \\
+& \norm{\mu(t)}_{\Ec^{s-1}(\Tbb^n)}
+
\int_{0}^T \beta(\tau)\Bigl(\norm{f(\tau)}_{\Ec^s(\Tbb^n)} + \norm{h(\tau)}_{E^s(\Omega_1)} + \norm{\mu(\tau)}_{\Ec^{s-1}(\Tbb^n)} \Bigr)\, d\tau \biggr]
}
for $0\leq t < T$.

\end{thm}
\begin{proof}
Instead of solving the IVP \eqref{linM.1}-\eqref{linM.2} directly. We first regularize the problem using a mollifier to smooth the coefficients
and the initial data with the resulting regularized IVP being
\lalign{linR}{
\del{\mu}((m^{\mu\nu}+\ep J_\lambda(b^{\mu\nu}))\del{\nu} u^\lambda) -\psi u^\lambda &= J_\lambda f  + J_\lambda (\chi_\Om h) + J_\lambda \mu \quad \text{in $[0,T]\times \Tbb^n$,} \label{linR.1} \\
(u^\lambda,\del{t}u^\lambda_\lambda)|_{t=0} &= (J_\lambda \ut_0, J_\lambda \ut_1)  \quad \text{ in $\Tbb^n$,}
\label{linR.2}
}
where $\lambda \in (0,1]$. Applying a standard existence theorem for linear hyperbolic equations, for example see \cite[Ch. 16, Proposition 1.7]{TaylorIII:1996}, we obtain a 1-parameter family of (unique) solutions
\eqn{ulamexist}{
u^\lambda \in \bigcap_{\ell=0}^{s+100} C^\ell([0,T),H^{s+100-\ell}(\Tbb^n))\quad 0 < \lambda \leq 1.
}

Our goal now is to derive $\lambda$-independent estimates for $u^\lambda$ and then
use these estimates to obtain
a solution to the IVP \eqref{linM.1}-\eqref{linM.2} by letting $\lambda \searrow 0$
and extracting a (weakly) convergent subsequence that converges to a solution.
The proof of the $\lambda$-independent estimates involves using elliptic estimates to estimate the
first $s-1$ time derivatives of $u^\lambda$ followed by hyperbolic estimates to estimate the remaining
$s$ and $s+1$ time derivatives.

\bigskip

\noindent \textit{\textbf{Elliptic estimates}:} Differentiating \eqref{linM.1} $k$ times  with respect to $t$, we observe that $u_k=\del{t}^k u$ satisfies the elliptic system
 \leqn{linB}{
\Delta u^\lambda_k  - \psi u^\lambda_k =  
u^\lambda_{k+2} + \ep\bigl(q^0_{k} + q^1_{k} + q^2_{k}\bigl)+  J_\lambda f_k+  J_\lambda \bigl(\chi_{\Omega_1} h_k \bigr) + J_\lambda \mu_k,
}
where
\lalign{linBa}{
q^0_{k} &= -\sum_{\ell=0}^{k} \binom{k}{\ell}\bigl( J_\lambda(b^{ij}_{k-\ell}) \del{i}\del{j}u^\lambda_\ell +  J_\lambda ( \del{i} b^{ij}_{k-\ell}) \del{j} u^\lambda_\ell
+ J_\lambda (b^{0j})_{k-\ell+1}\del{j} u^\lambda_\ell \bigr) \notag \\
& -\sum_{\ell=1}^{k}\binom{k}{\ell-1} \bigl( 2 J_\lambda( b^{0j}_{k-\ell+1})\del{j} u^\lambda_\ell
+  J_\lambda(\del{j} b^{0j}_{k-\ell+1})u^\lambda_\ell + b^{00}_{k-\ell+2} u^\lambda_\ell \bigr) - \sum_{\ell=2}^{k}\binom{k}{\ell-2} J_\lambda(b^{00}_{k-\ell+2}) u^\lambda_\ell, \notag \\
q^1_{k} & = - k J_\lambda(b^{00}_1) u^\lambda_{k+1} - 2 J_\lambda(b^{0j}_0) \del{j} u^\lambda_{k+1} - J_\lambda(\del{j} b^{0j}_0) u^\lambda_{k+1} -
J_\lambda(b_1^{00}) u^\lambda_{k+1}, \notag  \\
q^2_{k} &= -J_\lambda(b^{00}_0) u^\lambda_{k+2} \notag 
}
and
\eqn{Deltadef}{
\Delta = \delta^{ij}\del{i}\del{j}
}
is the Euclidean Laplacian.

\begin{lem} \label{linlemA}
The following estimates hold:
\alin{linlemA1}{
&\norm{q^0_{k}}_{\Hc^{0,s+1-(k+2)}(\Tbb^n)} \lesssim \bigl(\norm{b}_{\Ec^{s}(\Tbb^n)} + \norm{Db}_{\Ec^{s-1}(\Tbb^n)} \bigr) \norm{u^\lambda}_{E^{s+1,k}(\Tbb^n)},\\
&\norm{q^1_{k}}_{\Hc^{0,s+1-(k+2)}(\Tbb^n)} \lesssim \bigl(\norm{b}_{\Ec^{s}(\Tbb^n)} + \norm{Db}_{\Ec^{s-1}(\Tbb^n)} \bigr) \norm{u^\lambda_{k+1}}_{\Hc^{1,s+1-(k+1)}(\Tbb^n)}
\intertext{and}
&\norm{q^2_{k}}_{\Hc^{0,s+1-(k+2)}(\Tbb^n)} \lesssim \norm{b}_{\Ec^{s}(\Tbb^n)} \norm{u^\lambda_{k+2}}_{\Hc^{0,s+1-(k+2)}(\Tbb^n)}
}
for $0\leq k \leq s-1$.
\end{lem}
\begin{proof}
To begin, we observe that
\eqn{linlemA2}{
\norm{ J_\lambda(b^{ij}_{k-\ell}) \del{i}\del{j}u^\lambda_\ell}_{\Hc^{0,s+1-(k+2)}(\Tbb^n)} \lesssim
\norm{b_{k-\ell} }_{\Hc^{0,s-(k-\ell)}(\Tbb^n)}\norm{D^2u^\lambda_\ell}_{\Hc^{0,s-1-\ell}(\Tbb^n)} \quad 0\leq \ell \leq k \leq s-1
}
follows directly from Proposition \ref{elemE} since $s>n/2$. Similar arguments show that the estimates
\alin{linlemA3}{
&\norm{ J_\lambda(\del{i}b^{ij}_{k-\ell})\del{j}u^\lambda_\ell}_{\Hc^{0,s+1-(k+2)}(\Tbb^n)} \lesssim
\norm{Db_{k-\ell} }_{\Hc^{0,s-1-(k-\ell)}(\Tbb^n)}\norm{D u^\lambda_\ell}_{\Hc^{0,s-\ell}(\Tbb^n)} \\
&\norm{ J_\lambda(b^{0j}_{k-\ell+1})\del{j} u^\lambda_\ell}_{\Hc^{0,s+1-(k+2)}(\Tbb^n)} \lesssim
\norm{b_{k-\ell+1} }_{\Hc^{0,s-(k-\ell+1)}(\Tbb^n)}\norm{D u^\lambda_\ell}_{\Hc^{0,s-\ell}(\Tbb^n)} \\
}
hold for $0\leq \ell \leq k \leq s-1$. These inequalities
give us the following estimate for the first sum in $q^0_k$:
\alin{linlemA4}{
\biggl\| \sum_{\ell=0}^{k} \binom{k}{\ell}\bigl(  J_\lambda(b^{ij}_{k-\ell}) \del{i}\del{j}u^\lambda_\ell +  J_\lambda(\del{i} b^{ij}_{k-\ell}) \del{j} u^\lambda_\ell
+ & J_\lambda(b^{0j}_{k-\ell+1})\del{j} u^\lambda_\ell \bigr) \biggr\|_{\Hc^{0,s+1-(k+2)}(\Tbb^n)}  \\
&\lesssim \bigl(\norm{b}_{\Ec^{s}(\Tbb^n)} + \norm{Db}_{\Ec^{s-1}(\Tbb^n)} \bigr) \norm{u^\lambda}_{E^{s+1,k}(\Tbb^n)}
}
for $0\leq k\leq s-1$. The required estimates for the rest of $q^0_k$ and $q^1_{k}$, $q^2_{k}$ follow from similar
arguments.
\end{proof}

We collect the equations \eqref{linB} $(0\leq k \leq s-1)$ into a single system
\leqn{linC}{
L_\ep (\uv^\lambda_{s-1}) = \Kv_{s-1},
}
where
\eqn{linCa}{
\Kv_{s-1} = \fv_{s-1}+ J_\lambda \bigl(\chi_{\Omega_1} \hv_{s-1}\bigr) + \muv_{s-1}
+ \ep \bigl(0,\ldots,0,q^2_{s-2},q^1_{s-1}+q^2_{s-1}\bigr)^\text{Tr}
}
and
\eqn{linCb}{
L_\ep = L_0 - \ep L_1
}
with
\eqn{linCc}{
L_0= \begin{pmatrix}
\Delta-\psi &      0          &  -1          &  0            & \cdots  &  0         \\
    0       &   \Delta - \psi &   0          &  -1           &         &  0         \\
    0       &      0          &  \Delta-\psi &   0           & \ddots  &  \vdots    \\
    0       &      0          &   0          &   \Delta-\psi & \ddots  &   -1       \\
            &      \vdots     &              &               & \ddots  &     0      \\
    0       &       0         &   0          &    0          &  \cdots &   \Delta-\psi \\
 \end{pmatrix}
}
and
\eqn{linCd}{
L_1(\uv_{s-1}) =  \qv^0_{s-1} + (q^1_0 + q^2_0,\ldots,q^1_{s-3}+q^2_{s-3},q^1_{s-1},0)^\text{Tr}.
}

Next, we observe that Lemma \ref{linlemA} shows that the operator norm of
\eqn{L1map}{
L_1\: : \: X^{s+1,s-1} \longrightarrow \Xc^{s-1,s-1}
}
is bounded by
\eqn{L1est}{
\norm{L_1}_{\op} \leq C_{L_1}( \norm{b}_{\Ec^s(\Tbb)} + \norm{Db}_{\Ec^{s-1}(\Tbb)})
}
for some constant $C_{L_1} > 0$. Also, it is clear from Proposition \ref{potprop} and
the tridiagonal structure of $L_0$ that
$L_0 \, :\, X^{s+1,s-1} \rightarrow \Xc^{s-1,s-1}$ is an isomorphism.
From these facts, we get, via the Born series
\eqn{linlemB4}{
(\id - \ep L_0^{-1}L_1)^{-1} = \sum_{k=0}^\infty \ep^k (L_0^{-1}L_1)^k \quad \ep \norm{L_0^{-1}L_1} < 1,
}
that $L_{\ep} \, :\, X^{s+1,s-1} \rightarrow \Xc^{s-1,s-1}$
is invertible and the estimate
\leqn{linlemB6}{
\norm{L_\ep^{-1}}_{\op} \leq \frac{\norm{L_0^{-1}}_{\op}}{1-\ep ( \norm{b}_{\Ec^s(\Tbb^n)} + \norm{Db}_{\Ec^{s-1}(\Tbb^n)})\norm{L_0^{-1}}_{\op}C_{L_1}}.
}
holds
whenever $\delta \in (0,1]$ is chosen small enough so that $\ep=\delta^\sigma$ satisfies
\eqn{linlemB7}{
\ep ( \norm{b}_{\Ec^s(\Tbb)} + \norm{Db}_{\Ec^{s-1}(\Tbb^n)})\norm{L_0^{-1}}_{\op}C_{L_1} < 1.
}
It then follows directly from equation \eqref{linC}, the estimate \eqref{linlemB6} and Lemma \ref{linlemA} that there exists a constant $c_L=c_L(\norm{L_0^{-1}}_{\op},C_{L_1})>0$ such that
\lalign{uvest1}{
\norm{u^\lambda(t)}_{E^{s+1,s-1}(\Tbb^n)} \leq \frac{c_L}{1-\ep c_L R}\Bigl(&\ep R\bigl(\norm{u^\lambda(t)}_{E^{s+1,s-1}(\Tbb^n)}
+ \norm{u^\lambda_s(t)}_{E(\Tbb^n)} \bigr)   \notag \\ &+ \norm{f(t)}_{\Ec^{s-1}(\Tbb^n)}+ \norm{h(t)}_{E^{s-1}(\Omega_1)} + \norm{\mu(t)}_{\Ec^{s-1}(\Tbb^n)}
\Bigr)
 \label{uvest1.1}
}
for $0\leq t < T$,
where
\eqn{Rdef}{
R = \sup_{0\leq t < T}( \norm{b(t)}_{\Ec^s(\Tbb)} + \norm{Db(t)}_{\Ec^{s-1}(\Tbb)})
}
and $\delta$ is chosen small enough so that $\ep=\delta^\sigma$ satisfies
\leqn{Rbound}{
\ep c_L R < 1.
}
Choosing $\delta$ so that
\eqn{epdefa}{
\ep = \min\left\{\frac{1}{3c_L R},\frac{1}{3}\right\},
}
we can write \eqref{uvest1.1} as
\leqn{uvest2}{
\norm{u^\lambda(t)}_{E^{s+1,s-1}(\Tbb^n)} \leq \norm{u^\lambda_s(t)}_{E(\Tbb^n)}+ 2c_L \bigl( \norm{f(t)}_{\Ec^{s-1}(\Tbb^n)}+
\norm{h(t)}_{E^{s-1}(\Omega_1)} + \norm{\mu(t)}_{\Ec^{s-1}(\Tbb^n)} \bigr)
}
for $0\leq t < T$.
Writing $f(t)$ and $h(t)$ as
\eqn{fhint}{
f(t) = f(0) + \int_{0}^t \del{t}f(\tau)\, d\tau \AND  h(t) = h(0) + \int_{0}^t \del{t}h(\tau)\, d\tau,
}
respectively, we see that
\eqn{fhest}{
\norm{f(t)}_{\Ec^{s-1}(\Tbb^n)} + \norm{h(t)}_{E^{s-1}(\Omega_1)} \leq \norm{f(0)}_{\Ec^{s-1}(\Tbb^n)}+ \norm{h(0)}_{E^{s-1}(\Omega_1)} + \int_{0}^t\bigl( \norm{f(\tau)}_{\Ec^{s}(\Tbb^n)}
+ \norm{h(\tau)}_{\Ec^{s}(\Omega_1)}\bigr) \, d\tau.
}
Combining this inequality with \eqref{uvest2}, we arrive at the estimate
\lalign{uvest3}{
\norm{u^\lambda(t)}_{E^{s+1,s-1}(\Tbb^n)} \leq \norm{u^\lambda_s(t)}_{E(\Tbb^n)}&+ 2c_L\Bigl(\norm{f(0)}_{\Ec^{s-1}(\Tbb^n)} + \norm{h(0)}_{E^{s-1}(\Omega_1)} \notag \\ &+
\int_{0}^t\bigl( \norm{f(\tau)}_{\Ec^{s}(\Tbb^n)}
+ \norm{h(\tau)}_{\Ec^{s}(\Omega_1)}\bigr) \, d\tau  + \norm{\mu(t)}_{\Ec^{s-1}(\Tbb^n)}
\Bigr),
 \label{uvest3.1}
}
which holds for  $0\leq t < T$.

\bigskip

\noindent\textit{\textbf{Hyperbolic estimates:}} By assumption $\del{t}^s \mu = 0$. Therefore, differentiating \eqref{linM.1} $s$-times with respect to $t$ yields
\leqn{linH}{
\del{\mu}\bigl( \bigl(m^{\mu\nu} + \ep J_\lambda (b^{\mu\nu})\bigr)\del{\nu} u^\lambda_s \bigr) - \psi u^\lambda_s = \ep\del{\mu}\bigl( p^\mu - s J_\lambda(b^{\mu 0}_1) u_s^\lambda \bigr)+ J_\lambda f_s  + J_\lambda (\chi_\Om h_s),
}
where
\eqn{pdef}{
p^\mu = -\sum_{\ell=0}^{s-1} \binom{s}{\ell}J_\lambda(b^{\mu j}_{s-\ell}) \del{j} u_\ell^\lambda
- \sum_{\ell=0}^{s-2} \binom{s}{\ell} J_\lambda(b^{\mu 0}_{s-\ell})u_{\ell+1}^\lambda.
}

Recalling that $s>n/2$, it follows directly from Sobolev's inequality, see Theorem \ref{Sobolev},  and Corollary \ref{mollcor} that the inequalities
\lalign{best}{
\norm{J_\lambda b^{\mu \nu}(t)}_{L^\infty(\Tbb^n)} &\lesssim \norm{J_\lambda b^{\mu \nu}(t)}_{\Hc^{0,s}(\Tbb^n)} \lesssim \norm{b(t)}_{\Ec^{s}(\Tbb^n)} ,  \label{best.1}\\
\norm{\del{t}J_\lambda b^{\mu \nu}(t)}_{L^\infty(\Tbb^n)} &\lesssim \norm{J_\lambda b^{\mu \nu}_1(t)}_{\Hc^{0,s}(\Tbb^n)} \lesssim  \norm{b(t)}_{\Ec^{s+1}(\Tbb^n)}, \label{best.2} \\
\norm{\del{t}J_\lambda b^{\mu 0}_1(t)}_{L^n(\Tbb^n)} &\lesssim  \norm{J_\lambda b^{\mu 0}_2(t)}_{\Hc^{0,s-1}(\Tbb^n)}  \lesssim  \norm{\del{t}b(t)}_{\Ec^{s}(\Tbb^n)}  \label{best.3}
\intertext{and}
\norm{J_\lambda b^{\mu 0}_1(t)}_{L^\infty(\Tbb^n)} &\lesssim
\norm{J_\lambda b^{\mu 0}_1(t)}_{\Hc^{0,s}(\Tbb^n)}
\lesssim \norm{\del{t}b(t)}_{\Ec^{s}(\Tbb^n)}
}
are satisfied for $0\leq t < T$. Also, a slight adaptation of the arguments used to prove Lemma \ref{linlemA} show that
\lalign{pest}{
\norm{p^\mu(t)}_{L^2(\Tbb^n)} &\lesssim \norm{b(t)}_{\Ec^{s}(\Tbb^n)}\norm{u^\lambda(t)}_{E^{s+1,s-1}(\Tbb^n)} \label{pest.1}
\intertext{and}
\norm{\del{t} p^\mu (t)}_{L^2(\Tbb^n)} &\lesssim \bigl(\norm{b(t)}_{\Ec^{s}(\Tbb^n)}+\norm{\del{t}b(t)}_{\Ec^{s}(\Tbb^n)}\bigr)\bigl(\norm{u^\lambda(t)}_{E^{s+1,s-1}(\Tbb^n)}
+ \norm{u^\lambda_s(t)}_{H^1(\Tbb^n)}\bigr)
\label{pest.2}\norm{A(t)}_{\Ec^{s+1}(\Rbb^n)}
}
hold for $0\leq t < T$.

We now are in a position to apply the energy estimates from Theorem \ref{weakthm} to $u_s^\lambda$ since it solves
the wave equation \eqref{linH}. Doing so, we see that these energy estimates in conjunction with the inequalities \eqref{Rbound}, \eqref{best.1}-\eqref{pest.2} and
Corollary \ref{mollcor} show that $u^\lambda_s$ satisfies
\alin{usest1}{
\norm{u^\lambda_s(t)}_{E(\Tbb^n)} \leq C(c_L)\Bigl[\beta(0)&\norm{u^\lambda(0)}_{E^{s+1}(\Tbb^n)}
+ \int_{0}^t \beta(\tau)\Bigl(\norm{u^\lambda(\tau)}_{E^{s+1,s-1}} \\
 &+ \norm{u^\lambda(\tau)}_{E(\Tbb^n)}\Bigr)
+ \norm{f_s(\tau)}_{L^2(\Tbb^n)} + \norm{h_s(\tau)}_{L^2(\Omega_1)} \, d\tau
\Bigr]
}
for $0\leq t < T$, where
\eqn{beta(t)def}{
\beta(t) = 1+\norm{b(t)}_{\Ec^{s}(\Tbb^n)}+\norm{\del{t}b(t)}_{\Ec^{s}(\Tbb^n)}.
}
Combining this estimate with \eqref{uvest2} gives
\lalign{usest2}{
\norm{u^\lambda_s(t)}_{E(\Tbb^n)} &\leq C(c_L)\biggl[\beta(0)\norm{u^\lambda(0)}_{E^{s+1}(\Tbb^n)}  \notag
\\
 + \int_{0}^t \beta(\tau)& \Bigl(\norm{u^\lambda(\tau)}_{E(\Tbb^n)}
+ \norm{f(\tau)}_{\Ec^s(\Tbb^n)} + \norm{h(\tau)}_{E^s(\Omega_1)} +  \norm{\mu(\tau)}_{\Ec^{s-1}(\Tbb^n)} \Bigr)
\, d\tau
\biggr] \label{usest2.1}
}
for $0\leq t < T$.

Together, the estimates \eqref{uvest3.1} and \eqref{usest2.1} show that $u^\lambda$ satisfies the uniform bound
\alin{uvest4}{
\norm{u^\lambda(t)}_{E^{s+1}(\Tbb^n)} &\leq C(c_L)\Bigl[ \beta(0)\norm{u^\lambda(0)}_{E^{s+1}(\Tbb^n)} + \norm{f(0)}_{\Ec^{s-1}(\Tbb^n)}
+ \norm{h(0)}_{E^{s-1}(\Omega_1)} \notag \\
+ \norm{\mu(t)}_{\Ec^{s-1}(\Tbb^n)}& + \int_{0}^t  \beta(\tau)\Bigl(\norm{u^\lambda(\tau)}_{E^{s+1}(\Tbb^n)}
+ \norm{f(\tau)}_{\Ec^s(\Tbb^n)} + \norm{h(\tau)}_{E^s(\Omega_1)} +  \norm{\mu(\tau)}_{\Ec^{s-1}(\Tbb^n)}\Bigr)
\, d\tau
\Bigr] 
}
which, in turn, implies via Gronwall's inequality that
\lalign{uvest5}{
\norm{u^\lambda(t)}_{E^{s+1}(\Tbb^n)}&\leq C(c_L) e^{C(c_L)\int_{0}^T \beta(\tau)\,d\tau} \biggl[ \beta(0)\norm{u^\lambda(0)}_{E^{s+1}(\Tbb^n)}
+ \norm{f(0)}_{\Ec^{s-1}(\Tbb^n)}
+ \norm{h(0)}_{E^{s-1}(\Omega_1)} \notag \\
+& \norm{\mu(t)}_{\Ec^{s-1}(\Tbb^n)}
+
\int_{0}^T \beta(\tau)\Bigl(\norm{f(\tau)}_{\Ec^s(\Tbb^n)} + \norm{h(\tau)}_{E^s(\Omega_1)} + \norm{\mu(\tau)}_{\Ec^{s-1}(\Tbb^n)} \Bigr)\, d\tau \biggr]  \label{uvest5.1}
}
for $0\leq t < T$. This inequality implies that $u^\lambda$ is a bounded 1-parameter family of solutions in $X^{s+1}_T(\Tbb^n)$ to the IVP \eqref{linR.1}-\eqref{linR.2}, and consequently, we know, by standard arguments,
that there exists a weakly convergent subsequence, again denoted $u^\lambda$, that converges as $\lambda \searrow 0$ to the unique weak solution
$u\in X^{s+1}_T(\Tbb^n)$  of \eqref{linM.1}-\eqref{linM.2}. Moreover, by the uniqueness of weak limits, we see, by sending $\lambda \searrow 0$
in the estimate \eqref{uvest5.1}, that $u$ satisfies
\lalign{uvest6}{
\norm{u(t)}_{E^{s+1}(\Tbb^n)}&\leq C(c_L) e^{C(c_L)\int_{0}^T \beta(\tau)\,d\tau} \biggl[ \beta(0)\norm{u(0)}_{E^{s+1}(\Tbb^n)}
+ \norm{f(0)}_{\Ec^{s-1}(\Tbb^n)}
+ \norm{h(0)}_{E^{s-1}(\Omega_1)} \notag \\
+& \norm{\mu(t)}_{\Ec^{s-1}(\Tbb^n)}
+
\int_{0}^T \beta(\tau)\Bigl(\norm{f(\tau)}_{\Ec^s(\Tbb^n)} + \norm{h(\tau)}_{E^s(\Omega_1)} + \norm{\mu(\tau)}_{\Ec^{s-1}(\Tbb^n)} \Bigr)\, d\tau \biggr] \label{uvest6.1}
}
for $0\leq t < T$. Finally, a straightforward calculation shows that the difference $u-u_\lambda$ satisfies a linear equation of the type \eqref{linM.1}-\eqref{linM.2}, and
hence, also an estimate of the type \eqref{uvest6.1} from which it follows that
\eqn{uvest7}{
\norm{u-u_\lambda}_{X^{s+1}_T(\Tbb^n)} \lesssim \norm{u(0)-J_\lambda u(0)}_{E^{s+1}(\Tbb^n)} + c(\lambda)
}
for some constant $c(\lambda)$ satisfying $\lim_{\lambda \searrow 0} c(\lambda) = 0$. Since $u_\lambda \in CX^{s+1}_T(\Tbb^n)$, this estimate implies that
$u_\lambda(t)$ converges uniformly on $[0,T]$ to $u(t)$, thereby showing that $u\in CX^{s+1}_T(\Tbb^n)$.
\end{proof}

With the existence and uniqueness for the system \eqref{linM.1}-\eqref{linM.2} established, it is now straightforward,
using the finite speed of propagation, to conclude an analogous uniqueness and existence result for the original system
\eqref{linIVP.1}-\eqref{linIVP.2}.

\begin{thm} \label{linGthm}
Suppose $n\geq 3$, $s\in \Zbb_{>n/2}$, $T>0$, $A = (A^{\mu\nu}),\del{t}A \in \Xc_T^{s}(\Rbb^n)$, $DA \in \Xc_T^{s-1}(\Rbb^n)$, $F \in \Xc_T^{s}(\Rbb^n)$, $H \in Y^{s}_T(\Omega)$,
$(\Ut_0,\Ut_1) \in \Hc^{2,s+1}(\Rbb^n)\times \Hc^{2,s}(\Rbb^n)$
satisfy the compatibility conditions \eqref{lincompat1}, and let
\gath{linGthm1}{
\alpha = \sup_{0\leq t <  T}\bigl(\norm{A(t)}_{\Ec^s(\Rbb^n)}+\norm{D A (t)}_{\Ec^{s-1}(\Rbb^n)}\bigr), \quad \beta(t) = 1+ \norm{A(t)}_{\Ec^{s}(\Rbb^n)}+ \norm{\del{t}A(t)}_{\Ec^{s}(\Rbb^n)}
\intertext{and}
\gamma = \int_{0}^{T}\beta(\tau) \,d\tau.
}
Then the IVP \eqref{linIVP.1}-\eqref{linIVP.2} has a unique solution
$U \in CX_T^{s+1}(\Rbb^n)$ and  satisfies the estimate
\alin{linGthm2}{
\norm{U(t)}_{ E^{s+1}(\Rbb^n)}& \leq  c(\alpha,\beta(0),\gamma)\biggl[ \norm{U(0)}_{E^{s+1}(\Rbb^n)}
+ \norm{F(0)}_{\Ec^{s-1}(\Rbb^n)} \notag \\ & + \norm{H(0)}_{E^{s-1}(\Omega)}
+ \int_{0}^{T}\beta(\tau)\bigl( \norm{F(\tau)}_{\Ec^s(\Rbb^n)} + \norm{H(\tau)}_{E^s(\Omega)}\bigr)\, d\tau \biggr]
}
for $0\leq t < T$.
\end{thm}
\begin{proof}
We begin by letting $u \in CX^{s+1}_T(\Tbb^n)$ be the unique solution to the IVP \eqref{linPN.1}-\eqref{linPN.2} from Theorem \ref{linthm} with $\delta \in (0,1]$ chosen small enough so that $\ep=\delta^\sigma$ satisfies
\leqn{linGthm3}{
\ep = \min\left\{\frac{1}{3 c_L R},\frac{1}{3}\right\}
}
with
\leqn{linGthm5}{
R = \sup_{0\leq t <  T}\bigl(\norm{b(t)}_{\Ec^s(\Tbb^n)}+\norm{Db(t)}_{\Ec^{s-1}(\Tbb^n)}\bigr).
}
We know from Proposition \ref{muprop} that there exists an $\eta_0 \in (0,1/4]$ such that the initial data $(u|_{t=0},\del{t}u|_{t=0})$, the source term $\mu(t)$ and
the coefficient $\psi$ satisfy
\alin{linGthm6}{
\bigl(u|_{\{0\}\times Q_{\eta_0}},\del{t}u|_{\{0\}\times Q_{\eta_0}}\bigr) &= \bigl(\ut_0|_{Q_{\eta_0}},\ut_1|_{Q_{\eta_0}}\bigr),\\
\psi|_{Q_{\eta_0}} &= 0
\intertext{and}
\mu|_{[0,T]\times Q_{\eta_0}} & = 0,
}
respectively. Appealing to the finite propagation speed for solutions of wave equations, we see, shrinking $T$ if necessary, that
 $u$ solves the IVP \eqref{linQ.1}-\eqref{linQ.2} on the spacetime region $[0,T)\times Q_{\eta_0/2}$. Moreover, it
 follows directly
  from Theorem \ref{linthm} and
 Propositions \ref{muprop}, \ref{scalepropA} and \ref{scalepropB} that $u$ satisfies the estimate
\lalign{linGthm7}{
 \norm{u(t)}_{E^{s+1}(Q_{\eta_0/2})}\leq & c(\alphab,\betab(0),\gammab)
\Biggl[ \norm{U(0)}_{E^{s+1}(Q_1)}
+ \delta \norm{\Fb(0)}_{\Ec^{s-1}(Q_1)} \notag \\ & +  \delta \norm{\Hb(0)}_{E^{s-1}(Q_1^+)}
+ \int_{0}^{\delta T}\betab(\tau)\bigl( \norm{\Fb(\tau)}_{\Ec^s(Q_1)} + \norm{\Hb(\tau)}_{E^s(Q^+_1)}\bigr)\, d\tau \Biggr]  \label{linGthm7.1}
}
for $0\leq t < T$, where
\lalign{abgdef}{
\alphab &= \sup_{0\leq t <  \delta T}\bigl(\norm{\Ab(t)}_{\Ec^s(Q_1)}+\norm{D\Ab(t)}_{\Ec^{s-1}(Q_1)}\bigr), \label{abgdef.1}\\
\betab(t) &= 1+ \norm{\Ab(t)}_{\Ec^{s}(Q_1)}+\norm{\del{t}\Ab(t)}_{\Ec^{s}(Q_1)}, \notag \\ 
\gammab &= \frac{1}{\delta}\int_{0}^{\delta T}\betab(\tau) \,d\tau  \notag
}
and $\Ub$,$\Ab$,$\Fb$ and $\Hb$ are as defined previously by \eqref{epscale.1}-\eqref{epscale.5}. Next,
a simple change of variable argument shows that
\eqn{linGthm8}{
\norm{\Ub(t)}_{ E^{s+1}(Q_{\delta \eta_0/2} )} \leq c(1/\delta)\norm{u(t/\delta)}_{E^{s+1}(Q_{\eta_0/2})}
}
for $0 \leq t < \delta T$, while the inequality
\eqn{linGthm9}{
\frac{1}{\delta} \lesssim 1+\alphab^{1/\sigma}
}
follows from \eqref{linGthm3}, \eqref{abgdef.1} and Proposition \ref{scalepropA}. Using these estimates together with \eqref{linGthm7.1},
we see that
\lalign{linGthm10}{
\norm{\Ub(t)}_{ E^{s+1}(Q_{\delta\eta_0/2} )} \leq & c(\alphab,\betab(0),\gammab)
\biggl[ \norm{\Ub(0)}_{E^{s+1}(Q_1)}
+ \norm{\Fb(0)}_{\Ec^{s-1}(Q_1)} \notag \\ & + \norm{\Hb(0)}_{E^{s-1}(Q_1^+)}
+ \int_{0}^{T^*}\betab(\tau)\bigl( \norm{\Fb(\tau)}_{\Ec^s(Q_1)} + \norm{\Hb(\tau)}_{E^s(Q^+_1)}\bigr)\, d\tau \biggr] \label{linGthm10.1}
}
for $0\leq t < T^*$, where $T^*=\delta T$.

Since $u$ solves the IVP \eqref{linQ.1}-\eqref{linQ.2} on the spacetime region $[0,T)\times Q_{\eta_0/2}$, $\Ub$ must solve the
IVP \eqref{linP.1}-\eqref{linP.2} on $[0,T^*)\times Q_{\delta\eta/2}$. Recalling the definitions \eqref{barvars.1}-\eqref{barvars.5}, it is clear that
$U(t,x)=\Ub(t,\Phi_{x_0,\delta\eta_0/2}(x))$ satisfies the IVP \eqref{linIVP.1}-\eqref{linIVP.2} on
$[0,T^*)\times \Nc_{x_0,\delta\eta_0/2}$. We also see, with the help of \eqref{linGthm10.1}, that
$U$ satisfies the estimate
\lalign{linGthm11}{
\norm{U(t)}_{ E^{s+1}(\Nc_{x_0,\eta_0/2})}& \leq  c(\alpha,\beta(0),\gamma) \biggl[ \norm{U(0)}_{E^{s+1}(\Rbb^n)}
+ \norm{F(0)}_{\Ec^{s-1}(\Rbb^n)} \notag \\ & + \norm{H(0)}_{E^{s-1}(\Omega)}
+ \int_{0}^{T^*}\beta(\tau)\bigl( \norm{F(\tau)}_{\Ec^s(\Rbb^n)} + \norm{H(\tau)}_{E^s(\Omega)}\bigr)\, d\tau \biggr]  \label{linGthm11.1}
}
for $0\leq t < T^*$, where
\gath{linGthm12a}{
\alpha = \sup_{0\leq t \leq  T^*}\bigl(\norm{A(t)}_{\Ec^s(\Rbb^n)}+\norm{D A (t)}_{\Ec^{s-1}(\Rbb^n)}\bigr), \quad
\beta(t) = 1+ \norm{A(t)}_{\Ec^{s}(\Rbb^n)}+\norm{\del{t}A(t)}_{\Ec^{s}(\Rbb^n)}
\intertext{and}
\gamma = \int_{0}^{T^*}\beta(\tau) \,d\tau.
}

Since $x_0\in \Omega$ was chosen arbitrarily and $\Omega$ is bounded, we can, using the finite propagation speed and the uniqueness of
solutions, piece together a finite number of solutions $\{U_{j}\}_{j=1}^M$  to  \eqref{linIVP.1}-\eqref{linIVP.2} defined on regions
$\{[0,T^*)$$\times$ $\Nc_{x_j,\delta\eta_0/2}\}_{j=1}^M$ such that $\partial\Omega$ $\subset$ $\cup_{j=1}^M \Nc_{x_j,\delta\eta_0/2}$ to
obtain
a solution $\hat{U}$ to \eqref{linIVP.1}-\eqref{linIVP.2} defined on $[0,T^*)\times \Nc$, where $\Nc$ is an open neighborhood
of $\del{}\Omega$. Away from the boundary $\del{}\Omega$, the existence and uniqueness of solutions to  \eqref{linIVP.1}-\eqref{linIVP.2}
satisfying the usual energy estimates is
guaranteed by standard results. Piecing together this solution with $\hat{U}$, we obtain
a solution to \eqref{linIVP.1}-\eqref{linIVP.2} on a time interval $[0,T^*)$ with $T^* > 0$ independent of the initial data. Moreover, it is clear from \eqref{linGthm11.1} and the familiar energy
estimates for wave equations, that $U$ satisfies the estimate
\alin{linGthm12}{
\norm{U(t)}_{ E^{s+1}(\Rbb^n)}& \leq  c(\alpha,\beta(0),\gamma)\biggl[ \norm{U(0)}_{E^{s+1}(\Rbb^n)}
+ \norm{F(0)}_{\Ec^{s-1}(\Rbb^n)} \\ & + \norm{H(0)}_{E^{s-1}(\Omega)}
+ \int_{0}^{T^*}\beta(\tau)\bigl( \norm{F(\tau)}_{\Ec^s(\Rbb^n)} + \norm{H(\tau)}_{E^s(\Omega)}\bigr)\, d\tau \biggr]
}
for $0\leq t < T^*$. Iterating this estimate a finite number of times shows that we can take $T^*=T$. Finally, we note that uniqueness follows directly
from Theorem \ref{weakthm}.
\end{proof}

\sect{exist}{Proof of Theorem \ref{mainthmA} }

We are now ready to prove Theorem \ref{mainthmA}.

\subsect{eandu}{Existence and uniqueness} We establish existence and uniqueness of solutions using the
well-known strategy, also used by Koch \cite{Koch:1993}, of
setting up an appropriate iterative approximation scheme and showing convergence by establishing boundedness in a high norm
followed by contraction in a low norm.

\subsubsect{bound}{Boundedness in the high norm}
To establish the boundedness in the high norm, we
begin by defining the set
\eqn{Bcdef}{
\Bc_R = \{ \, Z \in CX^{s+1}_T(\Rbb^3) \, |\, \del{t}^\ell Z(0)=\Ut_\ell \AND \norm{Z}_{X^{s+1}_T(\Rbb^3)} \leq R \, \},
}
where the $\Ut_\ell$ are as defined by the initial data \eqref{compat1} that satisfy the compatibility conditions \eqref{compat2}.
We consider the map
\eqn{JmapA}{
J_T : \Bc_R \longrightarrow CX^{s+1}_{T}(\Rbb^3)
}
defined by
\eqn{JmapB}{
J_T(U) = Z,
}
where $Z$ is the unique solution to the IVP:
\lalign{JmapC}{
\del{\mu}\bigl(A^{\mu \nu}(U) \del{\nu} Z\bigr)  &= F(U,\del{}U)
+ \chi_{\Omega} H(U, \del{} U) \quad \text{in $[0,T)\times \Rbb^n$}, \notag \\ 
(Z,\del{t}Z)|_{t=0} & = (\Ut_0,\Ut_1) \quad \text{in $\Rbb^n$} \notag . 
}

Next, we define
\eqn{murho}{
\mu = \norm{U(0)}_{E^{s+1}(\Rbb^n)}.
}
Then the bounds
\lalign{FHbnds}{
\norm{A(U)}_{\Xc^{s+1}_T(\Rbb^n)} &\leq C(R), \label{FHbnds.0}\\
\norm{F(U,\del{}U)}_{\Xc^s_T(\Rbb^n)} &\leq C(R), \label{FHbnds.1} \\
\norm{F(U(0),\del{}U(0))}_{\Ec^{s-1}(\Rbb^n)} &\leq C(\mu), \label{FHbnds.2} \\
\norm{H(U,\del{}U)}_{X^s_T(\Omega)} &\leq C(R), \label{FHbnds.3}
\intertext{and}
\norm{H(U(0),\del{}U(0))}_{E^{s-1}(\Omega)} &\leq C(\mu) \label{FHbnds.4}
}
follow directly from Proposition \ref{fpropB}. Writing $A(U)$ as
\alin{AbndA}{
A(U(t)) = A(U(0)) + \int_{0}^t DA(U(\tau))\cdot \del{t}U(\tau) \, d\tau,
}
we see with the help of Proposition \ref{fpropB} that
\leqn{AbndB}{
\norm{A(U)}_{\Xc^{s}_T(\Rbb^n)} + \norm{D[A(U)]}_{\Xc^{s-1}_T(\Rbb^n)} \leq C(\mu) + T C(R).
}
Theorem \ref{linGthm} in conjunction with the bounds \eqref{FHbnds.0}-\eqref{AbndB} then implies that $Z$
satisfies the estimate
\eqn{ZestA}{
\norm{Z}_{CX^{s+1}_T(\Rbb^n)} \leq c(\mu, TC(R)).
}
From this estimate, it is clear that we can arrange that
\eqn{ZestB}{
\norm{Z}_{CX^{s+1}_T(\Rbb^n)} < R
}
by choosing $R$ large enough and $T$ sufficiently small. This shows that $J_T$ satisfies
\leqn{JmapD}{
J_T\bigl(\Bc_R\bigr) \subset \Bc_R,
}
thereby establishing the boundedness in the high norm.

\subsubsect{contract}{Contraction in the low norm}

Choosing $U_0,U_1 \in \Bc_R$, we set
\eqn{contractA}{
Z_0 = J_T(U_0) \AND Z_1 = J_T(U_1).
}
Then  $Z_0-Z_1$ satisfies
\lalign{lnormA}{
\del{\mu}\bigl(A^{\mu\nu}(U_0)\del{\nu}(Z_0-Z_1)\bigr) &= F(U_0,\del{}U_0)-F(U_1,\del{}U_1) + \chi_\Omega\bigl( H(U_0,\del{}U_0)-H(U_1,\del{}U_1)\bigr) \notag \\
& \text{\hspace{2.0cm}} - \del{\mu}\bigl( \bigl[A^{\mu\nu}(U_0)-A^{\mu\nu}(U_1) \bigr]\del{\nu}Z_1\bigr)
\quad \text{in $[0,T)\times \Rbb^n$} ,  \label{lnormA.1} \\
\del{t}^\ell (Z_0-Z_1)|_{t=0} &  = 0 \quad \text{in $\Rbb^n$ for $\ell=0,1,\ldots,s$.} \label{lnormA.2}
}
Writing $F(U_0,\del{}U_0)-F(U_1,\del{}U_1)$ as
\eqn{contractB}{
F(U_0,\del{}U_0)-F(U_1,\del{}U_1) = \int_{0}^1 DF\bigl(U_1 + \tau (U_0-U_1),\del{} U_1 + \tau ( \del{}U_0-\del{}U_1)\bigr)\, d\tau \cdot (U_0-U_1,\del{} U_0-\del{} U_1),
}
we see that
\alin{contractC}{
&\norm{F(U_0,\del{}U_0)-F(U_1,\del{}U_1) }_{L^2(\Rbb^n)} \\
&\text{\hspace{1.0cm}} \leq \left\|\int_{0}^1 DF\bigl(U_1 + \tau (U_0-U_1),\del{} U_1 + \tau ( \del{}U_0-\del{}U_1)\bigr)\, d\tau \right\|_{L^\infty(\Rbb^n)}\norm{U_0-U_1}_{H^{1}(\Rbb^n)}  \\
&\text{\hspace{1.0cm}} \leq  \left\|\int_{0}^1 DF\bigl(U_1 + \tau (U_0-U_1),\del{} U_1 + \tau ( \del{}U_0-\del{}U_1)\bigr)\, d\tau \right\|_{\Hc^{0,s}(\Rbb^n)}\norm{U_0-U_1}_{H^{1}(\Rbb^n)}.
}
where in deriving the last line we have used Sobolev's inequality, see Theorem \ref{Sobolev}. Applying Proposition \ref{fpropB} to the above expression, we obtain the estimate
\leqn{contractD}{
\norm{F(U_0,\del{}U_0)-F(U_1,\del{}U_1) }_{L^2(\Rbb^n)} \leq C(R)\norm{U_0-U_1}_{E(\Rbb^n)}.
}
Similar calculations together with the bound \eqref{JmapD} also show that
\lalign{contractE}{
\norm{\chi_\Omega\bigl(H(U_0,\del{}U_0)-H(U_1,\del{}U_1) \bigr) }_{L^2(\Rbb^n)} &\leq C(R)\norm{U_0-U_1}_{E(\Rbb^n)}, \label{contractE.1} \\
\norm{\del{\mu}\bigl( \bigl[A^{\mu\nu}(U_0)-A^{\mu\nu}(U_1) \bigr]\del{\nu}Z_1\bigr)}_{L^2(\Rbb)} &\leq  C(R)\norm{U_0-U_1}_{E(\Rbb^n)} \label{contractE.2}
}
and
\leqn{contractF}{
\norm{A^{\mu\nu}(U_0)-A^{\mu\nu}(0)}_{L^\infty(\Rbb^n)} + \norm{\del{t}[A^{\mu\nu}(U_0)]}_{L^\infty(\Rbb^n)} \leq C(R).
}

Since $Z_0-Z_1$ satisfies the IVP \eqref{lnormA.1}-\eqref{lnormA.2}, we are in a position to apply  the energy estimates for
weak solutions of wave equation from Theorem \ref{weakthm} to conclude, with the help of the bounds \eqref{contractD}-\eqref{contractF}, that $Z_0-Z_1$
satisfies the estimate
\eqn{contractG}{
\norm{Z_0(t)-Z_1(t)}_{E(\Rbb^n)} \leq C(R)\int_{0}^T  \norm{Z_0(\tau)-Z_1(\tau)}_{E(\Rbb^n)} + \norm{U_0(\tau)-U_1(\tau)}_{E(\Rbb^n)}\, d\tau
}
for  $0\leq t < T$.
Appealing to Gronwall's inequality, we see that
\eqn{contractH}{
\norm{Z_0-Z_1}_{X_T^1(\Rbb^n)} \leq C(R)e^{C(R)T} T  \sup_{0\leq t < T}\norm{U_0-U_1}_{X_T^1(\Rbb^n)}.
}
Choosing $T>0$ small enough, we get that
\eqn{contractK}{
\norm{J_T(U_0)-J_T(U_1)}_{X_T^1(\Rbb^n)} \leq \Half  \norm{U_0-U_1}_{X_T^1(\Rbb^n)},
}
and so, $J_T$ defines a contraction map on the subset
\eqn{contractL}{
\Bc_R \subset CX^{1}_T(\Rbb^n).
}
In particular, for any $U_0\in \Bc_R$, the sequence
\eqn{contractM}{
U_n = \overset{\text{n times}}{\overbrace{J_T \circ \cdots \circ J_T}}(U_0) \quad n=1,2\ldots
}
converges to a unique fixed point $U \in CX^{1}_T(\Rbb^n)$ of $J_T$, that is
$J_T(U) = U$
or in other words, a weak solution of the IVP \eqref{waveA.1}-\eqref{waveA.2}. Since the sequence $U_n$ is bounded in $CX^{s+1}_{T}(\Rbb^n)$
by virtue of the mapping property \eqref{JmapD} of $J_T$, we have, after passing to a subsequence, that $U_n$ converges weakly in $X^{s+1}_T(\Rbb^n)$
to a limit that must coincide with $U$ by the uniqueness property of weak limits. Consequently, $U$ satisfies the additional regularity
$U \in X^{s+1}_T(\Rbb^n)$, which can be upgraded to $U \in CX^{s+1}_T(\Rbb^n)$ with the help of Theorem \ref{linGthm}.

\subsubsect{unique}{Uniqueness}
Before we establish uniqueness, we first note that the inclusion
\eqn{uniqueAa}{
CX_{T}^{s+1}(\Rbb^n) \subset CX_{T}^2(\Rbb^n) \cap \bigcap_{\ell=0}^1 C^\ell\bigl([0,T),W^{1-\ell,\infty}(\Rbb^n)\bigr)
}
follows directly from Sobolev's inequality since $s>n/2$ by assumption.
To prove uniqueness, we suppose that
\lgath{uniqueA}{
U_0 \in  CX_{T}^2(\Rbb^n) \cap \bigcap_{\ell=0}^1 C^\ell\bigl([0,T),W^{1-\ell,\infty}(\Rbb^n)\bigr)  
\intertext{and}
U_1  \in CX^{s+1}_T(\Rbb^n) 
}
are two solutions of the IVP \eqref{waveA.1}-\eqref{waveA.2}. Then the difference $U_0-U_1$ satisfies
\lalign{uniqueB}{
&\del{\mu}\bigl(A^{\mu\nu}(U_0)\del{\nu}(U_0-U_1)\bigr) = F(U_0,\del{}U_0)-F(U_1,\del{}U_1) + \chi_\Omega\bigl( H(U_0,\del{}U_0)-H(U_1,\del{}U_1)\bigr)\notag\\
& \text{\hspace{6.0cm}} - \del{\mu}\bigl( \bigl[A^{\mu\nu}(U_0)-A^{\mu\nu}(U_1) \bigr]\del{\nu}U_1\bigr)
\quad \text{in $[0,T)\times \Rbb^n$,}  \label{uniqueB.1} \\
&\bigl( (U_0-U_1), \del{t}(U_0-U_1)\bigr)|_{t=0}   = (0,0) \quad \text{in $\Rbb^n$.}  \label{uniqueB.2}
}
Using similar arguments as in the previous section, it is not difficult to derive the bounds
\lalign{uniqueC}{
\norm{F(U_0,\del{}U_0)-F(U_1,\del{}U_1) }_{L^2(\Rbb^n)} &\leq C(\rho_0,\rho_1)\norm{U_0-U_1}_{E(\Rbb^n)}, \label{uniqueC.1}\\
\norm{\chi_\Omega\bigl(H(U_0,\del{}U_0)-H(U_1,\del{}U_1) \bigr) }_{L^2(\Rbb^n)}&\leq C(\rho_0,\rho_1)\norm{U_0-U_1}_{E(\Rbb^n)} \label{uniqueC.2}
}
and
\leqn{uniqueD}{
\norm{A^{\mu\nu}(U_0)-A^{\mu\nu}(0)}_{L^\infty(\Rbb^n)} + \norm{\del{t}[A^{\mu\nu}(U_0)]}_{L^\infty(\Rbb^n)} \leq C(\rho_0),
}
where
\eqn{uniqueE}{
\rho_0 = \sup_{0\leq t < T}\Bigl[ \norm{U_0(t)}_{W^{1,\infty}(\Rbb^n)} + \norm{\del{t}U_0(t)}_{L^\infty(\Rbb^n)} \Bigr] \AND \rho_1 = \norm{U_1}_{X^{s+1}_T(\Rbb^n)}.
}
We also observe that
\lalign{uniqueF}{
\norm{\del{\mu}\bigl( \bigl[A^{\mu\nu}(U_0)-A^{\mu\nu}(U_1) \bigr]\del{\nu}U_1\bigr)}_{L^2(\Rbb^n)}  \leq &\norm{\del{\mu} \bigl[A^{\mu\nu}(U_0)-A^{\mu\nu}(U_1) \bigr]}_{L^2(\Rbb^n)}
\norm{\del{}U_1}_{L^\infty(\Rbb^n)}\notag \\
&+ \norm{A^{\mu\nu}(U_0)-A^{\mu\nu}(U_1)}_{L^\infty(\Rbb^n)}\norm{\del{}^2U_1}_{L^\infty(\Rbb^n)} \notag\\
\leq   \norm{\del{\mu} \bigl[A^{\mu\nu}(U_0)-A^{\mu\nu}(U_1) \bigr]}_{L^2(\Rbb^n)}&\rho_1 + \norm{A^{\mu\nu}(U_0)-A^{\mu\nu}(U_1)}_{L^{2n/(n-2)}(\Rbb^n)}\norm{\del{}^2U_1}_{L^n(\Rbb^n)} \notag \\
\leq  \norm{\del{\mu} \bigl[A^{\mu\nu}(U_0)-A^{\mu\nu}(U_1) \bigr]}_{L^2(\Rbb^n)}&\rho_1 + \norm{A^{\mu\nu}(U_0)-A^{\mu\nu}(U_1)}_{H^1(\Rbb^n)}\norm{\del{}^2U_1}_{L^n(\Rbb^n)}\label{uniqueF.1},
}
where in deriving the last inequality we used Sobolev's inequality. Again, using similar arguments as in the previous section, it is not difficult to verify that
\leqn{uniqueG}{
 \norm{\del{\mu} \bigl[A^{\mu\nu}(U_0)-A^{\mu\nu}(U_1) \bigr]}_{L^2(\Rbb^n)} + \norm{A^{\mu\nu}(U_0)-A^{\mu\nu}(U_1)}_{H^1(\Rbb^n)} \leq C(\rho_0,\rho_1)\norm{U_0-U_1}_{E(\Rbb^n)}.
}
Finally, we estimate
\lalign{uniqueH}{
\norm{\del{}^2U_1}_{L^n(\Rbb^n)} &\leq \max\{\norm{\del{}^2U_1}_{L^n(\Omega)},\norm{\del{}^2U_1}_{L^n(\Omega^c)}\} \notag \\
& \lesssim  \max\{\norm{\del{}^2U_1}_{H^{s-1}(\Omega)},\norm{\del{}^2U_1}_{H^{s-1}(\Omega^c)}\} \notag \\
& \lesssim \norm{\del{}^2 U_1}_{H^{0,s-1}(\Rbb^n)} \notag \\
& \lesssim  \sum_{\ell=0}^2 \norm{\del{t}^\ell U_1}_{H^{2,s+1-\ell}(\Rbb^n)}  \notag \\
& \lesssim \rho_1, \label{uniqueH.1}
}
where we have again used Sobolev's inequality and the assumption $s>n/2$.

Since $U_0-U_1$ satisfies the IVP \eqref{uniqueB.1}-\eqref{uniqueB.2}, we can apply the energy estimates for
weak solutions of linear wave equation from Theorem \ref{weakthm} to conclude, with the help of the bounds \eqref{uniqueC.1}-\eqref{uniqueH.1}, that $U_0-U_1$
satisfies the estimate
\eqn{uniqueI}{
\norm{U_0(t)-U_1(t)}_{E(\Rbb^n)} \leq C(\rho_0,\rho_1)\int_{0}^T \norm{U_0(\tau)-U_1(\tau)}_{E(\Rbb^n)}\, d\tau
}
for  $0\leq t < T$, which in turn, implies that
\eqn{uniqueJ}{
\norm{U_0(t)-U_1(t)}_{E(\Rbb^n)} = 0 \quad 0\leq t < T,
}
by Gronwall's inequality. We conclude that $U_0=U_1$ and uniqueness holds.

\subsect{cont}{The continuation principle}

In order to establish the continuation principle, we assume, by way of contradiction, that $U \in CX^{s+1}_T(\Rbb^n)$
$(s\in \Zbb_{> n/2})$ is a solution of the wave equation \eqref{waveA.1} satisfying
\leqn{contassumpA}{
\limsup_{t\nearrow T} \norm{U(t)}_{E^{s+1}(\Rbb^n)} = \infty
}
and
\leqn{contassumpB}{
\norm{U}_{W^{1,\infty}((0,T)\times\Rbb^n)} \leq K < \infty.
}
Using the property of finite speed of propagation, it is enough to show that $U$ cannot locally satisfy
both \eqref{contassumpA} and \eqref{contassumpB}, where we can, by suitably shifting the origin of the time coordinate, assume that
$T$ is small as we like. We note that away from the boundary $\del{}\Omega$ where
there are no singular terms in the wave equation \eqref{waveA.1}, we can appeal to the standard continuation
principle, see for example \cite[Theorem 2.2]{Majda:1984}, to conclude that, locally, the solution cannot satisfy both \eqref{contassumpA} and
\eqref{contassumpB}. In light of this observation, we need only worry about the behavior of $U$ in a neighborhood of the boundary
$\del{}\Omega$ for arbitrarily small times. Furthermore, since
\eqn{scaleinvar}{
\norm{u}_{W^{1,\infty}((0,T),Q_1)} \lesssim \norm{\Ub}_{W^{1,\infty}((0,\delta T)\times Q_\delta)} \lesssim K
}
for all $ \delta \in (0,1]$ where $u$ and $\Ub$ are as defined previously by \eqref{epscale.1} and \eqref{barvars.1}, respectively,
it is enough to consider the
solution $U$ on an arbitrary small spacetime neighborhoods $(x_0,0)$ with $x_0 \in \del{}\Omega$ that satisfy the bound
\eqref{contassumpB}. We can, therefore,
use the scaling and projection technique from Sections \ref{linit} and \ref{linrs} to reduce the continuation
question to that of proving a continuation principle for the following scaled and projected system where we may choose
$\delta$ as small as we like:
\lalign{contnev}{
\del{\nu}\bigl((m^{\mu\nu} + \delta b^{\mu\nu}(\xv,u))\del{\mu} u\bigr) - \psi u &=
\delta(f(\xv,u\del{}u)+ \chi_{\Om}h(\xv,u\del{}u)) + \mu \quad \text{in $[0,T)\times \Tbb^n$,}\label{contev.1} \\
(u,\del{t}u)|_{t=0} &= (\hat{u}_0,\hat{u}_1) \quad \text{in $\Tbb^n$,} \label{contev.2}
}
where
\begin{itemize}
\item[(i)]
\alin{contscaleA}{
\hat{u}_0(x) &= \phi_1(x)\frac{U(\psi_{x_0,\delta}(0,\delta x))-U(\psi_{x_0,\delta}(\zero))}{\delta},\\
\hat{u}_1(x) &=  \phi_1(x) \del{t}U(0,\Psi_{x_0,\delta}(\delta x)),\\
(m^{\mu\nu}) &= \bigl(\det(J(\zero))\Jch^\mu_\alpha(\zero)\Jch^\nu_\beta(\zero)A^{\alpha\beta}(U(\psi_{x_0,\delta}(\zero))) \bigr) =
\text{diag}(-1,1,\ldots,1), \\
b^{\mu\nu}(\xv,u) &= \phi_1(x)\frac{\det(J(\delta\xv))\Jch^\mu_\alpha(\delta\xv)\Jch^\nu_\beta(\delta\xv)A^{\alpha\beta}\bigl(U(\psi_{x_0,\delta}(\zero))+\delta u)-m^{\mu\nu}}{\delta}, \\
f(\xv,u,\del{}u) &= \phi_1(x) \det(J(\delta\xv))F\bigl(U(\psi_{x_0,\delta}(\zero))+\delta u,\Jch(\delta\xv)\del{}u\bigr),
\intertext{and}
H(\xv,u,\del{}u) &= \phi_1(x) \det(J(\delta\xv))H\bigl(U(\psi_{x_0,\delta}(\zero))+\delta u,\Jch(\delta\xv)\del{}u\bigr),
}
\item[(ii)] $0<\delta \leq 1$, 
\item[(iii)]
the initial data $(\hat{u}_0,\hat{u}_1)$ satisfies
\alin{contscaleB}{
\norm{\hat{u}_0}_{\Hc^{s+1,2}(\Tbb^n)} \lesssim \norm{U(0)}_{\Hc^{s+1,2}(\Rbb^n)}
\intertext{and}
\norm{\hat{u}_0}_{\Hc^{s,2}(\Tbb^n)} \lesssim \norm{\del{t}U(0)}_{\Hc^{s,2}(\Rbb^n)},
}
\item[(iv)] and
\eqn{contscaleC}{
\mu = \sum_{\ell=0}^{s-1} \frac{t^\ell}{\ell !} \mu_\ell
}
with the $\mu_\ell$ determined as in Proposition \ref{muprop} so that
\alin{contscaleD}{
\norm{u(0)}_{E^{s+1}} &\lesssim \norm{U(0)}_{E^{s+1}(\Rbb^n)}
\intertext{and}
\norm{\mu(t)}_{\Ec^{s-1}(\Tbb^n)} & \leq (1+t^{s-1})C\bigl(\norm{U(0)}_{E^{s+1}(\Rbb^n)}\bigr).
}
\end{itemize}

In light of the above discussion, the following proposition completes the proof of Theorem \ref{mainthmA}.
\begin{prop} \label{contprop}
Suppose $u_\delta \in CX^{s+1}_{T_{\delta}(\Rbb^n)}$ is a family of solutions depending on $\delta \in (0,1]$ to the IVP \eqref{contev.1}-\eqref{contev.2} satisfying
the conditions (i)-(iv) above. If
\eqn{cotprop1}{
\norm{u_\delta}_{W^{1,\infty}((0,T_\delta)\times\Tbb^n)} \leq K < \infty,
}
then there exists a $\delta_0>0$ and a time $T^*_\delta > T_\delta$ for each $\delta \in (0,\delta_0]$ such that the solution $u_\delta$
can be continued to a solution of \eqref{contev.1}-\eqref{contev.2} on $[0,T^*_\delta)\times \Tbb^n$.
\end{prop}
\begin{proof}
Before proceeding with the proof, we will, in order to simplify calculations, suppress the explicit $\xv$-dependence of the functions $b^{\mu\nu}$,
$f$ and $h$.  Since $u_\delta$ satisfies
\leqn{contprop2}{
\del{\nu}\bigl((m^{\mu\nu} + \delta b^{\mu\nu}(u_\delta))\del{\mu} u_\delta\bigr) - \psi u_\delta =
\delta(f(u_\delta\del{}u_\delta)+ \chi_{\Om}h(u_\delta\del{}u_\delta)) \quad \text{in $[0,T_\delta)\times \Tbb^n$,}
}
we see after differentiating $k$-times, where $0\leq k \leq s-1$, with respect to $t$ that $\del{t}^k u_\delta$ satisfies
\lalign{contprop3}{
(\Delta -\psi)\del{t}^k u_\delta = \del{t}^{k+2} u_\delta +\delta\bigl[-\del{t}^{k+1}&\bigl(b^{00}(u_\delta)\del{t}u_\delta
+ b^{0i}(u_\delta)\del{i}u_\delta \bigr)
-\del{i} \del{t}^k\bigl( b^{i0}(u_\delta)\del{t}u_\delta  \notag\\
&+
b^{ij}(u_\delta)\del{j}u_\delta \bigr) + \del{t}^kf(u_\delta,\del{}u_\delta) + \chi_\Om \del{t}^k h(u,\del{}u) \bigr] + \del{t}^k \mu.
\label{contprop3.1}
}
Since
\eqn{contprop3a}{
\norm{\del{t}^k u_\delta}_{\Hc^{2,s+1-k}(\Tbb^n)} \lesssim \norm{(\Delta-\psi)\del{t}^k u_\delta}_{\Hc^{0,s-k-1}(\Tbb^n)} \quad 0\leq k \leq s-1
}
by Proposition \ref{potprop}, it follows from Proposition \ref{STpropC} and \eqref{contprop3.1} that $\del{t}^k u_\delta$
satisfies
\lalign{contprop4}{
\norm{\del{t}^k u(t)}_{\Hc^{2,s+1-k}(\Tbb^n)} \leq c\bigl(&\norm{\del{t}^{k+2}u}_{\Hc^{0,s-k-1}(\Tbb^n)}+
\norm{\del{t}^k \mu(t)}_{\Hc^{0,s-k-1}(\Tbb^n)} \bigr)  \notag \\
& + \delta C(K)
\bigl(1+ \norm{u(t)}_{\Hc^{2,s+1}(\Tbb^n)} + \norm{\del{t}u(t)}_{\Hc^{2,s}(\Tbb^n)} \bigr)\quad 0\leq t < T_\delta
\label{contprop4.1}
}
where
\leqn{contprop5}{
\norm{u_\delta}_{W^{1,\infty}((0,T_\delta)\times\Tbb^n)} \leq K \qquad 0 < \delta \leq 1.
}
We collect the estimates \eqref{contprop4.1}, for $0\leq k \leq s-1$, into the single matrix inequality
\leqn{contprop6}{
|M_\delta X_\delta(t)| \lesssim \norm{\del{t}^s u_\delta(t)}_{E(\Tbb^n)} + \norm{\mu(t)}_{\Ec^{s-1}(\Tbb^n)}  + \delta C(K)
\quad 0\leq t < T_\delta,
}
where
\eqn{contprop7}{
M_\delta = \begin{pmatrix}1-\delta C(K) & -\delta C(K)  &    0 &   c &    0      & 0 & & \cdots     &  &0 \\
                -\delta C(K)            & 1-\delta C(K) &    0 &   0 &    c      & 0 & &           &  &0 \\
                -\delta C(K)            & -\delta C(K)  &    1 &   0 &    0      & c&  &           &  &  \\
                -\delta C(K)            & -\delta C(K)  &    0 &   1 &    0      & 0&  & \ddots    &  &0 \\
                                        &    \vdots     &      &     &   \ddots  &  &  &           &  &c  \\
                                        &               &      &     &           &  &  &           &  &0 \\
                                        &               &      &     &           &  &  &           &  & \\
                -\delta C(K)            & -\delta C(K)  &    0 &  0  &   0       &  0 &  &    \cdots &  &1
                \end{pmatrix},
}
and
\eqn{contprop8}{
X_\delta(t) = \bigl(\norm{u_\delta(t)}_{\Hc^{2,s+1}(\Tbb^n)}, \norm{\del{t}u_\delta(t)}_{\Hc^{2,s}(\Tbb^n)}, \ldots ,
\norm{\del{t}^{s-1}u_\delta(t)}_{\Hc^{2,2}(\Tbb^n)} \bigr)^T.
}
Since $M_0$ is tri-diagonal, it follows that $M_0$ is invertible, and hence, that there exists a $\delta_0 \in (0,1]$ such
that $M_\delta$ is invertible with a uniformly bounded inverse for all $\delta \in (0,\delta_0]$. This fact together with \eqref{contprop6}
shows that $X_\delta(t)$ satisfies
\leqn{contprop9}{
|X_\delta(t)| \lesssim  \norm{\del{t}^s u_\delta(t)}_{E(\Tbb^n)} + \norm{\mu(t)}_{\Ec^{s-1}(\Tbb^n)}  + \delta C(K)
}
for all $(t,\delta) \in [0,T_\delta)\times (1,\delta_0]$.

Next, differentiating \eqref{contprop2} $s$-times with respect to $t$, we see that $\del{t}^s u_\delta$ is a weak
solution of
\eqn{contprop10a}{
\del{\nu}\bigl((m^{\mu\nu} + \delta b^{\mu\nu}(u_\delta))\del{\mu} \del{t}^s u_\delta\bigr) - \psi \del{t}^s u_\delta =
-\delta \del{\mu}\bigl([\del{t}^s,b^{\mu\nu}(u_\delta)\del{\nu}]u_\delta\bigr) +
\delta\bigl( \del{t}^s f(u_\delta\del{}u_\delta)+ \chi_{\Om} \del{t}^s h(u_\delta\del{}u_\delta)\bigr)
}
in $[0,T_\delta)\times \Tbb^n$. Applying the estimates from Propositions \ref{STpropA} and \ref{STpropB}, and Theorem \ref{weakthm}, we
obtain, with the help of the bound \eqref{contprop5}, the following energy estimate for $\del{t}^s u_\delta$:
\leqn{contprop11}{
\norm{\del{t}^s u_\delta(t)}_{E(\Tbb^n)} \leq C(K)\left(\norm{\del{t}^s u_\delta(0)}_{E(\Tbb^n)} + \int_{0}^t \norm{u_\delta(\tau)}_{E^{s+1}(\Tbb^n)} + 1 + \norm{\mu(\tau)}_{\Ec^{s-1}(\Tbb^n)} \, d\tau\right)
}
for $t\in (0,T_\delta)$. By assumption, $u_\delta$ and $\mu$ satisfy the bounds
\leqn{contprop10}{
\norm{u_\delta(0)}_{E^{s+1}(\Tbb^n)} \lesssim \norm{U(0)}_{E^{s+1}(\Rbb^n)} \AND \norm{\mu(t)}_{\Ec^{s-1}} \leq (1+t^{s-1})C\bigl(\norm{U(0)}_{E^{s+1}(\Rbb^n)}\bigr)
}
for $\delta \in (0,1]$. Combining these bounds with the estimates \eqref{contprop9} and \eqref{contprop11}, we get that
\eqn{contprop12}{
\norm{u_\delta(t)}_{E^{s+1}(\Tbb^n)} \leq C\bigl(K,T_\delta,\norm{U(0)}_{E^{s+1}(\Rbb^n)}\bigr)\left(1 + \int_{0}^t \norm{u_\delta(\tau)}_{E^{s+1}(\Tbb^n)}\, d\tau\right)
}
for $t\in (0,T_\delta)$, and hence, by Gronwall's inequality, that
\eqn{contprop13}{
\norm{u_\delta(t)}_{E^{s+1}(\Tbb^n)} \leq C\bigl(K,T_\delta,\norm{U(0)}_{E^{s+1}(\Rbb^n)}\bigr) \qquad 0\leq t < T_\delta.
}
In particular, this implies that
\eqn{contprop14}{
\limsup_{t\nearrow T_\delta}\norm{u_\delta(t)}_{E^{s+1}(\Tbb^n)} < \infty
}
for each $\delta \in (0,\delta_0]$.
From this point, we can follow standard arguments, for example, see the proof of Theorem 2.2, p. 46 of \cite{Majda:1984} , to conclude that for each $\delta \in (0,\delta_0]$, there exists
a $T^*_\delta > 0$ such that the solution $u_\delta$ extends to a solution on $[0,T^*_\delta)\times \Tbb^n$.
\end{proof}
\sect{disc}{Discussion and outlook}

As discussed in the introduction, the main application that we have in mind for the results presented in this article is to establish
the local existence and uniqueness of solutions to the Einstein equations coupled to elastic matter that describe the motion
of self-gravitating compact elastic bodies. While the complete details of the local existence and uniqueness proof
will be presented in a separate article \cite{Andersson_et_al:2013}, we give here the main ideas of the proof in order
to illustrate the role that the results of this article play in the proof.

Following \cite{BeigSchmidt:2003}, a single compact relativistic elastic body\footnote{The extension to non-colliding multiple interacting bodies
is straightforward.}, locally in time, is characterized by a map
\eqn{fmap}{
f \: : \: W \longrightarrow \Omega
}
from a space-time cylinder $W \cong [0,T]\times \Omega$ to a 3-dimensional
compact manifold $\Omega$ with boundary, known as the \emph{material manifold}. The body world tube
$W$ is taken  to be contained in an ambient Lorentzian spacetime $(M,g)$, where $M\cong [0,T]\times \Sigma$ for
some 3-manifold $\Sigma$. For simplicity of presentation, we assume that both $\Omega$ and $M$ can each be covered
by a single coordinate chart given by $(X^I)$ $(I=1,2,3)$ and $(x^\lambda)$ $(\lambda=0,1,2,3)$, respectively. In these local coordinates, we
can express $f$ and $g$ as
\eqn{fg}{
X^I=f^I(x^\lambda) \AND g=g_{\mu\nu}(x^\lambda)dx^\mu dx^\nu.
}
The field equations satisfied by $\{f^I,g_{\mu\nu}\}$ are then given
by
\lalign{EinElasA}{
G^{\mu\nu} &= 2\kappa T^{\mu\nu} \text{\hspace{0.3cm} in $M$,} \label{EinElasA.1}  \\
\nabla_{\mu} T^{\mu\nu} &= 0 \text{\hspace{1.1cm} in $W$,} \label{EinElasA.2}
}
where $G^{\mu\nu}$ is the Einstein tensor and
\eqn{Tdef}{
T_{\mu\nu} = 2\frac{\del{}\rho}{\del{} g^{\mu\nu}} - \rho g_{\mu\nu}
}
is the stress-energy of the elastic body with
\eqn{rhodef}{
\rho = \rho(f,H) \qquad (H^{IJ} := g^{\mu\nu}\del{\mu}f^I\del{\nu}f^J )
}
defining the proper energy density of the elastic body. By definition, $\rho$ is non-zero inside $W$ and vanishes outside. Letting $\Gamma$ denote the space-like boundary of $W$, the elastic field must also satisfy the boundary conditions
\leqn{EinElasB}{
n^\mu T_{\mu\nu} = 0 \text{\hspace{0.3cm} in $\Gamma$,}
}
where here $n^\mu$ denotes the outward pointing unit normal to $\Gamma$. Initial conditions for \eqref{EinElasA.1}-\eqref{EinElasA.2} are given
by
\lalign{EinElasC}{
(g_{\mu\nu},\Lc_t g_{\mu\nu}) &= (g^0_{\mu\nu},g^1_{\mu\nu}) \text{\hspace{0.3cm} in $\Sigma$,} \label{EinElasC.1}  \\
(f^I,\Lc_t f^I) &= (f_0^I,f_1^I) \text{\hspace{0.6cm} in $\Sigma\cap W$,} \label{EinElasC.2}
}
where $\Sigma$ forms the ``bottom'' of the spacetime slab $M\cong [0,T]\times \Sigma$,
$\Sigma\cap W$ forms the bottom of the spacetime cylinder $W\cong [0,T]\times \Omega$, $t=t^\mu\del{\mu}$ is a future pointing
time-like vector field tangent to $\Gamma$, and the initial data satisfies the \emph{constraint equations}
\leqn{EinElasD}{
t_{\mu} G^{\mu\nu} = 2\kappa t_{\mu} T^{\mu\nu} \text{\hspace{0.3cm} in $\Sigma$.}
}

The method we use to solve the initial value boundary problem (IVBP), given by \eqref{EinElasA.1}-\eqref{EinElasD}, begins with introducing harmonic coordinates as this allows us to
replace the full Einstein equations \eqref{EinElasA.1} with the \emph{reduced} equations
given by
\leqn{EinElasE}{
R_{\mu\nu} - \nabla_{(\mu}\xi_{\nu)} = 2\kappa\bigl( T_{\mu\nu}-\Third T g_{\mu\nu}\bigr) \qquad (\xi^\gamma := g^{\mu\nu}\Gamma_{\mu\nu}^\gamma).
}
For this method to work, we must choose the initial data so that the constraint
\leqn{xidata}{
\xi^\mu = 0 \quad \text{in $\Sigma$}
}
is also satisfied in addition to \eqref{EinElasD}.

The next step is to introduce the material representation via the map
\eqn{phiA}{
x^i = \phi^i(X^0,X^I) \qquad (i=1,2,3),
}
which is uniquely determined by the requirement
\eqn{phiB}{
f^I(X^0,\phi(X^0,X^I)) = X^I \quad \forall \; (X^0,X^I)\in [0,T]\times \Omega.
}
In the material representation, the elastic field is completely characterized by the map $\phi$ while the gravitational field is determined by
the components of the metric expressed in the material representation as follows
\eqn{gammadef}{
\gamma_{\mu\nu}(X^0,X^I) = g_{\mu\nu}(X^0,\phi(X^0,X^I)).
}
A straightforward calculation then shows that the field equations \eqref{EinElasA.1}-\eqref{EinElasA.2}, the boundary
conditions \eqref{EinElasB} and the initial conditions \eqref{EinElasC.1}-\eqref{EinElasC.2}, when expressed
in terms of the variables $\{\gamma_{\mu\nu},\phi^i\}$, take the form
\lalign{EinElasF}{
\frac{\del{}\;\;}{\del{}X^\Delta}\left(a^{\Delta\Lambda}\bigl(\gamma,\del{}\phit\bigr)\frac{\del{}\gamma_{\mu\nu}}{\del{}X^\Lambda}\right)& =
 q_{\mu\nu}\bigl(\gamma,\del{}\gamma,\del{}\phit\bigr) + \chi_{\Omega}p_{\mu\nu}\bigl(\gamma,\phi,\del{}\phi\bigr)
 &&\text{in $[0,T]\times \tilde{\Sigma} $,} \label{EinElasF.1} \\
\frac{\del{}\;\;}{\del{}X^\Delta}\left(F^{\Delta}_i \bigl(\gamma,\phi,\del{}\phi\bigr)\right)
&= w_i(\gamma,\phi,\del{}\phi)  &&\text{in $[0,T]\times \Omega $,} \label{EinElasF.2} \\
\nu_J F^{J}_i \bigl(\gamma,\phi,\del{}\phi\bigr)
&= 0  &&\text{in $[0,T]\times \del{}\Omega $,} \label{EinElasF.3} \\
\left(\gamma_{\mu\nu}, \frac{\del{}\gamma_{\mu\nu}}{\del{}X^0}  \right) &= (\gamma^0_{\mu\nu},\gamma^1_{\mu\nu}) &&\text{in
$\{0\}\times \tilde{\Sigma}$,} \label{EinElasF.4}  \\
\left(\phi^i, \frac{\del{}\phi^i}{\del{}X^0}\right) &= (\phi_0^i,\phi_1^i) &&\text{in $\{0\} \times \Omega$,} \label{EinElasF.5}
}
where $\Omega \subset \tilde{\Sigma}$ with $\tilde{\Sigma}$ defined by
$\phit(0,\tilde{\Sigma}) = \Sigma$, $\phit=E(\phi)$ with $E$ a suitable extension operator from $\Omega$ to $\tilde{\Sigma}$, $\nu_J$ is the outward pointing unit normal to $\del{}\Omega$
and
\eqn{EEdeldef}{
\del{}(\cdot) = \frac{\del{}(\cdot)}{\del{}X^\Delta} \qquad (\Delta=0,1,2,3)
}
is the spacetime gradient. From the point of view of local existence, we lose
nothing by assuming that $\tilde{\Sigma} \cong \Rbb^3$ and the $(X^I)$ are Cartesian coordinates on $\tilde{\Sigma}$.

\begin{rem} \label{discremA}
The dependence of the coefficients $a^{\Delta\Lambda}$ in
\eqref{EinElasF.1} on $\del{}\phi$ is problematic from a regularity perspective for the hyperbolic estimate of the top time derivative $\left(\frac{\del{}\;\;}{\del{}X^0}\right)^s \!\gamma_{\mu\nu}$ from
the proof of Theorem \ref{linthm}. This is because if we were to estimate the top $X^0$-derivative using the wave equation
\eqref{EinElasF.1} as in the proof of Theorem \ref{linthm}, we would require an estimate on the $(s+1)^{\text{th}}$ $X^0$-derivative of $\del{}\phi$,
 and this is one too many derivatives to be compatible with Koch's \cite{Koch:1993} estimates for equation \eqref{EinElasF.2}. To avoid this
 loss of derivatives scenario, we
instead use a first order formulation of the gravitational field equations based on
the variables $\{\gamma,\lambda\}$, where
\eqn{discremA1}{
\lambda_{\sigma\mu\nu} := (\del{\sigma}g_{\mu\nu})(X^0,\phi(X^0,X^I)),
}
to estimate the top $X^0$-derivative. We note that the lower $X^0$-derivatives are still estimated using elliptic estimates based
on the wave formulation \eqref{EinElasF.1} as in the proof of Theorem \ref{linthm}.

A straightforward calculation shows that the reduced equations \eqref{EinElasE} can be expressed
in terms of the  $\{\gamma_{\mu\nu},\lambda_{\sigma\mu\nu}\}$ variables as a symmetric hyperbolic system of the form
\eqn{discremA2}{
b^{\alpha\beta\kappa}(\gamma,\del{}\phi)\frac{\del{}\lambda_{\beta\mu\nu}}{\del{}X^\kappa} =
f(\gamma,\lambda,\del{}\phi) + \chi_{\Omega}h(\gamma,\phi,\del{}\phi),
}
with the point being that, unlike \eqref{EinElasF.1}, after differentiating this equation $s$-times with respect to $X^0$, we
obtain an $L^2$ estimate for $\left(\frac{\del{}\;\;}{\del{}X^0}\right)^s\! \lambda_{\mu\nu}$ with the highest $X^0$-derivative of $\del{}\phi$
appearing in the estimate being the
$s^{\text{th}}$ one. Importantly, this $L^2$ estimate
is, with the help
the estimates on $\phi$ coming from \eqref{EinElasF.2}, enough to obtain an appropriate $L^2$ estimate for $\left(\frac{\del{}\;\;}{\del{}X^0}\right)^s\! \del{}\gamma_{\mu\nu}$ thereby
avoiding any loss of derivatives.
\end{rem}

To proceed, we assume that the initial data \eqref{EinElasC.1}-\eqref{EinElasC.2} satisfy the constraint equations
\eqref{EinElasD}, \eqref{xidata} and also the \emph{compatibility conditions}
\lalign{EinElasG}{
\gamma^\ell_{\mu\nu} := \left(\frac{\del{}\;\;}{\del{}X^0}\right)^\ell\Bigl|_{X^0=0} \! \gamma_{\mu\nu} &\in \Hc^{m_{s+1-\ell},s+1-\ell}(\tilde{\Sigma})
\quad \ell=0,1,\ldots,s+1, \label{EinElasG.1} \\
\phi_\ell^i := \left(\frac{\del{}\;\;}{\del{}X^0}\right)^\ell\Bigl|_{X^0=0} \! \phi^i & \in H^{s+1-\ell}(\Omega) \quad \ell=0,1,\ldots,s+1
\label{EinElasG.2}
\intertext{and}
\left(\frac{\del{}\;\;}{\del{}X^0}\right)^\ell \!\bigl(\nu_J F^J_i(\gamma,\del{}\phi)\bigr)\Bigr|_{X^0=0} &\in H^{s-\ell}(\Omega)\cap H^1_0(\Omega),
  \label{EinElasG.3}}
where $s \in \Zbb_{> 5/2}$,
\eqn{mEEdef}{
m_j = \begin{cases} 2 & \text{if $j\geq 2$} \\
j & \text{otherwise} \end{cases}
}
and
\eqn{Hcdef}{
\Hc^{k,r}(\tilde{\Sigma}) = H^{r}(\Omega)\cap H^{k}(\tilde{\Sigma})\cap H^{r}(\tilde{\Sigma}\setminus \Omega)
}
We know from the results of \cite{Andersson_et_al:2013} that the set of initial data satisfying
the constraint equations and the compatibility conditions is non-empty. However, a complete classification of the space of initial
data satisfying these conditions appears to be very difficult, and in fact, the classification of the space of
initial data satisfying just the constraint equations \eqref{EinElasD} is far from complete.

Rather than solving the elastic boundary value problem \eqref{EinElasF.2}-\eqref{EinElasF.3} directly, we follow
\cite{Koch:1993} and differentiate it once with respect to $X^0$ to obtain the system
\lalign{EinElasH}{
\frac{\del{}\;\;}{\del{}X^\Delta}\left(L^{\Delta\Lambda}_{ij} \bigl(\gamma,\del{}\phi\bigr)\frac{\del{}\psi^j}{\del{}X^\Lambda}
+ Z^{\Delta\mu\nu}_i(\gamma,\del{}\phi)\frac{\del{}\gamma_{\mu\nu}}{\del{}X^0}\right)
&= Y_i\left(\gamma,\frac{\del{}\gamma_{\mu\nu}}{\del{}X^0},\phi,\del{}\phi,\del{}\psi\right), && \text{in $[0,T]\times \Omega $,} \label{EinElasH.1} \\
\del{0}\phi^i &= \psi^i && \text{in $[0,T]\times \Omega $,} \label{EinElasH.2} \\
\nu_J \left(L^{J\Lambda}_{ij} \bigl(\gamma,\del{}\phi\bigr)\frac{\del{}\psi^j}{\del{}X^\Lambda}
+ Z^{J\mu\nu}_i(\gamma,\del{}\phi)\frac{\del{}\gamma_{\mu\nu}}{\del{}X^0}\right)
&= 0 && \text{ in $[0,T]\times \del{}\Omega $,} \label{EinElasH.3}
}
where
\leqn{Ldef}{
L^{\Delta\Lambda}_{ij}(\gamma,\del{}\phi) = \frac{\del{}F^\Delta_i}{\del{}\frac{\del{}\phi^j}{\del{}X^\Lambda}}(\gamma,\del{}\phi)
}
is the elasticity tensor, as expressed in the material frame. In particular, we restrict ourself to elastic materials
for which the elasticity tensor \eqref{Ldef} satisfies Koch's coercivity condition\footnote{We note that this condition rules out perfect fluids.} \cite[Assumption 3, p. 12]{Koch:1993}. We note that Koch's other assumptions,
Assumptions 1,2 and 4 in \cite[pp. 12-13]{Koch:1993}, are satisfied automatically by the elasticity tensor of reasonable relativistic materials.

In order to solve the IVBP defined by \eqref{EinElasF.1}, \eqref{EinElasF.4}-\eqref{EinElasF.5} and \eqref{EinElasH.1}-\eqref{EinElasH.3},
we employ an iteration scheme, analogous to the one used in Section \ref{eandu}, defined by the map
\leqn{EEJmapA}{
J_T(\mu_{\mu\nu},\alpha^i,\beta^i) = (\gamma_{\mu\nu},\phi^i,\psi^i),
}
which maps the triple
\eqn{EEJmapB}{
(\mu_{\mu\nu},\alpha^i,\beta^i) \in \Bc_R,
}
where
\alin{EEBcdef}{
\Bc_R:= &\Bigl\{(\gamma,\phi,\psi)\in CX^{s+1}_T(\Rbb^3)\times CY^{s+1}_T(\Omega) \times CY^{s}_T(\Omega) \: \Bigl| \:
\norm{(\gamma,\phi,\psi)}\leq R, \\
 &(\del{X^0}^\ell\gamma,\del{X^0}^\ell \phi)|_{X^0=0}=(\gamma^\ell,\phi_\ell) \; \; \ell=0,1,\ldots,s+1 \quad \& \quad
 \del{X^0}^\ell \psi|_{X^0=0} = \phi_{\ell+1} \; \; \ell=0,1,\ldots s \Bigr\}
}
to a solution
\eqn{EEJmapC}{
(\gamma_{\mu\nu},\phi^i,\psi^i) \in CX^{s+1}_T(\Rbb^3)\times CY^{s+1}_T(\Omega) \times CY^{s}_T(\Omega)
}
of the IVBP
\lalign{EinElasK}{
\frac{\del{}\;\;}{\del{}X^\Delta}\left(a^{\Delta\Lambda}\bigl(\lambda,\betat,D\alphat\bigr)\frac{\del{}\gamma_{\mu\nu}}{\del{}X^\Lambda}\right)
- q_{\mu\nu}\bigl(\lambda,\del{}\lambda,\betat,D\alphat\bigr)  \hspace{1.5cm} & \notag \\
  - \chi_{\Omega}p_{\mu\nu}\bigl(\lambda,\beta,D\alpha\bigr) &= 0
 &&\text{in $[0,T]\times \tilde{\Sigma} $,} \label{EinElasK.1} \\
\frac{\del{}\;\;}{\del{}X^\Delta}\left(L^{\Delta\Lambda}_{ij} \bigl(\lambda,\beta,D\alpha\bigr)\frac{\del{}\psi^j}{\del{}X^\Lambda}
+ Z^{\Delta\mu\nu}_i(\lambda,\beta,D\alpha)\frac{\del{}\lambda_{\mu\nu}}{\del{}X^0}\right)  \hspace{1.5cm} & \notag \\
 -Y_i\left(\gamma,\frac{\del{}\gamma_{\mu\nu}}{\del{}X^0},\phi,\del{}\phi,\del{}\psi\right)
&= 0  &&\text{in $[0,T]\times \Omega $,} \label{EinElasK.2}\\
\del{0}\phi^i &= \beta^i  &&\text{in $[0,T]\times \Omega $,} \label{EinElasK.3} \\
\nu_J \left(L^{J\Lambda}_{ij} \bigl(\lambda,\beta,D\alpha\bigr)\frac{\del{}\psi^j}{\del{}X^\Lambda}
+ Z^{J\mu\nu}_i(\lambda,\beta,D\alpha)\frac{\del{}\lambda_{\mu\nu}}{\del{}X^0}\right)
&= 0 &&\text{in $[0,T]\times \del{}\Omega $,} \label{EinElasK.4} \\
\left(\gamma_{\mu\nu}, \frac{\del{}\gamma_{\mu\nu}}{\del{}X^0}  \right) &= (\gamma^0_{\mu\nu},\gamma^1_{\mu\nu}) &&\text{in
$\{0\}\times \tilde{\Sigma}$,} \label{EinElasK.5}  \\
\phi^i &= \phi_0^i &&\text{in $\{0\} \times \Omega$,} \label{EinElasF.6}\\
\left(\psi^i, \frac{\del{}\psi^i}{\del{}X^0}\right) &= (\phi_1^i,\phi_2^i) &&\text{in $\{0\} \times \Omega$,} \label{EinElasF.7}
}
where we are using
\eqn{EDdef}{
D(\cdot) = \frac{\del{}(\cdot)}{\del{}X^I}
}
to denote the spatial gradient.

The mapping property
\eqn{JTMp}{
J_T\: :\: \Bc_R \longrightarrow CX^{s+1}_T(\Rbb^3)\times CY^{s+1}_T(\Omega) \times CY^{s-1}_T(\Omega)
}
is a consequence of the linear estimates contained in Theorem \ref{linGthm} of this article and Theorem 2.4 of  \cite{Koch:1993}.
Furthermore, it follows from these estimates and the calculus inequalities of this article and those of \cite{Koch:1993} that
$J_T$ satisfies
\eqn{JTBR}{
J_T(\Bc_R) \subset \Bc_R
}
for $T>0$ small enough. This establishes boundedness in a high norm. Mimicking the arguments used in Section \ref{contract}
of this article and those in Section 3 of \cite{Koch:1993}, it can be shown that $J_T$, shrinking $T>0$ if necessary, defines a contraction in a suitable
low norm, and this, in turn, yields the existence of a unique solution
\eqn{EEsolAa}{
(\gamma_{\mu\nu},\phi^i,\psi^i) \in CX^{s+1}_T(\Rbb^3)\times CY^{s+1}_T(\Omega) \times CY^{s}_T(\Omega)
}
of the IVBP  \eqref{EinElasF.1}, \eqref{EinElasF.4}-\eqref{EinElasF.5} and \eqref{EinElasH.1}-\eqref{EinElasH.3}. It is
then a simple consequence of the above definitions that the pair
\eqn{EEsolA}{
(\gamma_{\mu\nu},\phi^i) \in CX^{s+1}_T(\Rbb^3)\times CY^{s+1}_T(\Omega)
}
is the unique solution to the IVBP \eqref{EinElasF.1}-\eqref{EinElasF.5}. Inverting the transformation used to define the material
representation, it is not difficult to verify that this solution yields a (unique) solution $(g_{\mu\nu},f^I)$ to the reduced IVBP
\eqref{EinElasA.2}-\eqref{EinElasE}.

The final step is to show that the vector field $\xi^\mu$ vanishes so that the solution $(g_{\mu\nu},f^I)$ also satisfies the full
Einstein equations \eqref{EinElasA.1}. This is accomplished by realizing that the boundary condition \eqref{EinElasB}
together with the elasticity field equations \eqref{EinElasA.2} imply that the stress energy tensor $T^{\mu\nu}$ satisfies
\eqn{Tdiv}{
\nabla_\mu T^{\mu\nu} = 0 \quad \text{in $M$,}
}
in the distributional sense. This is enough to conclude from the reduced equations \eqref{EinElasE}, with the help of
the contracted Bianchi identity, that $\xi^\mu$  weakly solves a linear wave equation
of the form
\eqn{XwaveA}{
\nabla_\mu \nabla^\mu \xi^\nu + C^\nu_\mu \xi^\mu = 0 \quad \text{in $M$.}
}
Moreover, it is a consequence of the constraint equations \eqref{EinElasD} and \eqref{xidata} that
\eqn{XwaveB}{
(\xi^\mu,\Lc_t \xi^\mu) = 0 \quad \text{in $\Sigma$}.
}
By uniqueness of weak solutions to linear wave equations, it follows that
\eqn{XwaveC}{
\xi^\mu = 0 \quad \text{in $M$,}
}
completing our local existence and uniqueness argument.

\bigskip

\noindent \emph{\textbf{Acknowledgments}}

\smallskip

This work was partially supported by the ARC grants DP1094582 and FT1210045.
Part of this work was completed during a visit of the author T.A.O. 
to the Albert Einstein Institute. We are grateful to the Institute for
its support and hospitality during these visits. We also thank B. Schmidt for many illuminating and productive discussions.
Finally, we thank the referee for their comments
and criticisms, which have served to improve the content
and exposition of this article.
\appendix

\sect{calculus}{Calculus inequalities}

In this appendix we state, for the convenience of the reader, some well known calculus inequalities for the standard Sobolev spaces $W^{s,p}(\Omega)$, and we derive a number of
related inequalities for the $H^{k,s}(\Gbb^n)$ spaces. In the following, $\Omega$ will always denote a bounded, open subset of $\Gbb^n$ with
a smooth boundary.

\subsect{Sineq}{Spatial inequalities}

The proof of the following inequalities are well known and may be found, for example, in
the books \cite{AdamsFournier:2003}, \cite{Friedman:1976} and \cite{TaylorIII:1996}. Alternatively, one can
also consult Appendix A of Koch's thesis \cite{Koch:1990} for detailed proofs.

\begin{thm}{\emph{[H\"{o}lder's inequality]}} \label{Holder}
If $0< p,q,r \leq \infty$ satisfy $1/p+1/q = 1/r$, then
\eqn{HolderA}{
\norm{uv}_{L^r(\Omega)} \leq \norm{u}_{L^p(\Omega)}\norm{v}_{L^q(\Omega)}
}
for all $u\in L^p(\Omega)$ and $v\in L^q(\Omega)$.
\end{thm}

\begin{thm}{\emph{[Sobolev's inequality]}} \label{Sobolev} Suppose $s\in \Zbb_{\geq 1}$ and
$1\leq p < \infty$.
\begin{itemize}
\item[(i)] If $sp<n$, then
\eqn{Sobolev1}{
\norm{u}_{L^q(\Omega)} \lesssim \norm{u}_{W^{s,p}(\Omega)} \qquad p\leq q \leq np/(n-s p)
}
for all $u\in W^{s,p}(\Omega)$.
\item[(ii)] (Morrey's inequality) If $sp > n$, then
\eqn{Sobolev3}{
\norm{u}_{C^{0,\mu}(\overline{\Omega})} \lesssim \norm{u}_{W^{s,p}(\Omega)} \qquad 0 < \mu \leq \min\{1,s-n/p\}
}
for all $u\in W^{s,p}(\Omega)$.
\end{itemize}
\end{thm}

\begin{thm}{\emph{[Interpolation]}} \label{interp}
Suppose $\ep_0 >0$, $1\leq p \leq \infty$, $k,s\in \Zbb_{\geq 0}$ and $k\leq s$. Then there exists
a constant $K>0$ such that
\eqn{thm}{
|u|_{k,p} \leq K\bigl(\ep |u|_{s,p} + \ep^{-k/(s-k)}\norm{u}_{L^p(\Omega)}\bigr)
}
for $0<\ep \leq \ep_0$, where $|\cdot|_{k,p}$ is the seminorm defined by
\eqn{interpB}{
|u|_{k,p} = \left(\sum_{|\alpha|=k}\norm{D^\alpha u}_{L^p(\Omega)}^p\right)^{1/p}.
}
\end{thm}

\begin{thm}{\emph{[Multiplication inequality]}} \label{calcpropB}
Suppose $1\leq p <\infty$, $s_1,s_2,\ldots s_{\ell+1}\in \Zbb$, $s_1,s_2,\ldots,s_\ell \geq s_{\ell+1}\geq 0$, and $\sum_{j=1}^\ell s_j -n/p > s_{\ell+1}$. Then
\eqn{calcpropB.1}{
\norm{u_1 u_2 \cdots u_\ell}_{W^{p,s_{\ell+1}}(\Omega)} \lesssim \norm{u_1}_{W^{p,s_1}(\Omega)} \norm{u_2}_{W^{p,s_2}(\Omega)} \cdots   \norm{u_\ell}_{W^{p,s_\ell}(\Omega)}
}
for all $u_i \in W^{p,s_i}(\Omega)$ $i=1,2,\ldots,\ell$.
\end{thm}

\begin{thm}{\emph{[Gagliardo-Nirenberg's inequality]}} \label{GNMa}
If $1\leq p,q,r\leq \infty$, $s\in \Zbb_{\geq 1}$ and $|\alpha|\leq s$, then
\eqn{GNMa1}{
\norm{D^\alpha u}_{L^r(\Omega)} \lesssim \norm{u}^{1-|\alpha|/s}_{L^q(\Omega)}
\norm{u}_{W^{s,p}(\Omega)}^{|\alpha|/s}
}
for all $u\in L^q(\Omega) \cap W^{s,p}(\Omega)$, where
\eqn{GNM2}{
\frac{s-|\alpha|}{sq} + \frac{|\alpha|}{sp} = \frac{1}{r}.
}
In particular
\eqn{GNMa2}{
\norm{D^\alpha u}_{L^{\frac{sp}{|\alpha|}}(\Omega)} \lesssim \norm{u}_{L^\infty(\Omega)}^{1-\frac{|\alpha|}{s}}
\norm{u}^{\frac{|\alpha|}{s}}_{W^{s,p}(\Omega)}.
}
\end{thm}

\begin{thm}{\emph{[Moser's inequality]}}  \label{GNMb}
Suppose $s\in \Zbb_{\geq 1}$, $1\leq p \leq \infty$, $|\alpha|\leq s$, $f\in C^s(\Rbb)$, $f(0)=0$, $u\in C^0(\Omega)\cap L^\infty(\Omega)\cap W^{s,p}(\Omega)$, and
$u(x) \in V$ for all $x\in \Omega$ where $V$ is open and bounded in $\Rbb$.  Then
\eqn{frop1}{
\norm{D^\alpha f(u)}_{L^p(\Omega)} \leq C(\norm{f}_{C^s(\overline{V})})(1+\norm{u}^{s-1}_{L^\infty(\Omega)})\norm{u}_{W^{s,p}(\Omega)}.
}
\end{thm}

\subsect{STineq}{Spacetime inequalities}
We now prove spacetime versions of the multiplication and Moser inequalities adapted to the $\Hc^{0,s}(\Gbb^n)$, $\Xc_T^{s}(\Gbb^n)$ and $X_T^{s}(\Omega)$ spaces.

\begin{prop} \label{elemE}
Suppose $s_1,s_2,\ldots s_{\ell+1}\in \Zbb$, $s_1,s_2,\ldots,s_\ell \geq s_{\ell+1}\geq 0$, and $\sum_{j=1}^\ell s_j -n/2 > s_{\ell+1}$. Then
\eqn{elemE.1}{
\norm{u_1 u_2 \cdots u_\ell}_{\Hc^{0,s_{\ell+1}}(\Tbb^n)} \lesssim \norm{u_1}_{\Hc^{0,s_1}(\Tbb^n)} \norm{u_2}_{\Hc^{0,s_2}(\Tbb^n)} \cdots   \norm{u_\ell}_{\Hc^{0,s_\ell}(\Tbb^n)}
}
for all $u_i \in \Hc^{0,s_i}(\Tbb^n)$ $i=1,2,\ldots,\ell$.
\end{prop}
\begin{proof}
By Theorem \ref{calcpropB}, we have that
\alin{elemE.3}{
\norm{u_1 u_2\cdots u_\ell}_{H^{s_{\ell+1}}(\Omega)}
\lesssim  \norm{u_1}_{H^{s_1}(\Omega)}\norm{u_2}_{H^{s_2}(\Omega)} \cdots \norm{u_\ell}_{H^{s_\ell}(\Omega)}
\intertext{and}
\norm{u_1 u_2\cdots u_\ell}_{H^{s_3}(\Omega^c)}
\lesssim  \norm{u_1}_{H^{s_1}(\Omega^c)}\norm{u_2}_{H^{s_2}(\Omega^c)} \cdots \norm{u_\ell}_{H^{s_\ell}(\Omega^c)}.
}
Moreover, it is obvious that
\eqn{elemE.5}{
\norm{u_1 u_2 \cdots u_\ell}^2_{L^2(\Tbb^n)} = \norm{u_1 u_2\cdots u_\ell}^2_{L^2(\Omega)} +  \norm{u_1 u_2\cdots u_\ell}^2_{L^2(\Omega^c)} \leq \norm{u_1 u_2 \cdots u_\ell}^2_{H^{s_{\ell+1}}(\Omega)} + \norm{u_1 u_2 \cdots u_\ell}^2_{H^{s_{\ell+1}}(\Omega^c)}.
}
The desired inequality
\eqn{elemE.6}{
\norm{u_1 u_2 \cdots u_\ell}_{\Hc^{0,s_{\ell+1}}(\Tbb^n)} \lesssim \norm{u_1}_{\Hc^{0,s_1}(\Tbb^n)} \norm{u_2}_{\Hc^{0,s_2}(\Tbb^n)} \cdots   \norm{u_\ell}_{\Hc^{0,s_\ell}(\Tbb^n)}
}
now follows directly from the above inequalities.
\end{proof}

The next four propositions are closely related to Lemma 3.2 and Theorem A.6 from \cite{Koch:1993}. Since proofs of
 Lemma 3.2 and Theorem A.6 are not provided in \cite{Koch:1993}, we, for the convenience of the reader, provide some of the details here.

\begin{prop} \label{fpropB} Suppose $s\in \Zbb_{> n/2}$, $f\in C^s(\Rbb)$, $f(0)=0$, $u\in \Xc_T^s(\Gbb^n)$ and
$v\in Y_T^s(\Omega)$ with $\Omega \subset \Gbb^n$. Then 
\eqn{propB1}{
\norm{\del{t}^\ell f(u)}_{\Hc^{0,s-\ell}(\Gbb^n)} \leq C(\norm{u}_{\Ec^s(\Gbb^n)}) \AND
\norm{\del{t}^\ell f(v)}_{H^{s-\ell}(\Omega)} \leq C(\norm{u}_{E^s(\Omega)}).
}
for $0\leq \ell \leq s$.
\end{prop}
\begin{proof}
We begin by differentiating $f(u)$ $\ell$-times $(0\leq \ell \leq s)$ with respect to $t$  to get
\leqn{fpropB3}{
\del{t}^\ell f(u) = \sum_{k_1+\cdots+k_m=\ell} f_{k_1,\ldots,k_m}(u) \del{t}^{k_1} u \cdots \del{t}^{k_m} u,
}
where $f_{k_1,\ldots,k_m} \in C^{s-\ell}(\Rbb)$. Noting that $s+ \sum_{j=1}^m (s-k_j) - n/2 = ms - \ell + s-n/2> s - \ell$, we
see that we can apply Proposition \ref{elemE} to \eqref{fpropB3} to get
\eqn{fpropB4}{
\norm{\del{t}^\ell f(u)}_{\Hc^{0,s-\ell}(\Tbb^n)} = \norm{f_{k_1,\ldots,k_m}(u)}_{\Hc^{0,s}(\Tbb^n)} \norm{\del{t}u^{k_1}}_{\Hc^{0,s-k_1}(\Tbb^n)} \cdots
 \norm{\del{t}^{k_m} u }_{\Hc^{0,s-k_m}(\Tbb^n)}.
}
Combining this estimates together with Theorems \ref{Sobolev} and \ref{GNMb}, we arrive at the
desired estimate
\eqn{fpropB5}{
\norm{\del{t}^\ell f(u)}_{\Hc^{0,s-\ell}(\Tbb^n)} \leq  C_\ell(\norm{u}_{\Ec^s(\Tbb^n)}).
}
The other estimate
\eqn{fpropB7}{
\norm{\del{t}^\ell f(v)}_{H^{s-\ell}(\Omega)} \leq C(\norm{v}_{E^s(\Omega)})
}
is proved in a similar manner.
\end{proof}

\begin{prop} \label{STpropA}
Suppose $s\in \Zbb_{\geq 1}$, $1\leq p \leq \infty$, $f\in C^s(\Rbb\times \Rbb^{n+1},\Rbb)$, $u\in W^{p,s+1}(\Omega)$,
$\del{t}u \in W^{p,s}(\Omega)$
and the higher time derivatives $\del{t}^\ell u$ $(\ell \geq 2)$ are obtained by formally differentiating
\eqn{STpropA2}{
\del{t}^2 u = a^{ij}(u,\del{}u)\del{i}\del{j}u + b^i(u,\del{}u)\del{i}\del{t}u + g(u,\del{}u) +h,
}
where $h\in W^{p,s+1}(\Omega)$, $\del{t}^\ell h =0$ and $a^{ij},b^j,g \in C^s(\Rbb\times \Rbb^{n+1},\Rbb)$.
Then $u$ satisfies the estimate
\eqn{STpropA3}{
\norm{\del{t}^j f(u,\del{} u)}_{W^{p,s-j}(\Omega)}
\leq C\bigl(\norm{u}_{W^{1,\infty}(\Omega)},\norm{\del{t}u}_{L^\infty(\Omega)}\bigr)
(1+\norm{u}_{W^{p,s+1}(\Omega)}+\norm{\del{t}u}_{W^{p,s}(\Omega)})
}
for $0\leq j \leq s$.
\end{prop}
\begin{proof}
First, we observe that estimate
\leqn{STpropA4}{
\norm{\del{t}^\ell f(u,Du,\del{t}u)}_{W^{p,s-\ell}(\Omega)}
\leq C\bigl(\norm{u}_{W^{1,\infty}(\Omega)},\norm{\del{t}u}_{L^\infty(\Omega)}\bigr)
(1+\norm{u}_{W^{p,s+1}(\Omega)}+\norm{\del{t}u}_{W^{s,p}(\Omega)})
}
holds for $\ell=0$, $1\leq p \leq \infty$ and $s\in \Zbb_{\geq 1}$ thanks to Theorem \ref{GNMb}. We now proceed by induction and assume that \eqref{STpropA4}
holds for $1\leq p \leq \infty$, $s\in \Zbb_{\geq 1}$,  $\ell=0,\ldots,\min\{s-1,j-1\}$, and maps $f\in C^s$, where the constant $C$
in \eqref{STpropA4} also depends implicitly
on the $C^s$ norm of $f$. In particular, this implies that
\leqn{STpropA5}{
\norm{\del{t}^{j-1}(D_1 f(u,Du,\del{t}u) \del{t}u)}_{W^{s-j+1,p}(\Omega)}
\leq C
(1+\norm{u}_{W^{s+1,p}(\Omega)}+\norm{\del{t}u}_{W^{s,p}(\Omega)}),
}
where
\leqn{STpropA6}{
C = C\bigl(\norm{u}_{W^{1,\infty}(\Omega)},\norm{\del{t}u}_{L^\infty(\Omega)}\bigr).
}

To proceed, we write the $j^\text{th}$ time derivative of $f(u,Du,\del{t}u)$ as
\leqn{STpropA6a}{
\del{t}^j\bigl[f(u,Du,\del{t}u)\bigr] =
\del{t}^{j-1}\bigl(D_1 f(u,Du,\del{t}u) \del{t} u\bigr) + \del{t}^{j-1}\bigl(D_2 f(u,Du,\del{t}u) \cdot \del{t}Du
\bigr) +
\del{t}^{j-1}\bigl(D_3 f(u,Du,\del{t}u) \del{t}^2 u\bigr).
}
Since we can already estimate the first term of the right hand side of \eqref{STpropA6a} using \eqref{STpropA5},
we turn to estimating the second term, which we write as follows
\leqn{STpropA7}{
\del{t}^{j-1}(D_{2}f(u,Du,\del{t}u)\cdot\del{t}Du) = \sum^{j-1}_{k=0} a_k \bigr(\del{t}^k D_{2}f\bigr)
\cdot \bigl(\del{t}^{j-1-k} D\del{t} u\bigr).
}
Next, we observe that
\leqn{STpropA8}{
D^\alpha\bigl[\del{t}^k D_{2}f
\cdot \del{t}^{j-1-k} D\del{t} u\bigr] =
\sum_{\beta+\gamma = \alpha} a_{\alpha\beta} \bigl(D^\beta \del{t}^k D_{2}f\bigr)\cdot\bigl( D^\gamma
\del{t}^{j-1-k} D \del{t}u\bigr) \qquad |\alpha|= s-j,
}
for appropriate constants $a_k$ and $a_{\alpha\beta}$. Letting
$C$ denote a constant of the form \eqref{STpropA6}, we now estimate \eqref{STpropA8} as follows:
\lalign{STpropA9}{
&\norm{D^\alpha\bigl[\del{t}^k D_{2}f
\cdot \del{t}^{j-1-k} D\del{t} u\bigr]}_{L^p(\Omega)}
 \lesssim \sum_{|\beta|+|\gamma|=s-j} \norm{\bigl(D^\beta \del{t}^k D_{2}f\bigr)\cdot\bigl( D^\gamma
\del{t}^{j-1-k} D \del{t}u\bigr)}_{L^p(\Omega)} \notag \\
&\text{\hspace{0.1cm}}\lesssim \sum_{|\beta|+|\gamma|=s-j}\norm{D^\beta \del{t}^k D_{2}f}_{L^{\frac{ps}{|\beta|+k}}(\Omega)}
\norm{D^\gamma
\del{t}^{j-1-k} D \del{t}u}_{L^{\frac{ps}{|\gamma|+j-k}}(\Omega)} \text{\hspace{0.2cm} by H\"{o}lder's inequality}  \notag \\
&\text{\hspace{0.1cm}}\lesssim \sum_{|\beta|+|\gamma|=s-j}
\norm{\del{t}^k D_{2}f}_{W^{|\beta|+k-k,\frac{ps}{|\beta|+k}}(\Omega)}
\norm{\del{t}^{j-1-k}\del{t}u}_{W^{|\gamma|+j-k- (j-k-1),\frac{ps}{|\gamma|+j-k}}(\Omega)} \notag\\
&\text{\hspace{0.1cm}}\leq C
\sum_{|\beta|+|\gamma|=s-j}\biggl(1+\norm{u}_{W^{|\beta|+k+1,\frac{ps}{|\beta|+k}}(\Omega)}+
\norm{\del{t}u}_{W^{|\beta|+k,\frac{ps}{|\beta|+k}}(\Omega)}\biggr) \notag \\
&\text{\hspace{0.2cm}} \times \biggl(1+\norm{u}_{W^{|\gamma|+j-k+1,\frac{ps}{|\gamma|+j-k}}(\Omega)}+
\norm{\del{t}u}_{W^{|\gamma|+j-k,\frac{ps}{|\gamma|+j-k}}(\Omega)}\biggr)
\text{\hspace{0.3cm} by induction hypothesis }  \notag \\
&\text{\hspace{0.1cm}}\leq C
\sum_{|\beta|+|\gamma|=s-j}\biggl(1+\norm{u}_{W^{s+1,p}(\Omega)}^{\frac{|\beta|+k}{s}}+
\norm{\del{t}u}_{W^{s,p}(\Omega)}^{\frac{|\beta|+k}{s}}\biggr) \notag \\
&\text{\hspace{5.4cm}} \times \biggl(1+\norm{u}_{W^{s,p}(\Omega)}^{\frac{|\gamma|+j-k}{s}}+
\norm{\del{t}u}_{W^{s,p}(\Omega)}^{\frac{|\gamma|+j-k}{s}}\biggr)
\text{\hspace{0.3cm} by Theorem \ref{GNMa}}  \notag \\
&\text{\hspace{0.1cm}}\leq C
(1+\norm{u}_{W^{s+1,p}(\Omega)}+\norm{\del{t}u}_{W^{s,p}(\Omega)}). \notag
}
This estimate together with the formula \eqref{STpropA8}, shows that we can, with the help
of Theorem \ref{interp}, estimate \eqref{STpropA7} by
\leqn{STpropA10}{
\norm{\del{t}^{j-1}(D_{2}f(u,Du,\del{t}u)\cdot \del{t}Du)}_{W^{s-j,p}(\Omega)}
\leq C
(1+\norm{u}_{W^{s+1,p}(\Omega)}+\norm{\del{t}u}_{W^{s,p}(\Omega)}),
}
where the constant $C$ is of the form \eqref{STpropA6}.

With the second term in \eqref{STpropA6a} estimated, we use relation
\leqn{STpropA11}{
\del{t}^2 u = a^{ij}(u,\del{}u)\del{i}\del{j}u + b^i(u,\del{}u)\del{i}\del{t}u + g(u,\del{}u) +h,
}
to write the third term in \eqref{STpropA6a} as
\eqn{STpropA12}{
\del{t}^{j-1}(D_3 f(u,Du,\del{t}u)\del{t}^2 u) = \del{t}^{j-1}\bigl[D_3 f(u,Du,\del{t}u) \bigl(a^{ij}(u,\del{}u)\del{i}\del{j}u + b^i(u,\del{}u)\del{i}\del{t}u + g(u,\del{}u) +h\bigr)\bigr].
}
Similar arguments employed above to derive \eqref{STpropA11} show also that
\leqn{STpropA13}{
\norm{\del{t}^{j-1}(D_{3}f(u,Du,\del{t}u)\del{t}^2u)}_{W^{s-j,p}(\Omega)}
\leq C
(1+\norm{u}_{W^{s+1,p}(\Omega)}+\norm{\del{t}u}_{W^{s,p}(\Omega)})
}
for a constant $C$ of the form \eqref{STpropA6}.
Together, the estimates \eqref{STpropA5}, \eqref{STpropA10} and \eqref{STpropA13} show that
\eqn{STpropA14}{
\norm{\del{t}^j f(u,Du,\del{t}u)}_{W^{p,s-j}(\Omega)}
\leq C\bigl(\norm{u}_{W^{1,\infty}(\Omega)},\norm{\del{t}u}_{L^\infty(\Omega)}\bigr)
(1+\norm{u}_{W^{p,s+1}(\Omega)}+\norm{\del{t}u}_{W^{s,p}(\Omega)}).
}
This completes the induction argument and the proof of the proposition.
\end{proof}

Similar arguments can be used to prove the following variant of Proposition \ref{STpropA}.

\begin{prop} \label{STpropC}
Suppose $s\in \Zbb_{\geq 1}$, $1\leq p \leq \infty$, $f\in C^{s+1}(\Rbb,\Rbb)$, $u\in W^{p,s+1}(\Omega)$,
$\del{t}u \in W^{p,s}(\Omega)$
and the higher time derivatives $\del{t}^\ell u$ $\ell \geq 2$ are obtained by formally differentiating
\eqn{STpropC2}{
\del{t}^2 u = a^{ij}(u,\del{}u)\del{i}\del{j}u + b^i(u,\del{}u)\del{i}\del{t}u + g(u,\del{}u) +h,
}
where $h\in W^{p,s+1}(\Omega)$, $\del{t}^\ell h =0$ and $a^{ij},b^j,g \in C^s(\Rbb\times \Rbb^{n+1},\Rbb)$.
Then $u$ satisfies the estimate
\eqn{STpropC3}{
\norm{\del{t}^j f(u)}_{W^{p,s+1-j}(\Omega)}
\leq C\bigl(\norm{u}_{W^{1,\infty}(\Omega)},\norm{\del{t}u}_{L^\infty(\Omega)}\bigr)
(1+\norm{u}_{W^{p,s+1}(\Omega)}+\norm{\del{t}u}_{W^{p,s}(\Omega)})
}
for $0\leq j \leq s+1$.
\end{prop}

\begin{prop} \label{STpropB}
Suppose $s\in \Zbb_{\geq 1}$, $1\leq p \leq \infty$, $f^{\mu\nu}\in C^{s+1}(\Rbb,\Rbb)$, $u\in W^{p,s+1}(\Omega)$,
$\del{t}u \in W^{p,s}(\Omega)$
and the higher time derivatives $\del{t}^\ell u$ $\ell \geq 2$ are obtained by formally differentiating
\eqn{STpropB1}{
\del{t}^2 u = a^{ij}(u,\del{}u)\del{i}\del{j}u + b^i(u,\del{}u)\del{i}\del{t}u + g(u,\del{}u) +h,
}
where $h\in W^{p,s+1}(\Omega)$, $\del{t}^\ell h =0$ and $a^{ij},b^j,g \in C^s(\Rbb\times \Rbb^{n+1},\Rbb)$.
Then
\eqn{STpropB2}{
\norm{\del{\mu}[\del{t}^k,f^{\mu\nu}(u)\del{\nu}]u}_{W^{p,s-k}(\Omega)}
\leq C\bigl(\norm{u}_{W^{1,\infty}(\Omega)},\norm{\del{t}u}_{L^\infty(\Omega)}\bigr)
(1+\norm{u}_{W^{p,s+1}(\Omega)}+\norm{\del{t}u}_{W^{p,s}(\Omega)})
}
for $0\leq k \leq s$.
\end{prop}
\begin{proof}
Differentiating the formula
\eqn{StpropB5}{
[\del{t}^k,f^{\nu\mu}(u)\del{\mu}]u = \sum_{\ell=0}^{k-1}\binom{k}{\ell}\bigl[\del{t}^{k-\ell}f^{\nu i}\del{i}\del{t}^\ell  u + \del{t}^{k-\ell}f^{\nu 0} \del{t}^{\ell+1}u\bigr],
}
we see that
\lalign{StpropB6}{
\del{j}\bigl([\del{t}^k,f^{j\mu}(u)\del{\mu}]u\bigr)
=  \sum_{\ell=0}^{k-1}\binom{k}{\ell}\bigl[&\del{t}^{k-\ell}(Df^{j i}\del{j}u)\del{i}\del{t}^\ell u +
 \del{t}^{k-\ell}f^{ji}\del{j}\del{i}\del{t}^\ell  u   \notag \\
 & +
\del{t}^{k-\ell}(Df^{j 0}\del{j}u) \del{t}^{\ell+1}u   + \del{t}^{k-\ell}f^{j 0} \del{t}^{\ell+1}\del{j}u \bigr] \label{StpropB6.1}
}
and
\lalign{StpropB7}{
\del{0}\bigl([\del{t}^k,f^{0\mu}(u)\del{\mu}]u\bigr)
=  \sum_{\ell=0}^{k-1}\binom{k}{\ell}\bigl[&\del{t}^{k-\ell}(Df^{0 i}\del{t}u)\del{i}\del{t}^\ell u +
 \del{t}^{k-\ell}f^{0i}\del{t}\del{i}\del{t}^\ell  u   \notag \\
 & +
\del{t}^{k-\ell}(Df^{0 0}\del{t}u) \del{t}^{\ell+1}u   + \del{t}^{k-\ell}f^{0 0} \del{t}^{\ell+2} u \bigr]. \label{StpropB7.1}
}

To estimate \eqref{StpropB6.1} and \eqref{StpropB7.1}, we start by differentiating the term
\leqn{STpropB8}{
\del{t}^{k-\ell}f^{ji}\del{j}\del{i}\del{t}^\ell  u
}
$s-k$ times
to get
\leqn{STpropB8a}{
D^\alpha\bigl(\del{t}^{k-\ell}f^{ji}\del{j}\del{i}\del{t}^\ell  u\bigr) = \sum_{\beta+\gamma=\alpha} a_{\alpha\beta} D^\beta\bigl(\del{t}^{k-\ell}f^{ji}\bigr)D^\gamma \bigl(\del{j}\del{i}\del{t}^\ell  u\bigr)
\quad |\alpha|=s-1-k
}
for appropriate constants $a_{\alpha\beta}$. Letting
$C$ denote a constant of the form,
\eqn{STpropB9}{
C\bigl(\norm{u}_{W^{1,\infty}(\Omega)},\norm{\del{t}u}_{L^\infty(\Omega)}\bigr),
}
we estimate \eqref{STpropB8a} as follows:
\lalign{STpropB10}{
&\norm{D^\alpha\bigl[\del{t}^{k-\ell}f^{ij}
\del{i}\del{j}\del{t}^{\ell} u\bigr]}_{L^p(\Omega)}
 \lesssim \sum_{|\beta|+|\gamma|=s-k} \norm{\bigl(D^\beta\del{t}^{k-\ell} f^{ij}\bigr)\bigl( D^\gamma
\del{i}\del{j}\del{t}^{\ell}u\bigr)}_{L^p(\Omega)} \notag \\
&\text{\hspace{0.1cm}}\lesssim \sum_{|\beta|+|\gamma|=s-k}\norm{D^\beta \del{t}^{k-\ell}f}_{L^{\frac{ps}{|\beta|+k-\ell-1}}(\Omega)}
\norm{D^{\gamma}D^2
\del{t}^{\ell} u}_{L^{\frac{ps}{|\gamma|+1+\ell}}(\Omega)} \text{\hspace{0.2cm} by H\"{o}lder's inequality}  \notag \\
&\text{\hspace{0.1cm}}\lesssim \sum_{|\beta|+|\gamma|=s-k}
\norm{\del{t}^{k-\ell}f}_{W^{|\beta|,\frac{ps}{|\beta|+k-\ell-1}}(\Omega)}
\norm{\del{t}^{\ell}u}_{W^{|\gamma|+2,\frac{ps}{|\gamma|+\ell+1}}(\Omega)} \notag\\
&\text{\hspace{0.1cm}}\leq C
\sum_{|\beta|+|\gamma|=s-k}\biggl(1+\norm{u}_{W^{|\beta|+k-\ell,\frac{ps}{|\beta|+k-\ell-1}}(\Omega)}+
\norm{\del{t}u}_{W^{|\beta|+k-\ell-1,\frac{ps}{|\beta|+k-\ell-1}}(\Omega)}\biggr) \notag \\
&\text{\hspace{0.2cm}} \times \biggl(1+\norm{u}_{W^{|\gamma|+2+\ell,\frac{ps}{|\gamma|+1+\ell}}(\Omega)}+
\norm{\del{t}u}_{W^{|\gamma|+1+\ell,\frac{ps}{|\gamma|+1+\ell}}(\Omega)}\biggr)
\text{\hspace{0.3cm} by Proposition \ref{STpropA} }  \notag \\
&\text{\hspace{0.1cm}}\leq C
\sum_{|\beta|+|\gamma|=s-k}\biggl(1+\norm{u}_{W^{s+1,p}(\Omega)}^{\frac{|\beta|+k-\ell-1}{s}}+
\norm{\del{t}u}_{W^{s,p}(\Omega)}^{\frac{|\beta|+k-\ell-1}{s}}\biggr) \notag \\
&\text{\hspace{5.4cm}} \times \biggl(1+\norm{u}_{W^{s,p}(\Omega)}^{\frac{|\gamma|+1+\ell}{s}}+
\norm{\del{t}u}_{W^{s,p}(\Omega)}^{\frac{|\gamma|+1+\ell}{s}}\biggr)
\text{\hspace{0.3cm} by Theorem \ref{GNMa}}  \notag \\
&\text{\hspace{0.1cm}}\leq C
(1+\norm{u}_{W^{s+1,p}(\Omega)}+\norm{\del{t}u}_{W^{s,p}(\Omega)}). \notag
}
This estimate together with the formula \eqref{STpropB8a} shows that we can, with the help
of Theorem \ref{interp}, estimate \eqref{STpropB8} by
\eqn{STpropB11}{
\norm{\del{t}^{k-\ell}f^{ji}\del{j}\del{i}\del{t}^\ell  u}_{W^{s-k}(\Omega)}
\leq C\bigl(\norm{u}_{W^{1,\infty}(\Omega)},\norm{\del{t}u}_{L^\infty(\Omega)}\bigr)
(1+\norm{u}_{W^{s+1,p}(\Omega)}+\norm{\del{t}u}_{W^{s,p}(\Omega)}),
}
for $0\leq \ell \leq k-1$. Using the same arguments, it is not difficult to verify similar estimates hold
for the remaining terms in \eqref{StpropB6.1} and \eqref{StpropB7.1}, which allows us to conclude that
\eqn{STpropB12}{
\norm{\del{\mu}[\del{t}^k,f^{\mu\nu}(u)\del{\nu}]u}_{W^{p,s-k}(\Omega)}
\leq C\bigl(\norm{u}_{W^{1,\infty}(\Omega)},\norm{\del{t}u}_{L^\infty(\Omega)}\bigr)
(1+\norm{u}_{W^{p,s+1}(\Omega)}+\norm{\del{t}u}_{W^{p,s}(\Omega)})
}
for $0\leq k \leq s$.
\end{proof}
\sect{potential}{Potential Theory}

In this appendix, we recall some results from potential theory that we require to prove energy estimates. We begin
by recalling the following well known result.\footnote{Here, $\Delta = \delta^{ij}\del{i}\del{j}$ is the flat Laplacian}

\begin{prop} \label{Lisoprop}
Suppose $p\in (1,\infty)$, $s\geq0$ and $\psi \in C^\infty(\Tbb^n)$ satisfies $\psi \geq 0$ on $\Tbb^n$
and $\psi(x_0) > 0$ for some $x_0 \in \Tbb^n$. Then the map
\eqn{Liso}{
\Delta - \psi : W^{s+1,p}(\Tbb^n) \longrightarrow W^{s-1,p}(\Tbb^n) \quad (s\geq 0)
}
is an isomorphism.
\end{prop}

Letting
\eqn{Lcdef}{
\Lc = (\Delta-\psi)^{-1} : W^{s-1,p}(\Tbb^n) \longrightarrow W^{s,p}(\Tbb^n)
}
denote the inverse of $\Delta-\psi$, we can represent $\Lc$ as
\eqn{Lcrep}{
\Lc v(x) = \int_{\Tbb^n}  E(x,y)v(y)\, d^n x
}
where $E$ is the integral kernel of $\Lc$.
 Fixing an open set $\Omega \subset \Tbb^n$ with $C^\infty$ boundary, we then define the \textit{single and double layer potentials} by
\leqn{Spotdef}{
\mathcal{S} v (x) =   \int_{\partial \Omega}  E(x,y)v(y)\, d\sigma (y) \quad x \notin \partial \Omega
}
and
\leqn{Dpotdef}{
\mathcal{D} v (x) =   \int_{\partial \Omega}  \frac{\partial E}{\partial \nu_y}(x,y)v(y)\, d\sigma (y) \quad x \notin \partial \Omega,
}
respectively. Here, $d\sigma$ is the natural area element on $\partial \Omega$ and $\nu$ is the outward unit conormal to $\Omega$.

\begin{prop} \label{potprop}
Suppose $p\in (1,\infty)$, $k\in \Zbb_{\geq 1}$, and $\psi(x_0) > 0$ for some\footnote{Recall that
 $\Omega^c = \Tbb^n\setminus \overline{\Omega}$.} $x_0 \in \Omega^c$. Then the linear map\footnote{Here $\Rc_\Omega$ denotes
 the restriction operator, i.e. for a function $f$ defined on $\Tbb^n$, $\Rc_\Omega f (x) := f(x)$ for all $x\in \Omega$.}
\eqn{potprop1}{
  \Rc_\Omega \circ\Lc \circ \chi_\Omega \: : \: W^{k,p}(\Omega)  \longrightarrow  W^{k+2,p}(\Omega)
}
is continuous (i.e. bounded).
\end{prop}
\begin{proof}
The proof follows from a straightforward adaptation of Proposition 3.6 in \cite{Andersson_et_al:2011}. Here, one  simply needs to use
the analogous mapping properties for the single and double layer potential, as defined above in \ref{Spotdef} and \ref{Dpotdef}, as
a replacement for the potential theory used in \cite{Andersson_et_al:2011} that was based on the (flat) Laplacian on $\Rbb^3$. With
this replacement, the proof from \cite{Andersson_et_al:2011} goes through directly without any further changes needed.
\end{proof}
\sect{weak}{Weak solutions of wave equations}

We recall some basic facts about weak solutions to linear wave equations.
We begin with the definition of a weak solution.
\begin{Def}
Suppose $a^{\mu\nu} \in W^{1,\infty}([0,T],L^\infty(\Gbb^n))$, $a^{\mu\nu}=a^{\nu\mu}$,
$p^\mu\in H^1([0,T],L^2(\Gbb^n))$, $q^{\nu}\in W^{1,\infty}\bigl([0,T],L^n(\Gbb^n)\bigr)\cap
L^\infty([0,T]\times\Gbb^n)$, and there exists a $\kappa > 0$ such that
\eqn{weakA}{
\kappa |\xi|^2 \leq a^{ij}\xi_i\xi_j \quad \text{for all $\xi = (\xi_i) \in \Rbb^n$}
\AND
a^{00} \leq -\kappa.
}
Then we say that $u \in H^1([0,T]\times \Gbb^n)$ is a \emph{weak solution} of
\lalign{weakB}{
\del{\mu}(a^{\mu\nu} \del{\nu} u) &= f + \del{\mu}( p^\mu + q^{\mu} u),  \label{weakB.1}\\
(u|_{t=0},\del{t}u|_{t=0})& = (u_0,u_1) \in H^1(\Gbb^n)\times L^2(\Gbb^n), \label{weakB.2}
}
if\footnote{Here, following standard notation, ``$\rightharpoonup$'' denotes weak convergence.}
\eqn{weakC}{
(u(t),\del{t}u(t)) \rightharpoonup (u_0,u_1) \quad \text{in $H^1(\Gbb^n)\times L^2(\Gbb^n)$}
}
and
\eqn{weakD}{
\ip{a^{\mu\nu}\del{\mu}u}{\phi}_{L^2([0,T]\times \Gbb^n)}
= -\ip{f}{\phi}_{L^2([0,T]\times \Gbb^n)} + \ip{p^\mu + q^{\mu} u}{\del{\mu}\phi}_{L^2([0,T]\times \Gbb^n)}
}
for all $\phi \in H^1([0,T],L^2(\Gbb^n))\cap L^2([0,T],H^1(\Gbb^n))$.
\end{Def}

With the above notion of a weak solution, the proof of the next theorem is just a special case
of Theorem 2.2 from \cite{Koch:1993}.

\begin{thm} \label{weakthm}
Suppose $u$ is a weak solution of \eqref{weakB.1}-\eqref{weakB.2}. Then $u \in C([0,T],H^1(\Gbb^n))\cap
C^1([0,T],L^2(\Gbb^n))$ and $u$ satisfies the estimate
\eqn{weakE}{
\norm{u(t_2)}_{E(\Gbb^n)} \leq c\left(\norm{u(t_1)}_{E(\Gbb^n)} + d_1 + \int_{t_1}^{t_2} d_2(\tau)\norm{u(\tau)}_{E(\Gbb^n)}
+d_3(\tau) \, d\tau
 \right)
}
for all $0\leq t_1 \leq t_2 \leq T$, where
\alin{weakF}{
d_1 &=  \norm{p(t_1)}_{L^2(\Gbb^n)} +(1+\norm{q(t_1)}_{L^\infty(\Gbb^n)})\norm{u(t_1)}_{L^2(\Gbb^n)},\\
d_2(t) &= 1+ \norm{\del{t}a(t)}_{L^\infty(\Gbb^n)} + \norm{q(t)}_{L^\infty(\Gbb^n)} +
\norm{\del{t}q(t)}_{L^n(\Gbb^n)} \\
d_3(t) &= \norm{f(t)}_{L^2(\Gbb^n)} + \norm{\del{t}p(t)}_{L^n(\Gbb^n)}
}
and $c=c(\kappa,\norm{a}_{L^\infty([0,T]\times \Gbb^n)})$.
\end{thm}
\sect{scale}{Field rescalings}

In this appendix, we establish the behavior of the norms $\Hc^{k,s}(Q_1)$ and $H^s(Q^+_1)$ under rescaling. These results
are used repeatedly in Section \ref{linear} when we exploit the freedom to localize the estimates for the linear IVP \eqref{linIVP.1}-\eqref{linIVP.2}.

\begin{prop} \label{scalepropA}
Suppose $0<\delta\leq 1$, $s,\ell\in \Zbb_{\geq 0}$, $n\geq 3$, $0\leq \sigma \leq 1$, $0\leq \sigma < s-n/2$, $s-\ell \geq 0$,
$f\in \Hc^{2,s}(Q_1)$,  $g\in \Hc^{0,s-\ell}(Q_1)$, $h\in \Hc^{m_{s-\ell},s-\ell}(Q_1)$
and let
\eqn{scalepropA.2}{
f_\delta(x) = \frac{f(\delta x)-f(0)}{\delta^\sigma}, \quad g_\delta(x) = g(\delta x) \AND h_\delta(x) = h(\delta x).
}
Then
\alin{scalepropA.3}{
\norm{f_\delta}_{\Hc^{2,s}(Q_1)} &\lesssim  \norm{f}_{\Hc^{2,s}(Q_1)}, \\
\norm{g_\delta}_{\Hc^{0,s-\ell}(Q_1)} &\lesssim \frac{1}{\delta^\ell}\norm{g}_{\Hc^{0,s-\ell}(Q_1)}
\intertext{and}
\norm{h_\delta}_{\Hc^{m_{s-\ell},s-\ell}(Q_1)} &\lesssim \frac{1}{\delta^\ell}\norm{h}_{\Hc^{m_{s-\ell},s-\ell}(Q_1)}.
}
\end{prop}
\begin{proof}
First, a short calculation shows that
\eqn{feq1}{
\norm{D^s f_\delta}_{L^2(Q^\pm_1)}^2 = \delta^{2(s-\sigma)-n}\norm{D^s f}_{L^2(Q^\pm_\delta)}^2 \leq \delta^{2(s-\sigma)-n} \norm{D^s f}_{L^2(Q^\pm_1)}^2,
}
and in particular, that
\leqn{fep2}{
\norm{D^s f_\delta}_{L^2(Q^\pm_1)} \leq \norm{D^s f}_{L^2(Q^\pm_1)}
}
since $s-\sigma -n/2 \geq 0$ by assumption.

Next, we observe that
\alin{fep3}{
\norm{D^2 f_\delta}_{L^2(Q_1)}^2 &= \delta^{2(2-\sigma)-n} \norm{D^2 f}_{L^2(Q_\delta)}^2 \\
&= \delta^{2(2-\sigma)-n}\left( \int_{Q^-_\delta}|D^2 f| \, d^n x + \int_{Q^+_\delta}|D^2 f|^2 \, d^n x\right)
}
which implies via H\"{o}lder's inequality, Theorem \ref{Holder}, that
\lalign{fep4}{
\norm{D^2 f_\delta}_{L^2(Q_1)}^2 & \leq \delta^{2-n} \bigl( \norm{1}_{L^p(Q^-_\delta)}\norm{D^2 f}_{L^{2q}(Q^-_\delta)}^2 + \norm{1}_{L^p(Q^+_\delta)}\norm{D^2 f}_{L^{2q}(Q^+_\delta)}^2\bigr)\notag \\
& \leq 2^{(n-1)/p}\delta^{2(2-\sigma)-n+n/p} \bigl(\norm{D^2 f}_{L^{2q}(Q^-_1)}^2 + \norm{D^2 f}_{L^{2q}(Q^+_1)}^2\bigr) \notag 
}
for $1/q + 1/p=1$. Choosing $p=n/n-2$ and hence $q=n/2$, the above inequality becomes
\eqn{fep5}{
\norm{D^2 f_\delta}_{L^2(Q_1)}^2 \leq  2^{(n-1)(n-2)/n}\delta^{2(1-\sigma)}\bigl(\norm{D^2 f}_{L^{n}(Q^-_1)}^2 + \norm{D^2 f}_{L^{n}(Q^+_1)}^2\bigr)
}
However, by Sobolev's inequality, Theorem \ref{Sobolev}, we have that
\eqn{fep6}{
\norm{D^2 u}_{L^n(Q^\pm_1)} \lesssim \norm{D^2 f}_{H^{n/2-1}(Q^\pm_1)},
}
and this allows us to conclude that
\leqn{fep7}{
\norm{D^2 f_\delta}_{L^2(Q_1)} \lesssim \bigl(\norm{f}_{H^{s}(Q^-_1)} + \norm{f}_{H^{s}(Q^+_1)}\bigr),
}
since $s> n/2$ and $\sigma \leq 1$.

Next, we observe that
\eqn{fep8}{
|f_{\delta}(x)| \leq |x|^\sigma \frac{|f(\delta x)-f(0)|}{|\delta x - 0|^\sigma} \leq  \max\bigl\{\norm{f}_{C^{0,\sigma}(Q_{\delta}^{+})},\norm{f}_{C^{0,\sigma}(Q_{\delta}^{-})} \bigr\}
\leq \max\bigl\{\norm{f}_{C^{0,\sigma}(Q_{1}^{+})},\norm{f}_{C^{0,\sigma}(Q_{1}^{-})} \bigr\}
}
for all $x\in Q_1$. This together with Sobolev's inequality gives
\eqn{fep9}{
\norm{f_{\delta}}_{L^\infty(Q_1)} \lesssim  \norm{f}_{\Hc^{0,s}(Q_1)},
}
which, in turn, implies, with the help of H\"{o}lder's inequality, that
\leqn{fep10}{
\norm{f_{\delta}}_{L^2(Q_1)} \lesssim  \norm{f}_{\Hc^{0,s}(Q_1)}.
}
The inequalities \eqref{fep2}, \eqref{fep7} and \eqref{fep10} together with interpolation, see Theorem \ref{interp}, then
show that
\eqn{fep11}{
\norm{f_\delta}_{\Hc^{2,s}(Q_1)} \lesssim \norm{f}_{\Hc^{2,s}(Q_1)},
}
while, a short calculation shows that
\leqn{fep12}{
\norm{D^{s-\ell}g_\delta}^2_{L^2(Q^{\pm}_1)} = \delta^{2(s-\ell)-n}\norm{D^{s-\ell}g}^2_{L^2(Q^\pm_\delta)}
\leq \frac{1}{\delta^{2\ell}} \norm{D^{s-\ell}g}^2_{L^2(Q^\pm_1)},
}
since $s>n/2$ implies that $2 s-n>0$.

We now consider the two cases\footnote{We can avoid the case $s-\ell = n/2$ for $n$ even by replacing $s$ by a non-integral $\st$ which is slightly less than
$s$ while using the version of Sobolev's inequality that is valid for the fractional Sobolev spaces. 
}

\bigskip

\noindent \textit{Case 1:} $s-\ell > n/2$

\smallskip

Suppose now that $s-\ell > n/2$. Then
\eqn{fep13}{
\norm{g_\delta}_{L^\infty(Q^\pm_1)} \leq \norm{g}_{L^\infty(Q^\pm_1)} \lesssim \norm{g}_{H^{s-\ell}(Q^\pm_1)}
}
by Sobolev's inequality, and so
\leqn{fep14}{
\norm{g_\delta}_{L^2(Q^\pm_1)} \leq \norm{g}_{L^\infty(Q^\pm_1)}\norm{1}_{L^2(Q^\pm_1)} \lesssim \norm{g}_{H^{s-\ell}(Q^\pm_1)}
}
follows by H\"{o}lder's inequality.

\bigskip

\noindent \textit{Case 2:} $s-\ell < n/2$

\smallskip

Suppose now that $s-\ell < n/2$. Then
\leqn{fep15}{
\norm{g_\delta}_{L^2(Q^\pm_1)} = \delta^{-n/2}\norm{g}_{L^2(Q^\pm_\delta)}
\leq \delta^{-n/2}\norm{1}_{L^{n/(s-\ell)}(Q^\pm_\delta)} \norm{g}_{L^q}(Q^\pm_\delta)
}
for
\eqn{fep16}{
\frac{1}{q} = \frac{1}{2} - \frac{(s-\ell)}{n}
}
by H\"{o}lder's inequality. But
$\norm{1}_{L^{n/(s-\ell)}} = 2^{(n-1)(s-\ell)/n}\delta^{s-\ell}$,
and so, we see from \eqref{fep15} that
\eqn{fep18}{
\norm{g_\delta}_{L^2(Q^\pm_1)} \lesssim \frac{1}{\delta^\ell} \norm{g}_{L^q(Q^\pm_1)}
}
since $s-n/2 > 0$. However,
\eqn{fep19}{
\norm{g}_{L^q(Q^\pm_1)}
 \lesssim \norm{g}_{H^{s-\ell}(Q^\pm_1)}
}
by Sobolev's inequality, and therefore, we have that
\leqn{fep20}{
\norm{g_\delta}_{L^2(Q^\pm_1)} \lesssim \frac{1}{\delta^\ell} \norm{g}_{H^{s-\ell}(Q^\pm_1)}.
}

In either case, the inequalities \eqref{fep14} and\eqref{fep20} when combined with \eqref{fep12} and
interpolation show that
\eqn{fep21a}{
\norm{g_\delta}_{H^{s-\ell}(Q^\pm_1)} \lesssim \frac{1}{\delta^\ell}\norm{g}_{H^{s-\ell}(Q^\pm_1)},
}
and so we see that
\leqn{fep21}{
\norm{g_\delta}_{H^{0,s-\ell}(Q_1)} \lesssim \frac{1}{\delta^\ell}\norm{g}_{H^{0,s-\ell}(Q_1)}.
}

Continuing on, a simple calculation yields
\leqn{fep22}{
\norm{D^{m_{s-\ell}} h_\delta}_{L^2(Q_1)} \lesssim \delta^{m_{s-\ell}-n/2} \norm{D^{m_{s-\ell}} h}_{L^2(Q_1)}. \leq  \delta^{
s-n/2}\frac{1}{\delta^{s-m_{s-\ell}}}\norm{D^{m_{s-\ell}} h}_{L^2(Q_1)}.
}
Four cases $s-\ell = 0$, $s-\ell =1$, $s-\ell =2$ and $s-\ell \geq 3$ now follow.

\bigskip

\noindent\textit{Case 1: $s-\ell = 0$}

\smallskip
If $s=\ell$, then  $m_{s-\ell}=0$ and the estimate
\leqn{fep23}{
\norm{h_\delta}_{L^2(Q_1)} \lesssim  \frac{1}{\delta^s} \norm{h}_{\Hc^{0,0}(Q_1)}
}
is a direct consequence of \eqref{fep22} since $s-n/2 > 0$.

\bigskip

\noindent\textit{Case 2: $s-\ell = 1$}

\smallskip
If $\ell = s-1$, then we see from \eqref{fep22} and $m_1=1$  that
\leqn{fep24}{
\norm{Dh_\delta}_{L^2(Q_1)} \lesssim  \delta^{s-n/2}\frac{1}{\delta^{s-1}}\norm{Dh}_{L^2(Q_1)} \lesssim  \frac{1}{\delta^{s-1}} \norm{Dh}_{\Hc^{1,1}(Q_1)}
}
since $s-n/2 > 0$.

\bigskip

\noindent\textit{Case 3: $s-\ell = 2$}

\smallskip
If $\ell = s-2$, then we see from \eqref{fep22} and $m_2=2$  that
\leqn{fep25}{
\norm{D^2 h_\delta}_{L^2(Q_1)} \lesssim  \delta^{s -n/2}\frac{1}{\delta^{s-2}}\norm{D^2h}_{L^2(Q_1)} \lesssim  \frac{1}{\delta^{s-2}} \norm{D^2h}_{\Hc^{2,2}(Q_1)}
}
since $s-n/2 > 0$.

\bigskip

\noindent\textit{Case 4: $s-\ell \geq 3$}

\smallskip
If $s-\ell \geq 3$, then $m_{s-\ell}=2$ and two cases $s-2-\ell < n/2$ and  $s-2-\ell > n/2$ follow.

\bigskip

\noindent \textit{Case 4a:} $s-2-\ell < n/2$
\smallskip

\alin{fep26}{
\norm{D^2 h_\delta}_{L^2(Q_1)}^2 &= \delta^{4-n} \norm{D^2 h}_{L^2(Q_\delta)}^2 \\
&= \delta^{4-n}\left( \int_{Q^-_\delta}|D^2 h| \, d^n x + \int_{Q^+_\delta}|D^2 h|^2 \, d^n x\right).
}
Using H\"{o}lder's inequality, we can write this as
\lalign{fep27}{
\norm{D^2 h_\delta}_{L^2(Q_1)}^2 & \leq \delta^{4-n} \bigl( \norm{1}_{L^p(Q^-_\delta)}\norm{D^2 h}_{L^{2q}(Q^-_\delta)}^2 + \norm{1}_{L^p(Q^+_\delta)}\norm{D^2 h}_{L^{2q}(Q^+_\delta)}^2\bigr) \notag \\
& \leq 2^{(n-1)/p}\delta^{4-n+n/p} \bigl(\norm{D^2 h}_{L^{2q}(Q^-_1)}^2 + \norm{D^2 h}_{L^{2q}(Q^+_1)}^2\bigr) \label{fep27.1}
}
for $1/q + 1/p=1$. But, we notice that
\leqn{fep28}{
\norm{D^2 h}_{L^{\frac{2n}{n-2(s-2-\ell)}}(Q^\pm_1)} \lesssim
\norm{D^2 h}_{H^{s-2-\ell}(Q^\pm_1)} \lesssim \norm{h}_{H^{s-\ell}(Q^\pm_1)}
}
where in obtaining the first inequality we used Sobolev's inequality,
while
\leqn{fep29}{
\norm{D^2 h_\delta}_{L^2(Q_1)}^2 \lesssim \frac{1}{\delta^{2\ell}} \biggl(\norm{D^2 h}_{L^{\frac{2n}{n-2(s-2-\ell)}}(Q^-_1)}^2 +
\norm{D^2 h}_{L^{\frac{2n}{n-2(s-2-\ell)}}(Q^+_1)}^2\biggr)
}
follows from setting
\eqn{fep30}{
q =\frac{n}{n-2(s-2-\ell)} \AND p =  \frac{n}{2(s-2-\ell)}
}
in the inequality \eqref{fep27.1} and recalling that $s>n/2$. Combining the two inequalities \eqref{fep28} and \eqref{fep29}, we arrive at
\leqn{fep31}{
\norm{D^2 h_\delta}_{L^2(Q_1)} \lesssim \frac{1}{\delta^{\ell}}\norm{h}_{H^{0,s-\ell}(Q_1)} \lesssim
\frac{1}{\delta^{\ell}}\norm{h}_{H^{2,s-\ell}(Q_1)}.
}

\bigskip

\noindent \textit{Case 4b:} $s-2-\ell>n/2$

\smallskip

Suppose now that $s-2-\ell>n/2$. Then
\eqn{fep32}{
\norm{D^2 h_\delta}_{L^\infty(Q^\pm_1)} \lesssim \delta^2 \norm{D^2 h}_{L^\infty(Q^\pm_1)} \lesssim  \norm{D^2 h}_{H^{s-2-\ell}(Q^\pm_1)} \lesssim \norm{h}_{H^{s-\ell}(Q^\pm_1)}
}
by Sobolev's inequality. Using H\"{o}lder's inequality, it is not difficult to see that the above inequality implies that
\leqn{fep33}{
\norm{D^2 h_\delta}_{L^2(Q_1)} \lesssim \norm{h}_{H^{0,s-\ell}(Q_1)} \lesssim
\frac{1}{\delta^{\ell}}\norm{h}_{H^{2,s-\ell}(Q_1)}.
}

In either case, \eqref{fep31} and \eqref{fep33} show that
\leqn{fep34}{
\norm{D^2 h_\delta}_{L^2(Q_1)}  \lesssim
\frac{1}{\delta^{\ell}}\norm{h}_{H^{2,s-\ell}(Q_1)}
}
holds.

From the inequalities \eqref{fep21}, \eqref{fep23}, \eqref{fep24}, \eqref{fep25}, \eqref{fep34} and interpolation, we conclude that
\eqn{fep35}{
\norm{h_\delta}_{\Hc^{m_{s-\ell},s-\ell}(Q_1)} \lesssim \frac{1}{\delta^\ell}\norm{h}_{\Hc^{m_{s-\ell},s-\ell}(Q_1)}.
}

\end{proof}

We will also need the following version of Proposition \ref{scalepropA} for the $H^s(Q^+_1)$ spaces.
Since it can be established using similar arguments, we omit the details.

\begin{prop} \label{scalepropB}
Suppose $0<\delta\leq 1$, $s,\ell\in \Zbb_{\geq 0}$, $n\geq 3$, $0\leq \sigma \leq 1$, $0\leq \sigma < s-n/2$, $s-\ell \geq 0$
$f\in H^s(Q^+_1)$, $g\in H^{s-\ell}(Q^+_1)$ and let
\eqn{scalepropB.2}{
f_\delta(x) = \frac{f(\delta x)-f(0)}{\delta^\sigma} \AND g_\delta(x) = g(\delta x).
}
Then
\alin{scalepropB.1}{
\norm{f_\delta}_{H^{s}(Q^+_1)} &\lesssim  \norm{f}_{H^{s}(Q^+_1)} \\
\intertext{and}
\norm{g_\delta}_{H^{s-\ell}(Q^+_1)} &\lesssim \frac{1}{\delta^\ell}\norm{g}_{H^{s-\ell}(Q^+_1)}.
}
\end{prop}
\bibliographystyle{amsplain}
\bibliography{trans}

\end{document}